\documentclass[10pt,a4paper]{amsart}
\usepackage{amssymb,amsmath}

\usepackage{array}
\usepackage{curves}
\usepackage{epic}
\usepackage{eepic}
\usepackage{epsfig}
\usepackage{graphics}
\usepackage{psfrag}

\usepackage[colorlinks=printing,backref]{hyperref}      

\newcommand{\mass}{\Lambda}
\newcommand{\stiff}{\Sigma}

\def\meanint{-\hspace*{-9pt}\int}
\def\dspmeanint{-\hspace*{-11pt}\int}

\def\regmesh{{\rm reg}(\Tau)}
\def\delt{{ {\scriptstyle \Delta}t }}
\def\ptdelt{{ {\scriptscriptstyle \Delta}t }}
\def\Ndelt{{ N_{\!\ptdelt} }}
\def\Vol{{\text{\footnotesize \sl Vol}}}

\def\xKIL{{x_\ptKIL}}
\def\xdKIdL{{x_\ptdKIdL}}

\def\KdotIKplus{{\K_{\!\scriptscriptstyle\odot}\I\K_{\!\scriptscriptstyle\oplus}}}
\def\ptKdotIKplus{{\ptK_{\!\scriptscriptstyle\odot}\ptI\ptK_{\!\scriptscriptstyle\oplus}}}
\def\ptKdotIKplus{{\ptK_{\!\scriptscriptstyle\odot}\ptI\ptK_{\!\scriptscriptstyle\oplus}}}
\def\dKdotKplus{{d_{\ptK_{\!\scriptscriptstyle\odot},\ptK_{\!\scriptscriptstyle\oplus}}}}
\def\nuKdotKplus{{\vec e_{\ptK_{\!\scriptscriptstyle\odot}\!,\hspace*{-1pt}\ptK_{\!\scriptscriptstyle\oplus}}}}
\def\NKdotIKplus{{\vec n_{\ptK_{\!\scriptscriptstyle\odot}\ptI\ptK_{\!\scriptscriptstyle\oplus}}}}

\def\Kdot{{\K_{\!\scriptscriptstyle\odot}}}
\def\Kplus{{\K_{\!\scriptscriptstyle\oplus}}}
\def\ptKdot{{\ptK_{\!\scriptscriptstyle\odot}}}
\def\ptKplus{{\ptK_{\!\scriptscriptstyle\oplus}}}
\def\ddotplus{{d_{\scriptscriptstyle\odot,\oplus}}}
\def\Ndotplus{{\vec n_{\!\scriptscriptstyle\odot,\oplus}}}
\def\nudotplus{{\vec e_{\!\scriptscriptstyle\odot,\oplus}}}

\def\xKdot{{x_{\ptK_{\!\scriptscriptstyle\odot}}}}
\def\xKplus{{x_{\ptK_{\!\scriptscriptstyle\oplus}}}}
\def\xdot{{x_{{\!\scriptscriptstyle\odot}}}}
\def\xplus{{x_{{\!\scriptscriptstyle\oplus}}}}

\def\wdot{{w_{\!\scriptscriptstyle\odot}}}
\def\wplus{{w_{\!\scriptscriptstyle\oplus}}}

\def\xdKi{{x_{\ptK^{\!*}_{\!i}}}}
\def\xdKiun{{x_{\ptK^{\!*}_{\!i\hspace*{-1pt}+\!1}}}}
\def\xidemi{{x^*_{\hspace*{-1pt}i\hspace*{-1pt},\hspace*{-1pt}i\hspace*{-1pt}{\scriptscriptstyle +}\!{\scriptscriptstyle 1}}}}

\def\dKidKiun{{d_{\ptK^{\!*}_{\!i},\ptK^{\!*}_{\!i\hspace*{-1pt}+\!1}} }}
\def\nudKidKiun{{\vec e_{\ptK^{\!*}_{\!i},\ptK^{\!*}_{\!i\hspace*{-1pt}+\!1}} }}
\def\diiun{{d^{*}_{\!i,i\hspace*{-1pt}+\!1} }}
\def\nuiiun{{\vec e^{\,*}_{\!i,i\hspace*{-1pt}+\!1} }}

\def\ptdKiIdKiun{{{\ptK^{\!*}_{\!i}\ptI\ptK^{\!*}_{\!i\hspace*{-1pt}+\!1}}}}

\def\epsSK{{\epsilon_\ptS^\ptK}}
\def\epsSdK{{\epsilon_\ptS^\ptdK}}
\def\ptepsSK{{ \textstyle\epsilon_{\scriptscriptstyle S}^{\scriptscriptstyle K} }}
\def\ptepsSdK{{ \textstyle\epsilon_{\scriptscriptstyle S}^{ \scriptscriptstyle K^{\!*}} }}

\def\xdotplus{{x^*_{\!\scriptscriptstyle\odot\!,\!\oplus}}}

\newcommand{\R}{{\mathbb R}}
\newcommand{\N}{{\mathbb N}}

\newcommand\bel{\begin{equation}\label}
\newcommand\ee{\end{equation}}
\newcommand\ba{\begin{array}}
\newcommand\bal{\begin{array}{l}}
\newcommand\ea{\end{array}}
\def\dsp{\displaystyle}

\newtheorem{theo}{Theorem}[section]
\newtheorem{prop}[theo]{Proposition}
\newtheorem{lem}[theo]{Lemma}
\newtheorem{rem}[theo]{Remark}
\newtheorem{defi}[theo]{Definition}

\newcommand{\llangle}{\langle\!\langle}
\newcommand{\rrangle}{\rangle\!\rangle}

\def\ptl{{\partial}}

\newcommand{\grad}{{\nabla}}
\def\dist{{\rm dist\,}}
\def\diam{{\rm diam\,}}

\newcommand\1{\hbox{\hbox{1}\kern-.3em I}}

\def\dsp{\displaystyle}
\def\div{{\rm div\,}}
\def\Div{{\rm div\, }}
\def\grad{{\,\nabla }}

\def\dist{{\rm dist\,}}

\def\ph{{\varphi}}

\def\Om{{\Omega}}
\def\ptl{{\partial}}

\def\ba{\begin{array}}
\def\bal{\begin{array}{l}}
\def\ea{\end{array}}
\def\be{\begin{equation}}
\def\bel{\begin{equation}\label}
\def\ee{\end{equation}}
\def\iint{\int\!\!\!\!\int}

\def\K{{\scriptstyle K}}
\def\L{{\scriptstyle L}}
\def\dK{{\scriptstyle K^*}}
\def\dL{{\scriptstyle L^*}}
\def\diezK{{\scriptstyle K^\diamond}}

\def\ptK{{\scriptscriptstyle K}}
\def\ptL{{\scriptscriptstyle L}}
\def\ptdK{{\scriptscriptstyle K^*}}
\def\ptdL{{\scriptscriptstyle L^*}}
\def\ptdiezK{{\scriptscriptstyle K^{\diamond}}}

\def\ptD{{\scriptstyle D}}

\def\Tau{{\boldsymbol{\mathfrak T}}}
\def\Drond{{\boldsymbol{\mathfrak D}}}
\def\Frond{{\mathcal M}}
\def\Grond{{\mathcal G}}

\def\Nrond{{\scriptstyle \mathcal N}}
\def\dNrond{{{\scriptstyle \mathcal N}^*}}
\def\Srond{{\boldsymbol{\mathfrak S}}}
\def\Mrond{{\boldsymbol{\mathfrak M}^{\hspace*{-1pt}o}}}

\def\MrondOm{{\boldsymbol{\mathfrak M}^{\hspace*{-1pt}o}_{\hspace*{-1pt}\Om}}}
\def\MrondGN{{\boldsymbol{\mathfrak M}^{\hspace*{-1pt}o}_{\hspace*{-1pt}\Gamma_{\hspace*{-1pt}N}}}}
\def\ptMrondGN{{\scriptscriptstyle \boldsymbol{\mathfrak M}^{\hspace*{-0.7pt}o}_{\hspace*{-0.7pt}\Gamma_{\hspace*{-0.7pt}N}}}}
\def\ptMrondOm{{\scriptscriptstyle \boldsymbol{\mathfrak M}^{\hspace*{-0.7pt}o}_{\hspace*{-0.7pt}\Om}}}

\def\dMrond{{\boldsymbol{\mathfrak M}^{\hspace*{-1pt}*}}}
\def\diezMrond{{\boldsymbol{\mathfrak M}^{\hspace*{-1pt}\scriptstyle \diamond}}}
\def\ptTau{{\scriptscriptstyle \boldsymbol{\mathfrak T}}}
\def\ptDrond{{\scriptscriptstyle \boldsymbol{\mathfrak D}}}

\def\ptMrond{{\scriptscriptstyle \boldsymbol{\mathfrak M}^{\hspace*{-0.7pt}o}}}
\def\ptdMrond{{\scriptscriptstyle \boldsymbol{\mathfrak M}^{\hspace*{-0.7pt}*}}}
\def\ptdiezMrond{{\scriptscriptstyle
\boldsymbol{\mathfrak M}^{\hspace*{-0.7pt}\scriptscriptstyle \diamond}}}

\def\wMrond{{w^\ptMrond}}

\def\wdMrond{{w^\ptdMrond}}

\def\ptbarTau{{\overline{\ptTau}}}

\def\xK{{x_\ptK}}
\def\xL{{x_\ptL}}
\def\xdK{{x_\ptdK}}
\def\xdL{{x_\ptdL}}

\def\pt{\ensuremath{\partial_t}}

\def\I{{\!|\!}}
\def\ptI{{{\scriptscriptstyle\!|\!}}}

\def\ProjD{{\text{Proj}_\ptD}}
\def\dProjD{{\text{Proj}^*_\ptD}}
\def\size{{\text{size}}}

\def\dKL{{d_{\ptK\!\ptL}}}
\def\ddKdL{{d_{\ptdK\!\ptdL}}}

\def\nuK{{\nu_{\ptK}}}

\def\DM{{\scriptstyle D}}
\def\SDM{{\scriptstyle S}}
\def\ptD{{\scriptscriptstyle D}}
\def\ptS{{\scriptscriptstyle S}}

\def\Vrond{{\scriptstyle\mathcal V}}
\def\dVrond{{{\scriptstyle\mathcal V}^*}}
\def\ptVrond{{\scriptscriptstyle\mathcal V}}
\def\ptdVrond{{{\scriptscriptstyle\mathcal V}^*}}

\def\ptptl{{\scriptscriptstyle \ptl}}
\def\KIL{{\K\I\L}}

\def\dKIdL{{\dK\I\dL}}
\def\ptKIL{{\ptK\ptI\ptL}}
\def\ptdKIdL{{\ptdK\ptI\ptdL}}

\def\mKIL{{m_\ptKIL}}
\def\mdKIdL{{m_\ptdKIdL}}

\def\Delt{{{\scriptstyle \Delta} t}}

\def\diam{{\rm diam\,}}

\def\Bleft{{\biggl[\hspace*{-3pt}\biggl[}}
\def\Bright{{\biggr]\hspace*{-3pt}\biggr]}}
\def\Aleft{{\biggl\{\!\!\!\biggl\{}}
\def\Aright{{\biggr\}\!\!\!\biggr\}}}
\def\Cleft{{\biggl<\hspace*{-4pt}\biggl<}}
\def\Cright{{\biggr>\hspace*{-4pt}\biggr>}}

\def\char{{1\!\mbox{\rm l}}}

\newcommand{\eps}{\varepsilon}

\newcommand{\const}{\mathrm{Const}}

\newcommand{\Grad}{\mathrm{\nabla}}

\newcommand{\bM}{\ensuremath{\mathbf{M}}}
\newcommand{\Iap}{I_{\mathrm{app}}}

\newcommand{\abs}[1]{\left| #1 \right|}

\newcommand{\pOm}{\ensuremath{\partial\Omega}}

{%

\begin{enumerate}}%
{\end{enumerate}
}
{%

\begin{enumerate}}%
{\end{enumerate}
}
{%

\begin{enumerate}}%
{\end{enumerate} }

\def\ph{{\varphi}}
\def\ptDelt{{{\scriptscriptstyle \Delta} \scriptstyle t}}
\def\Delt{{{\scriptstyle \Delta} t}}

\def\T{\scriptstyle T}
\def\Trond{{\mathcal T}}

\begin{document}

\title[DDFV schemes for the bidomain cardiac  model]{Convergence
of discrete duality finite volume schemes for the cardiac bidomain model}

\date{\today}

\author[B. Andreianov]{B. Andreianov}
\address[Boris Andreianov]{\newline
         Laboratoire de Math\'ematiques CNRS UMR 6623\newline
         Universit\'e de Franche-Comt\'e\newline
         16 route de Gray\newline
         25 030 Besanc
         $\hspace*{-3.3pt}_{_{^{\mathsf ,}}}$on
         Cedex, France} \email[]{boris.andreianov\@@univ-fcomte.fr}

\author[M. Bendahmane]{M. Bendahmane}
\address[Mostafa Bendahmane]{\newline
        Universit\'e Victor S\'egalen - Bordeaux 2 \newline
146 rue L\'eo Saignat, BP 26 \newline
33076 Bordeaux, France}
\email[]{mostafa\underline{~}bendahmane\@@yahoo.fr}

\author[K. H. Karlsen]{K. H. Karlsen}
\address[Kenneth Hvistendahl Karlsen]{\newline
         Centre of Mathematics for Applications \newline
         University of Oslo\newline
         P.O. Box 1053, Blindern\newline
         N--0316 Oslo, Norway}
\email[]{kennethk\@@math.uio.no}
\urladdr{http://folk.uio.no/kennethk}

\author[C. Pierre]{C. Pierre}
\address[Charles Pierre]{\newline
         Laboratoire de Math\'ematiques et Applications\newline
         Universit\'e de Pau et du Pays de l'Adour\newline
         av. de l'Universit\'e BP 1155\newline
         64013 Pau Cedex, France} \email[]{charles.pierre\@@univ-pau.fr}

\thanks{This article was written as part of the the international research program
on Nonlinear Partial Differential Equations at the Centre for
Advanced Study at the Norwegian Academy of Science
and Letters in Oslo during the academic year 2008--09.
The authors thank Dr.~Lars Grasedyck (Max Planck Institute, Leipzig)
for providing and helping us with the H-Matrices Library.}

\keywords{Cardiac electrical activity, bidomain model, finite volume schemes,
convergence, degenerate parabolic PDE}

\begin{abstract}
We prove convergence of discrete duality finite volume (DDFV) schemes
on distorted meshes for a class of simplified macroscopic bidomain models
of the electrical activity in the heart. Both time-implicit and
linearised time-implicit schemes  are treated.
A short description is given of the 3D DDFV
meshes and of some of the associated discrete calculus tools.
Several numerical tests are presented.
\end{abstract}

\maketitle

\tableofcontents

\section{Introduction}
\label{sec:intro}

We consider the heart of a living organism that occupies a
fixed domain $\Om$, which is assumed to be a bounded
open subset of $\mathbb{R}^3$ with Lipschitz boundary
$\pOm$.  A prototype model for the cardiac electrical
activity is the following
nonlinear reaction-diffusion system
\begin{equation}
    \label{S1-p}
   \left\{ \begin{split}
        & \pt v-\Div \left(\bM_i(x)\Grad u_i\right)
        +h[v]= \Iap, \qquad (t,x)\in Q,\\
        & \pt v + \Div \left(\bM_e(x)\Grad u_e\right)
        +h[v]=\Iap,
        \qquad (t,x)\in Q,
    \end{split}
    \right.
\end{equation}
where $Q$ denotes the time-space cylinder $(0,T)\times \Om$.

This model, called the bidomain model, was first proposed in the
late 1970s by Tung \cite{Tung1978} and is now the generally accepted
model of electrical behaviour of cardiac tissue 
(see Henriquez \cite{Henriquez}, Keener and Sneyd \cite{Keener}). 
The functions $u_i=u_i(t,x)$ and $u_e=u_e(t,x)$ represent the
\textit{intracellular} and \textit{extracellular} electrical
potentials, respectively, at time $t\in (0,T)$ and location $x \in \Om$.
The difference $v=u_i-u_e$ is known as the \textit{transmembrane}
potential. The  conductivity properties of the two media
are modelled by anisotropic, heterogeneous tensors
$\bM_i(x)$ and $\bM_e(x)$. The surface capacitance of the
membrane is usually represented by a positive constant 
$c_m$; upon rescaling, we can assume $c_m=1$.
The stimulation currents applied to the intra- and extracellular spaces are
represented by an $L^2(Q)$ function $\Iap=\Iap(t,x)$.
Finally, the \textit{transmembrane ionic current} $h[v]$ is
computed from the potential $v$. The system is closed by
choosing a relation that links $h[v]$ to $v$ and
specifying appropriate initial-boundary conditions.
We stress that realistic models include a system of ODEs for computing the ionic
current as a function of the transmembrane potential
and a series of additional ``gating variables'' aiming to
model the ionic transfer across the cell
membrane (see, e.g., \cite{LuoRudy,Hodgkin-Huxley,Noble,Keener}).
This makes the relation $h=h[v]$ non-local in time.

Herein we focus on the issue of discretisation
in space of the bidomain model. The presence of the
ODEs, some of them being quite stiff,
greatly complicates the issue of discretisation in time.
It also results in a huge gap between theoretical convergence results and
the practical computation of a reliable solution. We surmise that the precise
form of the relations that link $h[v]$ to $v$ is not essential for
the validation of the space discretisation techniques. 
Therefore, as in \cite{BK1,BK2}, we study \eqref{S1-p}
under the greatly simplifying assumption
that the ionic current is represented locally, in time
and space, by a nonlinear function $h(v)$.
However, such a simplification allows to mimic, to a certain extent,
the depolarisation sequence in the cardiac tissue, taking
the ionic current term $h[v]$ to be a cubic polynomial
(bistable equation); this choice models the fast 
inward sodium current that initiates
depolarisation (cf., e.g., \cite{ColliGuerriTentoni90}).

In the context of electro-cardiology the relevant boundary
condition would be a Neumann condition
for the fluxes associated with the intra- and
extracellular electrical potentials:
\begin{equation*}
    \bM_{i,e}(x)\Grad
        u_{i,e}\cdot n=s_{i,e} \quad \text{ on } (0,T)\times\ptl\Om.
\end{equation*}
 It serves to couple the heart electrical activity with the much weaker electrical
phenomena taking place in the torso. The simplest case is the
one of the isolated heart, namely $s_{i,e}=0$. For the mathematical
study we are heading to, we consider rather general
mixed Dirichlet-Neumann boundary conditions of the form
\begin{equation}
    \label{S2}
    u_{i,e}=g_{i,e} \quad  \text{ on } (0,T)\times\Gamma_D,
        \qquad  \bM_{i,e}(x)\Grad
        u_{i,e}\cdot n=s_{i,e} \quad \text{ on } (0,T)\times\Gamma_N,
\end{equation}
where $\partial \Om$ is partitioned into sufficiently regular parts
$\Gamma_N$ and $\Gamma_D$, and $n$ denotes the
$\mathcal H^2$-a.e. defined exterior unit normal
vector to the Neumann part $\Gamma_N$ of the boundary $\ptl\Om$.
To keep the analysis simple, let us assume 
that $s_{i,e}\in L^2((0,T)\times\Gamma_N)$; for $g_{i,e}$, we assume 
$g_{i,e}\in L^2(0,T; H^{1/2}(\Gamma_D))$ (in fact, we consider 
$g_{i,e}$ extended to $L^2(0,T;H^1(\Om))$ functions).

Regarding the initial data, we prescribe only the 
transmembrane potential:
\begin{equation}
    \label{S3}
    v(0,x)=v_0(x), \quad x\in \Om.
\end{equation}

Clearly, \eqref{S1-p} and \eqref{S3} are invariant under the
simultaneous change of $u_{i,e}$ into $u_{i,e}+k$, $k\in \R$.
In the case $\ptl\Om=\Gamma_N$, $\Gamma_D=\O$,
also  \eqref{S2} is invariant under this change; therefore,
for the sake of being definite,
we normalise $u_{e}$ by assuming
\begin{equation}\label{eq:Ui-normalized}
\text{whenever $\Gamma_D=\O$},
\qquad \int_\Om u_e(t,\cdot)= 0
\quad\text{for a.e. $t\in (0,T)$}.
\end{equation}
It is easy to see that the existence
of solutions to \eqref{S1-p},\eqref{S3} requires
the compatibility condition
\begin{equation}\label{eq:Neumann-compatibility}
\text{whenever $\Gamma_D=\O$},
\qquad \int_{\ptl \Om} s_i(t,\cdot)+s_e(t,\cdot)= 0
\quad\text{for a.e. $t\in (0,T)$}.
\end{equation}

Notice that the diffusion operators $\bM_{i,e}(x)\Grad u_{i,e}$ in \eqref{S1-p}
are linear in the gradient $\grad u_{i,e}$, heterogeneous \& anisotropic,
and time-independent; these assumptions
seem to be sufficiently general to capture the phenomena
of the electrical activity in the heart.
More general models with time-dependent and
nonlinear in $\Grad u_{i,e}$  diffusion of the Leray-Lions
type were studied in \cite{BK1}. Here we assume that
$\Bigl(\bM_{i,e}(x)\Bigr)_{x\in\Om}$ is a family of symmetric matrices,
uniformly bounded and positive definite:
\begin{equation*}
\exists \gamma \;\; \text{for a.e. $x\in\Om$},\;\; \forall \xi\in\R^3,
\;\;\; \frac 1\gamma |\xi|^2\leq
(\bM_{i,e}(x)\xi)\cdot\xi \leq \gamma |\xi|^2.
\end{equation*}
In particular, we have $\bM_{i,e}\in L^\infty(\Om)$.

Now let us describe in detail the ionic current function $h=h(v)$.
We assume that $h:\R\to \R$ is a continuous function, and that there 
exist $r\in (2,+\infty)$ and constants $\alpha,L,l>0$ such that
\begin{equation}
    \label{eq:h-conseq}
    \frac{1}{\alpha}\abs{v}^{r}\le \abs{h(v)v}
    \le \alpha\left(\abs{v}^{r}+1\right),
\end{equation}
\begin{equation}
    \label{monotone-bis}
    \tilde h: z\mapsto h(z)+Lz+l \quad \text{is strictly increasing
    on $\R$, with $\lim_{z\to 0}\tilde h(z)/z=0$}.
\end{equation}
For the later use, we set
$$b:z\mapsto \tilde h(z)/z, \;\; b(0)=0.$$
It is rather natural, although not necessary, to require in addition that
\begin{equation}\label{eq:strict-monot-of-h}
\forall \,z,s\in \R \quad  (\tilde h(z)-\tilde h(s))(z-s) \geq 
\frac {1}{C}(1+|z|+|s|)^{r-2} |z-s|^2.
\end{equation}
According to \cite{Colli1,Colli4}, the most appropriate value 
is $r=4$, which means that the non-linearity 
$h$ is of cubic growth at infinity. Assumptions 
\eqref{eq:h-conseq},\eqref{monotone-bis}
are automatically satisfied by any cubic polynomial $h$
with positive leading coefficient.

A number of works have been devoted to the theoretical and
numerical study of the above bidomain model.
Colli Franzone and Savar\'e \cite{Colli4} prove
the existence of weak solutions for the model
with an ionic current term driven by a single
ODE, by applying the theory of evolution
variational inequalities in Hilbert spaces.
Sanfelici \cite{San:2002} considered the same approach 
to prove the convergence of Galerkin approximations for the bidomain model. 
Veneroni in \cite{veneroni-06} extended this technique
to prove existence and uniqueness results for
more sophisticated ionic models.
Bourgault, Coudi\`{e}re and Pierre \cite{coudiere1} prove
existence and uniqueness results for the bidomain
equations, including the FitzHugh-Nagumo and Aliev-Panfilov models,
by applying a semigroup approach and also by using the
Faedo--Galerkin method and compactness techniques.
Recently, Bendahmane and Karlsen \cite{BK1} proved
the existence and uniqueness for a nonlinear
version of the simplified bidomain equations \eqref{S1-p} by using
a uniformly parabolic regularisation of the
system and the Faedo--Galerkin method.

Regarding finite volume (FV) schemes for cardiac problems, a first
approach is given in Harrild and Henriquez \cite{henriquez_fv}.
Coudi\`{e}re and Pierre \cite{cp:05} prove
convergence of an implicit FV approximation to the monodomain
equations.  We mention also the work of Coudi\`{e}re,
Pierre and Turpault \cite{CoudierePierreTurpault}
on the well-posedness and testing of the DDFV method
for the bidomain model.  Bendahmane and
Karlsen \cite{BK2} analyse a FV method for
the bidomain model with  Dirichlet boundary conditions, supplying
various existence, uniqueness and convergence results.
Finally, Bendahmane, B\"urger and Ruiz \cite{Benetal} analyse
a parabolic-elliptic system with Neumann boundary conditions, adapting
the approach in  \cite{BK2}; they also provide numerical experiments.

In this paper, as in \cite{BK2}, we use a finite volume approach for the space
discretisation of \eqref{S1-p} and the backward Euler  scheme
in time. Due to a different choice of the finite volume discretisation,
we drop the restrictions on the mesh and on the
isotropic and homogeneous structure of the tensors $\bM_{i,e}$
imposed in \cite{BK2}. We also consider general boundary conditions \eqref{S2}.
The space discretisation strategy we use is essentially the one
described and  implemented by Pierre \cite{Pierre} and
Coudi\`ere et al.~\cite{CoudierePierreTurpault,Pierre-FVCA5}.
More precisely, we utilise different types of DDFV discretisations
of the 3D diffusion operator; in addition to the
scheme of \cite{Pierre,Pierre-FVCA5}, we examine the schemes
described in \cite{ABK-FVCA5,Herm4,AndrBend} (see
also \cite{ABK}) and \cite{CoudiereHubert}.  It should be noticed 
that $2D$ bidomain simulations on slices of the $3D$ heart
are also of interest. The standard 2D DDFV construction can be applied to problem
\eqref{S1-p},\eqref{S2},\eqref{S3} on 2D polygonal domains; the 3D convergence results
readily extend to the 2D case.

The DDFV approximations were designed
specifically for anisotropic and/or nonlinear diffusion problems, and
they work on rather general (eventually, distorted, non-conformal and locally
refined) meshes. We refer to Hermeline \cite{Herm0,Herm,Herm2,Herm3,Herm4},
Domelevo and Omn\`es \cite{DomOmnes}, Delcourte, Domelevo
and Omn\`es \cite{DomDelcOmnes}, Andreianov, Boyer
and Hubert \cite{ABH``double''}, and Herbin and Hubert \cite{HerbinHubert-bench}
for background information on DDFV methods. Most of these works
treat 2D linear anisotropic, heterogeneous diffusion
problems, while the case of discontinuous diffusion operators 
have been treated by Boyer and Hubert in \cite{BH}.
Hermeline \cite{Herm3,Herm4} treats the analogous
3D problems,  \cite{DomDelcOmnes} treats the Stokes problem,
and the work \cite{ABH``double''} is devoted to
the nonlinear Leray-Lions framework.

A number of numerical simulations of the full 
bidomain system (the PDE \eqref{S1-p} 
for $u_{i,e}$ plus ODEs for $h[v]$) coupled with the torso can be found
in \cite{lines-02,SC.4.Lines.2003.b,CoudierePierreTurpault,Sundnes.2005.1,Sundnes.2005.2}.

Our study can be considered as a theoretical and numerical
validation of the DDFV discretisation strategy for the
bidomain model. For both a fully time-implicit scheme and a linearised
time-implicit scheme, we prove convergence  of different DDFV
discretisations  to the unique solution of  the bidomain model \eqref{S1-p}.
Then numerical experiments are reported to document some of the 
features of the DDFV space discretisations. A rescaled version of
model \eqref{S1-p}, together with a cubic shape for
$v\mapsto h[v]$, is used to simulate the propagation of 
excitation potential waves in an anisotropic medium. 
In our tests, we combine 2D and 3D DDFV 
schemes for the diffusion terms with fully explicit 
discretisation of the ionic current term; thus
numerical experiments validate this scheme, although we
were not able to justify its convergence theoretically.
Convergence of the numerical solutions towards the continuous
one is measured in three different ways: the first two ones are 
aimed at physiological applications (convergence 
for the activation time and for the propagation velocity), whereas 
the third one corresponds to the norm used in 
Theorem~\ref{th:convergence}. Implementation is detailed.
Due to a large number of unknowns and a relatively large
stencil of the 3D DDFV schemes, a careful preconditioning
is needed for the bidomain system matrix
that has to be inverted at each time step.
The preconditioning strategy we adopted here is 
developed in \cite{pierre-2010-precond}: it provides an almost 
linear complexity with respect to the matrix size for the system matrix inversion. 
The preconditioning combines the idea of hierarchical matrices 
decomposition \cite{h-mat-1,h-mat-2} with heuristics 
referred to as the \textit{monodomain} approximation \cite{colli-taccardi-05}.

The remaining part of this paper is organised as follows:
In Section \ref{sec:Theory-WellPosedness} we give the
definition of a weak solution to \eqref{S1-p},\eqref{S2},\eqref{S3}. Moreover,
we recast the problem into a variational form, from which
we deduce an existence and uniqueness result. In Section \ref{sec:DDFV}
we describe one of the 3D DDFV schemes, while
in Section \ref{sec:Scheme-and-Results} we formulate two
``backward Euler in time" \& ``DDFV in space" finite volume schemes,
and state the main convergence results. The proofs of
these results are postponed to Section \ref{sec:ConvergenceProof}; their
basis being Section \ref{sec:DiscreteCalculus}, where we
recall some mathematical tools for studying DDFV
schemes. Finally, Section \ref{sec:Numerics} is devoted
to numerical examples.

\section{Solution framework and well-posedness}\label{sec:Theory-WellPosedness}
We introduce the space
$$
\text{$V=$ closure of the set
$\Bigl\{v\in C^\infty(\R^3), \;v|_{\Gamma_{\!D}}=0 \Bigr\} $ in
the $H^1(\Om)$ norm.}
$$
In the case $\Gamma_D=\O$, we also use the
quotient space $V_0:=V/\{v\in V,\, v\equiv \const \}$. The
dual of $V $ is denoted by $V'$, with a corresponding duality pairing
$\langle \cdot,\cdot\rangle$.

We assume that the Dirichlet
data $g_{i,e}$  in \eqref{S2} are sufficiently regular, so that
$$
\text{$g_{i,e}$ are the traces on $(0,T)\times \Gamma_D$ of a
couple of $L^2(0,T;H^1(\Om))$ functions}$$
(we keep the same notation for the functions $g_{i,e}$ and their traces).
For the sake of simplicity, we assume that
$$\text{the Neumann data $s_{i,e}$
belong to $L^2((0,T)\times \Gamma_N)$.}$$
Finally, we require that
$$
\text{the initial function $v_0$ belongs to $L^2(\Om)$}.
$$

\begin{defi}\label{def:weak-sol}
A weak solution to
Problem \eqref{S1-p},\eqref{S2},\eqref{S3} is
a triple of functions
\begin{equation}\label{eq:Spaces-for-u-and-v}
\text{$(u_i,u_e,v):\Omega\to \R^3$ s.t.~$u_{i,e}-g_{i,e}\in
L^2(0,T;V)$, $v=u_i-u_e$, $v\in L^r(Q)$},
\end{equation}
 and such that \eqref{S1-p},\eqref{S2},\eqref{S3} are satisfied in $\mathcal
D'([0,T)\times (\Omega\cup\Gamma_N))$. In the case $\Gamma_D=\O$,
we normalise $u_e$ by requiring \eqref{eq:Ui-normalized}.
\end{defi}

\begin{rem}\label{rem:AL-def}\rm
It is not difficult to show that Definition~\ref{def:weak-sol} is
equivalent to a ``variational'' formulation of
Problem \eqref{S1-p},\eqref{S2},\eqref{S3}, in the spirit
of Alt and Luckhaus \cite{AL}.  Indeed, a triple $(u_i,u_e,v)$
satisfying \eqref{eq:Spaces-for-u-and-v} is a weak solution
of Problem \eqref{S1-p},\eqref{S2},\eqref{S3} if and only if
\eqref{S1-p},\eqref{S2},\eqref{S3}
are satisfied in the space $L^2(0,T;V')+L^{r'}(Q)$. This means precisely that the
distributional derivative $\pt v$ can be identified with an
element of $L^2(0,T;V')+L^{r'}(Q)$, and with this
identification there holds
\begin{equation}
 \label{eq:weak-sol-1}
    \begin{split}
        & \int_0^T \langle\pt v,\ph\rangle
        + \iint_Q \Bigl(\bM_i(x)\, \Grad u_i\cdot\Grad \ph
        +h(v)\ph\Bigr)-\int_0^T\!\!\!\int_{\Gamma_{\!\!N}} s_i\,\ph
        = \iint_Q \Iap\ph,\\
        & \int_0^T  \langle\pt v,\ph\rangle
        - \iint_Q \Bigl(\bM_e(x)\,\Grad u_e\cdot\Grad \ph
        +h(v)\ph\Bigr)-\int_0^T\!\!\!\int_{\Gamma_{\!\!N}} s_e\,\ph
        = \iint_Q \Iap\ph,
    \end{split}
\end{equation}
for all $\ph\in L^2(0,T;V)\cap L^r(Q)$, and
\begin{equation*}
    \begin{split}
        & \int_0^T  \langle\pt v,\ph\rangle
        = -\iint_Q  v\,\pt\ph- \int_\Om v_0(\cdot)\,\ph(0,\cdot)
    \end{split}
\end{equation*}
for all $\ph\in L^2(0,T,V)$ such that $\pt \ph \in L^\infty(Q)$
and $\ph(T,\cdot)=0$.
\end{rem}

We have the following chain rule:
\begin{lem}\label{lem:chain-rule}
Assume that $ v\in L^2(0,T;V)\cap L^r(Q)$ and $\ptl_t v\in L^2(0,T;V')+L^{r'}(Q)$. Then
$$
\int_0^T \langle \pt v\,,\,\zeta(t)v\rangle \;=\;
-\iint_Q \frac{v^2}{2}\,\pt \zeta\; - \int_\Om \frac{v_0^2}{2}\,\zeta(0),
\quad \forall \zeta\in \mathcal D([0,T)).
$$
\end{lem}
This type of result is well known; for example, it can be proved along
the lines of  Alt and Luckhaus \cite{AL} and Otto \cite{Otto} (see
also \cite{LionsMagenes} and \cite[Theor\`eme II.5.11]{BoyerFabrie}).

The following lemma is a technical tool  adapted to the weak
formulation of Definition~\ref{def:weak-sol}.

\begin{lem}\label{lem:regularisation}
Let $\Omega$ be a Lipschitz domain.  There exists a family
of linear operators $(\mathcal R_\eps)_{\eps>0}$ from
$L^2(0,T,V)$ into $\mathcal D(\R\times\R^d)$ such that \\
- for all $z\in L^2(0,T,V)$, $\mathcal R_\eps(z)$ converges to $z$
 in $L^2(0,T,V)$;\\
- for all $z\in L^r(Q)\cap L^2(0,T,V)$,  $\mathcal R_\eps(z)$
converges to $z$  in $L^r(Q)$.
\end{lem}

Let us stress that the linearity of $\mathcal R_\eps(\cdot)$
is essential for the application of this lemma.
It is used to regularise $u_{i,e}$, so that one
can take $\mathcal R_\eps(u_{i,e})$ as test functions in
\eqref{eq:weak-sol-1}; for example, a priori estimates
for weak solutions and uniform bounds on their
Galerkin approximations will be obtained in this way.
In addition, a straightforward
application of the lemma is the following
uniqueness result:

\begin{theo}\label{theo:L^2-contraction}
Assume \eqref{eq:h-conseq} and \eqref{monotone-bis}.
Then there exists a unique weak solution $(u_i,u_e,v)$ to
Problem \eqref{S1-p},\eqref{S2},\eqref{S3}. Moreover, if
$(\hat u_i,\hat u_e,\hat v)$ is another weak
solution of Problem \eqref{S1-p},\eqref{S2},\eqref{S3}
corresponding to the initial
function $\hat v_0\in L^2(\Omega)$, then
$$
\text{ for a.e. $t\in(0,T)$,
$\|v(t)-\hat v(t)\|_{L^2(\Om)} \leq
e^{\sqrt{2Lt}}\|v_0-\hat v_0\|_{L^2(\Om)}$}.
$$
In addition, if \eqref{eq:strict-monot-of-h} holds
then $v$ depends continuously in
$L^r(Q)$ on $v_0$ in $L^2(\Om)$.
\end{theo}

Continuous dependence of the solution
on $\Iap$, $s_{i,e}$, $g_{i,e}$ can be
shown with the same technique, using
in addition the Cauchy-Schwarz
inequality on $(0,T)\times\Gamma_N$
and the trace inequalities for $H^1$ functions.

\begin{proof}
Let $\zeta\in \mathcal D([0,T))$, $\zeta\geq 0$.
We take $\zeta(t)\mathcal R_\eps(u_i-\hat u_i)(t,x)$ as
test function in the first equation of
\eqref{eq:weak-sol-1}, and $\zeta(t)\mathcal
R_\eps(u_e-\hat u_e)(t,x)$ in the second equation of
\eqref{eq:weak-sol-1}. We subtract the resulting equations and
apply the chain rule of Lemma~\ref{lem:chain-rule}; using
the linearity of $\mathcal R_\eps(\cdot)$ and the other
properties listed in Lemma~\ref{lem:regularisation},
and subsequently sending $\eps\to 0$, we finally arrive at
\begin{align*}
&\iint_Q -\frac{(v-\hat v)^2}{2}\,\pt \zeta
-\int_\Om \frac{(v_0-\hat v_0)^2}{2}\, \zeta(0)
+ \iint_Q \Bigl(\,h(v)-h(\hat v)\,\Bigr)\,(v-\hat v)\zeta
\\ & \qquad\quad +\iint_Q \biggl( \bM_i(x)
\Bigl(\Grad u_i-\Grad\hat u_i\Bigr)
\cdot(\Grad u_i-\Grad \hat u_i)
\\ & \qquad\qquad \qquad\qquad
+\bM_e(x)\Bigl(\Grad u_e-\Grad\hat u_e\Bigr)
\cdot(\Grad u_e-\Grad \hat u_e) \biggr)\,\zeta=0.
\end{align*}
For a.e.~$t>0$, we let $\zeta$ converge to the characteristic
function of $[0,t]$. Thanks to the monotonicity assumption
\eqref{monotone-bis} on $\tilde h$, we deduce
\begin{align*}
\int_\Om (v-\hat v)^2(t) & \leq \dsp \int_\Om (v-\hat v)^2(t)
+ \int_0^t\int_\Om \Bigl(\,\tilde h(v)-\tilde h(\hat v)\,\Bigr)\,(v-\hat v)
\\ & \leq \int_\Om (v_0-\hat v_0)^2\,
+ 2L \int_0^t \int_\Om (v-\hat v)^2.
\end{align*}
By the Gronwall inequality, the $L^2$ continuous dependence
property stated in the theorem follows.

Next, if \eqref{eq:strict-monot-of-h} holds, from the H\"older inequality
and the evident estimate
$$
|v-\hat v|^r\leq
\left(|v|+|\hat v|\right)^{r-2}|v-\hat v|^2
\quad \text{(recall $r\geq 2$)},
$$
we infer that $\|v-\hat v\|_{L^r(Q)}$ goes to zero as
$\|v_0-\hat v_0\|_{L^2(\Om)}$ tends to zero.

Finally, if $\hat v_0\equiv v_0$, not only do we
have $v\equiv \hat v$, but also $\hat u_{i,e}=u_{i,e}$ because of
the strict positivity of $M_i$ and the
boundary/normalisation condition in $V$.
\end{proof}

It remains to prove the regularisation result.

\begin{proof}[Proof of Lemma~\ref{lem:regularisation}]
For simplicity we consider separately the two basic cases.\\[3pt]
$\bullet$ Pure Dirichlet BC case.\\[3pt]
Extend $z$ by zero for $t\notin (0,T)$.
Take a standard family of mollifiers $(\rho_\eps)_{\eps>0}$ on
$\R^{d+1}$ supported in the ball of radius $\eps$ centred at the origin.
Introduce the set $\Om_\eps:=\{x\in\Om\,|\,
\dist(x,\ptl\Om)<\eps\}$. Take $\theta_\eps$ such that
$\theta_\eps\in \mathcal D(\Om)$, $\theta_\eps\equiv 1$ in
$\Om\setminus\Om_\eps$, $0\leq \theta_\eps\leq 1$, and
$\|\grad \theta_\eps\|_{L^\infty(\Om)}\leq \const/\eps$. Define
$$
\mathcal R_\eps(z)(t,x):=
\Bigl(\rho_\eps(t,x)\Bigr)*
\Bigl(\theta_\eps(x)\, z(t,x)\Bigr).
$$
By construction, $\mathcal R_\eps$
maps $L^1(Q)$ to $C^\infty(\R\times\R^d)$.
From standard properties of mollifiers and the
absolute continuity of the Lebesgue integral, one easily deduces
that if $z\in L^r(Q)$, then $z_\eps:=\mathcal R_\eps(z)$ converges
to $z$ in $L^r(Q)$ as $\eps\to 0$. Next, consider $z\in L^2(0,T;V)$.
We have $z\in L^2(Q)$, and thus $z_\eps\to z$ in $L^2(Q)$ as
above. In particular, $(z_\eps)_{\eps>0}$ is bounded in
$L^2(Q)$. Similarly, $ \rho_\eps*(\theta_\eps \grad z)$ is
bounded in $L^2(Q)$ and converges
to $\grad z\equiv \grad z$ in $L^2(Q)$. 
Since $\grad z_\eps=\rho_\eps*(\theta_\eps\,\grad z)
+ \rho_\eps*(\grad \theta_\eps\, z), $ it remains to show that 
$\rho_\eps*(\grad \theta_\eps\, z)$ converges to
zero in $L^2(Q)$ as $\eps\to 0$. By standard properties of mollifiers, it is 
sufficient to prove that $\grad \theta_\eps\,  z \to 0$ in 
$L^2(\R^d)$ as $\eps\to 0$, which follows from an appropriate 
version of the Poincar\'e inequality.
 
Indeed, in the case $\ptl \Om$ is Lipschitz regular, we can 
fix $\eps_0>0$ and cover $\Om_{\eps_0}$ by a finite
number of balls $({\mathcal O}_i)_{i\in I}$ (eventually rotating the 
coordinate axes in each ball) such that for all $i\in I$, for all $\eps<\eps_0$ the 
set $\Om_\eps\cap {\mathcal O}_i$ is contained in the strip
$\{\Psi_i(x_2,x_3)< x_1 < \Psi_i(x_2,x_3) +C\eps\}$ for some 
Lipschitz continuous function $\Psi_i$ on $\R^2$
and some $C>0$. Hence by the standard Poincar\'e inequality 
in domains of thickness $\eps$, we have $
\|z(t,\cdot)\|_{L^2(\Om_\eps)} \leq C\eps \,\|\grad z(t,\cdot)\|_{L^2(\Om_\eps)}$. Then
$$
\int_0^T\!\!\!\int_{\Omega_\eps}|\grad\theta_\eps\,z|^2\leq
\frac{\const}{\eps^2}\int_0^T\!\!\!\int_{\Omega_\eps}|z|^2\leq
\frac{\const}{\eps^2}C\eps^2\int_0^T\!\!\!\int_{\Omega_\eps}|\grad z|^2,
$$
and the right-hand side converges to zero as $\eps\to 0$, by the
absolute continuity of the Lebesgue integral.\\[3pt]
$\bullet$ Pure Neumann BC case.\\[3pt]
We use a linear extension operator $\mathcal E$ from $V$ into
$H^1(\R^d)$ such that $V\cap L^r(\Om)$ is mapped into
$H^1(\R^d)\cap L^r(\R^d)$. Such an operator is constructed in a
standard way, using a partition of unity, boundary rectification
and reflection (see, e.g., Evans \cite{Evans:1998sm}). We then
define $\mathcal R_\eps$ by the formula
$\mathcal R_\eps(z)=\rho_\eps*(\mathcal E (z))$.\\[3pt]
$\bullet$ The general case: mixed Dirichlet-Neumann BC.\\[3pt]
It suffices to define $\Om_\eps:=\{x\in\Om\,|\,
\dist(x,\Gamma_D)<\eps\}$, introduce $\theta_\eps$ as in the
Dirichlet case, introduce $\mathcal E$ as in the Neumann case,
and take $\mathcal R_\eps(z)=\rho_\eps*(\theta_\eps {\mathcal E}(z))$.
\end{proof}

\begin{rem}\label{rem:spaceE}\rm
We have seen that the following space  appears naturally:
\begin{equation*}
E:=\Bigl\{(u_i,u_e)\;
\Bigl|\;\,u_{i,e}-g_{i,e}\in L^2(0,T,V), \;
v:=u_i-u_e\in L^r(Q) \Bigr\}.
\end{equation*}
Introducing its dual $E'$ and the corresponding duality pairing
$\llangle \cdot,\cdot \rrangle$; we have
 $$
 \llangle (\chi,\xi) \,,\, (\ph,\psi) \rrangle
 =\lim_{\eps\to 0} \langle\chi,\mathcal R_\eps \ph\rangle
 +\langle\xi,\mathcal R_\eps \ph \rangle
$$
 whenever the limit exists.

Now, using Remark~\ref{rem:AL-def} and
Lemma~\ref{lem:regularisation}, it is not difficult to
recast Problem  \eqref{S1-p},\eqref{S2},\eqref{S3} into
the following formal framework:
\begin{align*}
&\text{find $(u_i,u_e)\in E$ such that $(\ptl_t v,-\ptl_t v)\in E'$
and \eqref{S1-p},\eqref{S2},\eqref{S3} hold in $E'$},
\\ & \text{namely, for all $ (\ph,\psi)\in E$,}\\
&\int_0^T \llangle (\ptl_t v,-\ptl_t v)\,,\,(\ph,\psi) \rrangle
\\ & \qquad
+ \iint_Q \biggl( M_i(x,\Grad u_i)\cdot\Grad \ph
- M_e(x,\Grad u_e)\cdot\Grad \psi
+h(v)(\ph-\psi)\biggr)\\ &\qquad\qquad
-\int_0^T\!\!\!\int_{\Gamma_{\!\!N}} \Bigl(s_i\,\ph-s_e\,\psi)
 = \iint_Q \Iap(\ph-\psi),\\
&\text{and for all  $(\ph,\psi)\in E$ such that
$\ptl_t \ph,\ptl_t\psi\in L^\infty(Q)$ and $\ph(T,\cdot)=0=\psi(T,\cdot)$,}\\
&\int_0^T  \llangle (\pt v,-\pt v) \,,\, (\ph,\psi) \rrangle
= -\iint_Q  v\,\pt(\ph-\psi)- \int_\Om
v_0(\cdot)\,(\ph(0,\cdot)-\psi(0,\cdot)).
\end{align*}
\end{rem}
In view of Remarks \ref{rem:AL-def} and \ref{rem:spaceE}, we can
apply some of the techniques used by Alt and Luckhaus \cite{AL} to
deduce an existence result from the uniform boundedness in $E$
of the Galerkin approximations of our problem (cf. \cite{coudiere1}). The
uniform bound in $E$ is obtained using
the chain rule of Lemma~\ref{lem:chain-rule}, the Gronwall inequality
and the assumptions \eqref{eq:h-conseq},\eqref{monotone-bis}
on the ionic current. The arguments of the existence proof
will essentially be reproduced in Section~\ref{sec:ConvergenceProof};
therefore we omit the details here.

In view of the uniqueness and continuous dependence
result of Theorem~\ref{theo:L^2-contraction} and its proof, we
can end this section by stating a well-posedness result.

\begin{theo}\label{theo:existence}
Assume \eqref{eq:Neumann-compatibility}, \eqref{eq:h-conseq}
and \eqref{monotone-bis}. There exists one and only one solution to
Problem \eqref{S1-p},\eqref{S2},\eqref{S3}.
If in addition \eqref{eq:strict-monot-of-h} holds, then the
solution depends continuously in the space
$E$ on the initial datum in $L^2(\Om)$.
\end{theo}

\section{The framework of DDFV schemes}
\label{sec:DDFV}

We make an idealisation of the heart by assuming that it
occupies a polyhedral domain $\Omega$ of $\R^3$. We discretise the
diffusion terms in \eqref{S1-p} using the implicit Euler
scheme in time and the so-called Discrete Duality
Finite Volume (DDFV) schemes in space. The DDFV
schemes were introduced for the discretisation of linear diffusion
problems on $2D$ unstructured, non-orthogonal meshes by Hermeline
\cite{Herm0,Herm} and by Domelevo and Omn\`es \cite{DomOmnes}.
They turned out to be well suited for approximation of anisotropic
and heterogeneous linear or non-linear diffusion problems.

Our application requires a 3D analogue of the 2D DDFV schemes.
Three versions of such 3D DDFV schemes have already been developed;
we shall refer to them as $(A)$, $(B)$ and $(C)$.
We refer to \cite{Pierre,Pierre-FVCA5} for version $(A)$;
version $(B)$ that we describe in Section~\ref{ssec:meshingB} below
was developed in \cite{Herm3,Herm4} and \cite{ABK-FVCA5,ABK,AndrBend};
we refer to \cite{CoudiereHubert} for version $(C)$.
In this paper, we show the convergence of any of
these schemes, using only general properties of DDFV approximations.

\subsection{Generalities}

In the 3D DDFV approach of \cite{Pierre,Pierre-FVCA5} (version $(A)$)
and in the one of \cite{ABK-FVCA5,ABK,AndrBend}, \cite{Herm4,Herm3}
(version $(B)$), the meshes consist
of control volumes of two kinds, the primal and the dual ones.
Version $(C)$ also includes a third mesh. For cases $(B)$ and $(C)$, primal 
volumes and dual volumes form two partitions of $\Omega$, up to a set of
measure zero. In case $(A)$, the primal volumes form a partition 
of $\Om$, and the dual volumes cover $\Om$
twice, up to a set of measure zero.
 Some of the dual and primal volumes are considered as ``Dirichlet boundary'' volumes,
while the others are the ``interior'' volumes (this includes the volumes located near the Neumann part
$\Gamma_N$ of $\ptl\Om$). With each (primal or dual) interior 
control volume we associate unknown values for
$u_i,u_e,v$; Dirichlet boundary conditions are imposed on the boundary volumes. 
The Neumann boundary conditions
will enter the definition of the discrete divergence operator near the 
boundary; it is convenient to take them
into account by introducing additional unknowns associated 
with ``degenerated primal volumes'' that are parts of
the Neumann boundary $\Gamma_N$.

We consider the space $\R^\ptTau$  of discrete functions on $\Om$; 
a discrete function $u^\ptTau\in
\R^\ptTau$ consists of one real value per interior control volume.
On $\R^\ptTau$ an appropriate inner
product $\Bleft\cdot,\cdot\Bright_{\Om}$ is introduced, which 
is a bilinear positive form.

Both primal and dual volumes define a partition of $\Om$ into
diamonds, used to represent discrete gradients
and other discrete fields on $\Om$. The space $(\R^\ptDrond)^3$ of discrete
fields on $\Om$ serves to define the fluxes through the boundaries
of control volumes. A discrete field $\vec \Frond^\ptTau \in
(\R^\ptDrond)^3$ on $\Om$ consists of one $\R^3$ value per
``interior'' diamond. On $(\R^\ptDrond)^3$ an
appropriate inner product $\Aleft\cdot,\cdot,\Aright_{\Om}$ is introduced.

A discrete duality finite volume scheme is determined by the mesh,
the discrete divergence operator $\div_{\!\!
s^\ptTau}^{\!\ptTau}:(\R^\ptDrond)^3\longrightarrow \R^\ptTau$
obtained by the standard finite volume discretisation procedure
(with values $s^\ptTau$ given by the Neumann 
boundary condition on $\Gamma_N$),
and by the associated discrete gradient operator.
More precisely, the discrete gradient operator
$\grad_{\!g^{\!\ptTau}}^\ptTau : \R^\ptTau\longrightarrow
(\R^\ptDrond)^3$
is defined on the space of discrete functions
extended by values $g^\ptTau$
in volumes adjacent to $\Gamma_D$; it is defined in
such a way that the discrete duality property holds:
\begin{equation}\label{eq:discr-duality}
\forall v\in \R^\ptTau, \; \forall \vec
\Frond^\ptTau\in(\R^\ptDrond)^3,\; \quad \Bleft
-\div_{\!\!s^\ptTau}^{\!\ptTau} \vec \Frond^\ptTau,
v^\ptTau\Bright_{\Om} = \Aleft \vec
\Frond^\ptTau, \grad^\ptTau_0
v^\ptTau \Aright_{\Om}
+ \Cleft s^\ptTau,v^{\ptl\ptTau} \Cright_{\Gamma_N}.
\end{equation}
Here $\grad^\ptTau_0$ corresponds to the homogeneous 
Dirichlet boundary condition\footnote{our
notation follows \cite{ABH``double''}; a slightly different viewpoint was
used in \cite{ABK-FVCA5,ABK,AndrBend}, where the homogeneous Dirichlet
boundary data were included into the definition of the space
$\R^\ptTau_0$ of discrete functions defined also on the control
volumes adjacent to $\Gamma_D\equiv \ptl\Om$.} $g^\ptTau=0$ on $\Gamma_D$, and
$s^\ptTau$ denotes the discrete Neumann boundary
datum for $\vec \Frond^\ptTau\cdot n$.
Further, $\Cleft \cdot,\cdot \Cright_{\Gamma_N}$ denotes an
appropriately defined product on the
Neumann part $\Gamma_N$ of the boundary $\ptl\Om$, and
$v^{\ptl\ptTau}$ denotes  the boundary  values on $\Gamma_N$ of
$v^\ptTau$. The precise definitions of these objects are
given below for version $(B)$.

In \cite{Pierre,Pierre-FVCA5} and 
\cite{ABK-FVCA5,ABK,AndrBend},\cite{Herm3,Herm4}, the
definitions of dual volumes and $\Bleft\cdot,\cdot\Bright_{\Om}$ differ; but
both methods can be analysed with the same formalism.
The construction in \cite{CoudiereHubert} only differs by its
use of three meshes based on three kinds of control volumes.
This also changes the definition of  $\Bleft\cdot,\cdot\Bright_{\Om}$. 
The main difference between the three frameworks lies in the 
interpretation of $u^\ptTau\in\R^\ptTau$ in terms of functions.
In each case $u^\ptTau\in\R^\ptTau$  is thought as a piecewise constant function.
The three following lifting between $\R^\ptTau$ and $L^1(\Om)$ are considered: 
\begin{equation}\label{eq:DiscreteSol}
u^\ptTau:=\left\{
\begin{array}{ll}
\dsp \frac{1}{3} v^\ptMrond\!\!+\!\frac{1}{3} v^\ptdMrond & \text{for version $(A)$ described in \cite{Pierre,Pierre-FVCA5} }\\[9pt]
\dsp \frac{1}{3} v^\ptMrond\!\!+\!\frac{2}{3}
v^\ptdMrond &  \text{for version $(B)$ described in
\cite{ABK-FVCA5,ABK,AndrBend}, \cite{Herm3,Herm4}}\\[9pt]
\dsp \frac{1}{3} v^\ptMrond\!\!+\!\frac{1}{3} v^\ptdMrond\!+\!\frac{1}{3}v^{\scriptstyle \ptdiezMrond}&
\text{for version $(C)$ described in \cite{CoudiereHubert}},
\end{array}
\right.
\end{equation}
with $v^\ptMrond$ and $v^\ptdMrond$ representing the discrete solutions on the 
primal and the dual mesh, respectively,
and with $v^{\scriptstyle \ptdiezMrond}$ (in the scheme 
of \cite{CoudiereHubert}) representing the solution
on the third mesh. We have for instance $v^\ptMrond(x)=
\sum_{\ptK\in\ptMrond} v_\ptK\,\char_\ptK(x)$ (with $ \char_\ptK$ the 
characteristic function of $\ptK$), the 
definitions of$v^\ptdMrond$, $v^{\scriptstyle \ptdiezMrond}$ are analogous.

 In all the three cases, appropriate definitions of the spaces
$\R^\ptTau,(\R^\ptDrond)^3$,
the scalar products $\Bleft\cdot,\cdot\Bright_{\Om},
\Aleft\cdot,\cdot\Aright_{\Om},\Cleft \cdot,\cdot \Cright_{\Gamma_N}$, and the
operators $\div^{\!\ptTau},\grad^\ptTau$, lead to the discrete
duality property \eqref{eq:discr-duality}.

\subsection{A description of version $(B)$}\label{ssec:meshingB}

In this subsection we describe the objects and the associated discrete gradient
and divergence operators for version $(B)$ of the scheme.
More details and generalisations can be found in \cite{AndrBend}. \medskip

\subsubsection{Construction of ``double'' meshes}~
\medskip

\noindent $\bullet$ A {\it partition} of $\Om$ is a finite set of
disjoint open  polyhedral subsets of $\Om$ such that $\Om$ is
contained in their union, up to a set of zero three-dimensional
measure.

A ``double'' finite volume mesh of $\Om$ is a triple
$\Tau=\Bigl(\overline{\Mrond},\overline{\dMrond},\Drond \Bigr)$
described in what follows.

\medskip
\noindent $\bullet$ First, let $\MrondOm$  be a partition of $\Om$ into
open polyhedral with triangular or quadrangular faces. We assume them convex.
Assume that $\ptl\Om$ is the disjoint union of polygonal parts
$\Gamma_D$ (for the sake of being definite,
we assume it to be closed) and $\Gamma_N$ (that we therefore assume to be open).
Then we require that each face of the polyhedra in $\MrondOm$
either lies inside $\Om$, or it lies  in $\Gamma_D$, or it lies
in $\Gamma_N$ (up to a set of zero two-dimensional measure).
Each $\K\in\MrondOm$ is called a {\it primal control volume} and
is supplied with an arbitrarily chosen {\it
centre} $\xK$; for simplicity, we assume $x_\ptK\in\K$.

Further, we call $\MrondGN$ (respectively, $\ptl\Mrond$)
the set of all faces of control volumes that are
included in $\Gamma_N$ (resp., in $\Gamma_D$).
These faces are considered as degenerate control 
volumes; those of $\ptl\Mrond$ are called {\it boundary primal volumes}.
For $\K\in\MrondGN$ or $\K\in\ptl\Mrond$, we choose a centre
$\xK\in \K$.

Finally, we denote $\Mrond:=\MrondOm\cup\MrondGN$; $\Mrond$ is 
the set of {\it interior primal volumes}; and we denote by $\overline{\Mrond}$ the union
$\Mrond\cup\ptl\Mrond\equiv \Bigl(\MrondOm\cup\MrondGN\Bigr)\cup \ptl\Mrond$.

\medskip
\noindent $\bullet$ We call {\it neighbours} of $\K$, all control
volumes $\L\in\overline{\Mrond}$  such that $\K$ and $\L$ have a
common face (by convention, a degenerate 
volume $\K\in \MrondGN$ or $\K\in\ptl\Mrond$ has a unique face,
which coincides with the degenerate volume itself).
 The set of all neighbours of $\K$ is denoted by
$\Nrond(\K)$. Note that if $\L\in \Nrond(\K)$, then
$\K\in\Nrond(\L)$; in this case we simply say that $\K$ and $\L$
are (a couple of) neighbours.  If $\K$,$\L$ are  neighbours, we denote by
$\KIL$ the {\it interface (face)} 
$\ptl\K\cap\ptl\L$ between  $\K$ and $\L$.

\medskip
\noindent $\bullet$ We call {\it vertex} (of $\MrondOm$) any
vertex of any control volume $\K\in\MrondOm$.
A generic vertex  of $\MrondOm$ is denoted by
$\xdK$; it will be associated later with a unique {\it dual
control volume} $\dK\in \overline{\dMrond}$. Each face $\KIL$ 
is supplied with a {\it face centre} $x_\ptKIL$ which should lie
in  $\KIL$ (the more general situation is described in \cite{AndrBend}). For
two {\it neighbour vertices} $\xdK$ and $\xdL$ (i.e., vertices
of $\Mrond$ joined by an edge of some interface $ \KIL$
or boundary face), we denote by $\xdKIdL$ the
middle-point of the segment $[\xdK,\xdL]$.

\medskip
\noindent $\bullet$ Now if $\K\in \overline{\Mrond}$ and
$\L\in\Nrond(\K)$, assume $\xdK,\xdL$ are two neighbour vertices
of the interface $\KIL$. We denote by $\T^{\ptK}_{\ptdK,\ptdKIdL}$
the tetrahedra formed by
 the points $\xK, \xdK, \xKIL,\xdKIdL$. A generic tetrahedra
$\T^{\ptK}_{\ptdK,\ptdKIdL}$ is called an {\it element} of the
mesh and is denoted by $\T$ (see Figure~\ref{Fig-PureSubdiamond}); the
set of all elements is denoted by $\Trond$.

\begin{figure}[bth]
\begin{center} 
\includegraphics[width=2.4in]{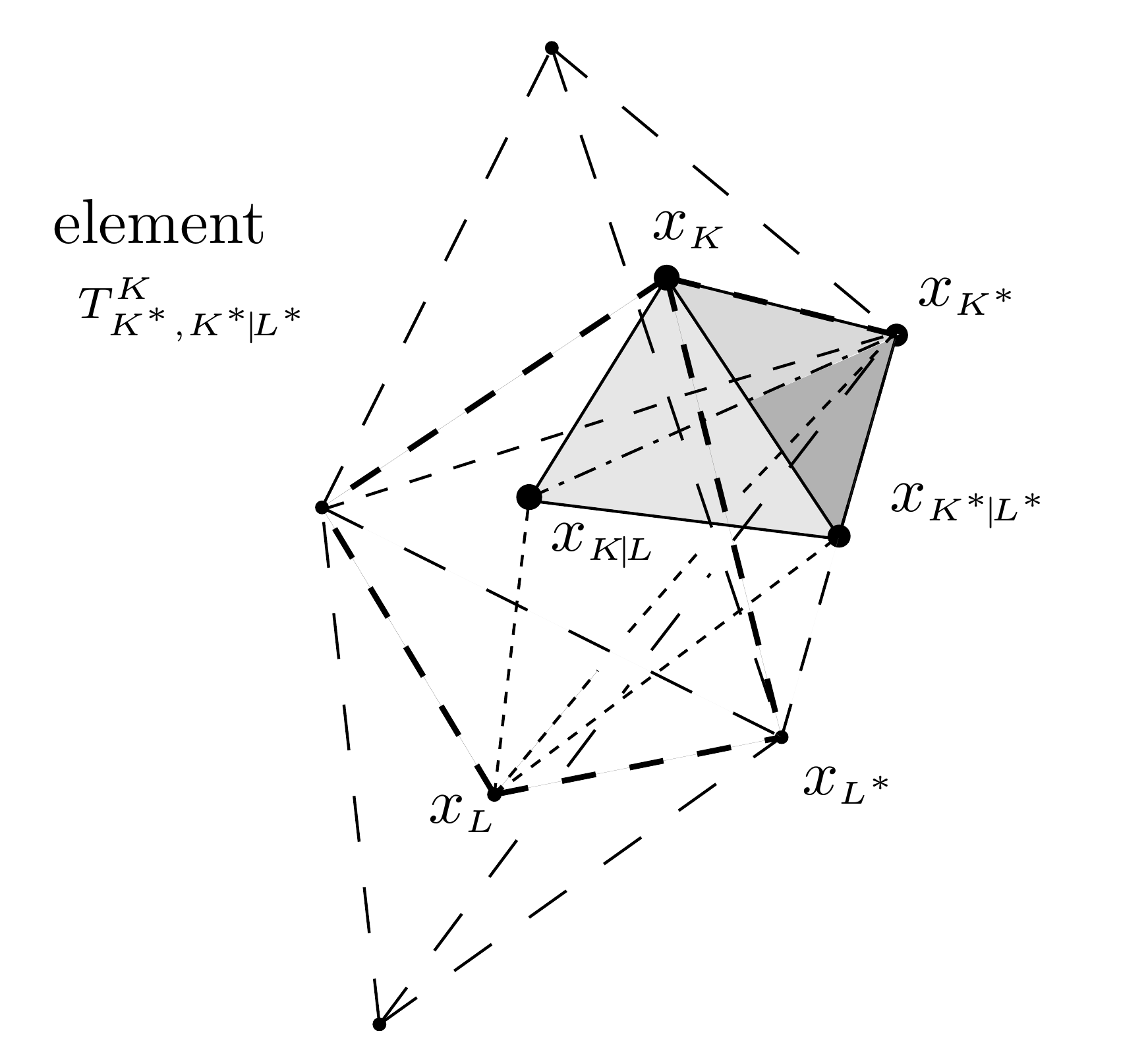}
\end{center}
\caption{ One of the twelve elements in diamond
$\DM^{\ptKIL}$ with
\hspace*{55pt}~triangular base $\KIL$}
\label{Fig-PureSubdiamond}
\end{figure}

\medskip
\noindent $\bullet$
Define the volume $\dK$ associated with  a vertex $\xdK$ of
$\MrondOm$ as the union of all elements ${\T}\in\Trond$ having $\xdK$
for one of its vertices. The collection $\overline{\dMrond}$ of
all such $\dK$ forms another partition of $\Om$.
If $\xdK\in\Om\cup \Gamma_N$, we say that $\dK$ is
an (interior) {\it dual control volume} and
write $\dK\in\dMrond$; and if $\xdK\in {\Gamma_D}$, we say that $\dK$
is a {\it boundary dual control volume} and write
$\dK\in\ptl\dMrond$. Thus $\overline{\dMrond}=\dMrond\cup
\ptl\dMrond$. Any vertex of any dual control volume $\dK\in\overline{\dMrond}$
is called a {\it dual vertex} (of $\overline{\dMrond}$).
Note that by construction, the set of vertices coincides with the set of
dual centres $\xdK$; the set of dual vertices consists of centres
$\xK$, face centres $x_\ptKIL$ and edge centres (middle points)
$\xdKIdL$. Picturing dual volumes in 3D is a hard task; cf. \cite{Pierre} for
version $(A)$ and \cite{CoudiereHubert} for  version $(C)$.

\medskip
\noindent $\bullet$ We denote by $\dNrond(\dK)$ the set of {\it
(dual) neighbours} of a dual control volume $\dK$, and by
$\dKIdL$, the {\it (dual) interface} $\ptl\dK\cap\ptl\dL$ between
dual neighbours $\dK$ and $\dL$.

\medskip
\noindent  $\bullet$ Finally, we introduce the partitions of $\Om$
into diamonds and subdiamonds. If $\K,\L\in\overline{\Mrond}$ are neighbours,
let $H_\ptK$ be the convex hull of $\xK$ and $\KIL$ and $H_\ptL$ be the 
convex hull  of $\xL$ and $\KIL$. Then the union $H_\ptK\cup H_\ptL$ 
is called a {\it diamond} and is
denoted by $\DM^{\ptKIL}$.

If $\K,\L\in \overline{\Mrond}$ are neighbours, and $\xdK,\xdL$
are neighbour vertices of the corresponding interface $\KIL$, then
the union of the four elements
$\T^{\ptK}_{\ptdK,\ptdKIdL}$, $\T^{\ptK}_{\ptdL,\ptdKIdL}$,
$\T^{\ptL}_{\ptdK,\ptdKIdL}$,  and $\T^{\ptL}_{\ptdL,\ptdKIdL}$ is
called {\it subdiamond} and denoted by
$\SDM^\ptKIL_\ptdKIdL$. In this way, each diamond
$\DM^\ptKIL$ gives rise to $l$ subdiamonds (where $l$ is the
number of vertices of $\KIL$); cf. the next item and Fig.~\ref{Fig-3D}. 
Each subdiamond is associated with
a unique interface $\KIL$, and thus with a unique diamond
$\DM^\ptKIL$. We will write $\SDM\subset \DM$ to signify that $\SDM$
is associated with $\DM$.

We denote by  $\Drond$,$\Srond$ the sets of all
diamonds and the set of all subdiamonds, respectively. Generic
elements of $\Drond$,$\Srond$ are denoted by $\DM$,$\SDM$,
respectively. Notice that $\Drond$ is a partition of a subdomain of
$\Omega$ (only a small neighbourhood of
$\Gamma_N$ in $\Om$ is not covered by diamonds).
\begin{figure}[h]
\begin{center} 
\includegraphics[width=5in]{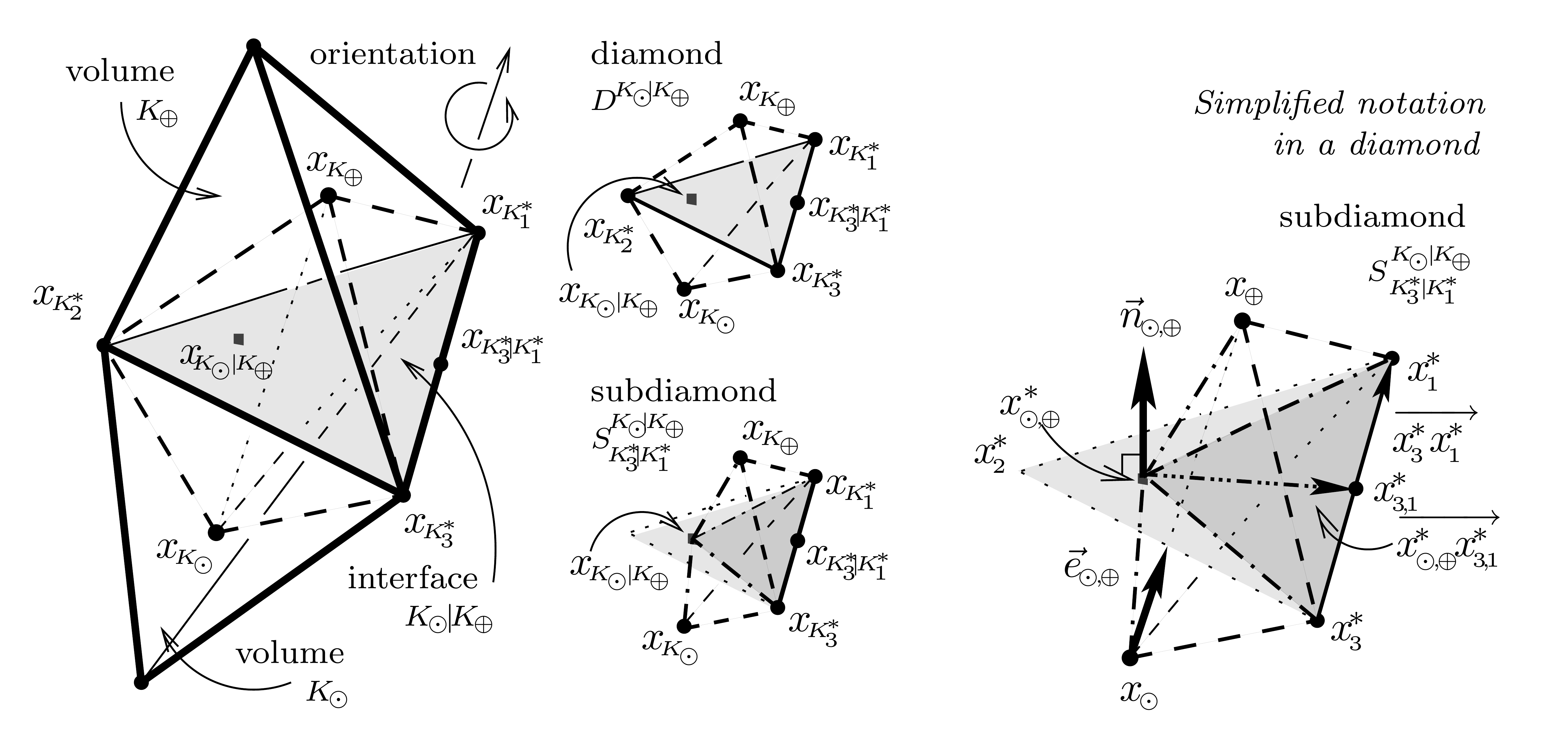}
\end{center}
\caption{Primal volumes, diamond; subdiamond\,and 
zoom\,on\,it \hspace*{45pt}~(Version $\!(B)$ of 3D DDFV mesh) } \label{Fig-3D}
\end{figure}

\medskip
\noindent $\bullet$ (See Figure~\ref{Fig-3D}) The following notations are only
needed for an explicit expression of
the discrete divergence operator (and also for the proof of
the discrete duality given in \cite{AndrBend}). It is convenient
to orient the axis $\xK\xL$ of each diamond $\DM$.
Whenever the orientation is of importance, the primal
vertices defining the diamond will be denoted by
$\xKdot,\xKplus$ in such a way that the vector
$\overrightarrow{\xKdot\xKplus}$ has the positive orientation.
The oriented diamond is then denoted by $\DM^{\ptKdotIKplus}$.
We denote by $ \nuKdotKplus$ the corresponding unit
vector, and by $ \dKdotKplus$, the length of $\overrightarrow{\xKdot\xKplus}$.
We denote by $ \NKdotIKplus$ the unit normal vector to
$\KdotIKplus$ such that $\NKdotIKplus\cdot \nuKdotKplus >0$.

Fixing the normal $\NKdotIKplus$ of $\KdotIKplus$
induces an orientation of the corresponding face
$\KdotIKplus$, which is a convex polygon with $l$ vertices
(we only use $l=3$ or $4$): we denote the
vertices of $\KdotIKplus$ by $\xdKi$, $i\in
[\![1,l]\!]$, enumerated in the direct sense. By convention, we
assign $x_{\ptK^{\!*}_{\!l\!+\!1}}:=x_{\ptK^{\!*}_{1}}$. We denote
by $\nudKidKiun$ the unit normal vector pointing from $\xdKi$
towards $\xdKiun$, and by $ \dKidKiun$, the length of
$\overrightarrow{\xdKi\xdKiun}$.

To simplify the notation, we will drop the $\K$'s in the
subscripts and denote the objects introduced above by
$\xdot$,$\xplus$,$\nudotplus$,$\ddotplus$,$\Ndotplus$ and by
$x^*_i$,$\nuiiun$,$\diiun$ whenever $\DM^{\ptKdotIKplus}$ is
fixed. We also denote by $\xidemi$ the middle-point
$x_\ptdKiIdKiun$ of the segment $[x_i,x_{\!i\hspace*{-1pt}+\!1}]$,
and by $\xdotplus$, the centre $x_\ptKdotIKplus$ of
$\KdotIKplus$.

\medskip
\noindent $\bullet$ For a diamond $\DM=\DM^{\ptKdotIKplus}$, we
denote by $\ProjD$ the orthogonal projection of $\R^3$
onto the line spanned by the vector $\nuKdotKplus$; we denote by $\dProjD$  the
orthogonal projection of $\R^3$ onto the plane containing
the interface $\KdotIKplus$.

\medskip
\noindent $\bullet$ We denote by $\Vol(A)$ the three-dimensional
Lebesgue measure of $A$ which can stand for a control volume, a
dual control volume, or a diamond. In particular, for 
$\K\in\MrondGN$, $\Vol(\K)=0$: these volumes are degenerate.
For a subdiamond
$\SDM=\SDM^{\ptKdotIKplus}_{\ptdKiIdKiun}$,
 we have the formula $\Vol(\SDM)=\frac 16
\,\langle\,\overrightarrow{\xdot\xplus},\overrightarrow{
\xdotplus\xidemi }, \overrightarrow{x^*_i x^*_{i\hspace*{-1pt}+\!1}} \,\rangle$.
Note the mixed product is positive, thanks to
our conventions on the orientation in $\DM^\ptKdotIKplus$ and
because we have assumed that $\xdotplus\in \KdotIKplus$.

\begin{rem}\rm \label{rem:Diamonds,Subdiamonds}
Diamonds permit to define the discrete
gradient operator, while subdiamonds permit to
give formulas for the discrete divergence operator (see \eqref{DiscrGrad},
\eqref{DiscrGradD3} and \eqref{DiscrDiv},
\eqref{DiscrDivM} below, respectively).

In the context of $2D$ ``double'' schemes, introducing
diamonds is quite standard (see, e.g., \cite{ABH``double'',DomOmnes}).
Subdiamonds are ``hidden'' in the $2D$ construction : they
actually coincide with diamonds.
\end{rem}

\medskip
\subsubsection{Discrete functions, fields, and boundary data.}~\\[-5pt]
\label{ssec:fctsfieldsdata}

\noindent $\bullet$
A {\it discrete function} $w^\ptTau$ on ${\Om}$ is a pair
$\Bigl(\wMrond,\wdMrond\Bigr)$ consisting of two sets of
real values $w^\ptMrond=(w_\ptK)_{\ptK\in\ptMrond}$ and
$w^\ptdMrond=(w_\ptdK)_{\ptdK\in\ptdMrond}$. The set of all such
functions is denoted by $\R^\ptTau$.

\medskip\noindent
$\bullet$ A {\it discrete field} $\vec \Frond^\ptTau$
on $\Om$ is a set $\Bigl( \vec {\mathcal F}_\ptD
\Bigr)_{\ptD\in\ptDrond}$ of vectors of $\R^d$. The set of all
discrete fields is denoted by $(\R^d)^\ptDrond$.
 If $\vec\Frond^\ptTau$ is a discrete field on
$\Om$, we assign $\vec \Frond_\ptS=\vec \Frond_\ptD$ whenever
$\SDM\subset \DM$.

\medskip
\noindent $\bullet$
A {\it discrete Dirichlet datum} $g^\ptTau$ on $\Gamma_D$ is a pair
$\Bigl(g^{\ptptl\ptMrond},g^{\ptptl\ptdMrond}\Bigr)$ consisting of two sets of
real values $g^{\ptptl\ptMrond}=(g_\ptK)_{\ptK\in\ptptl\ptMrond}$ and
$g^{\ptptl\ptdMrond}=(g_\ptdK)_{\ptdK\in\ptptl\ptdMrond}$.
In practise, $g_\K$ (resp., $g_\dK$) can be obtained
by averaging the ``continuous'' Dirichlet
datum $g$ over the boundary volume $\K\subset \Gamma_D$  (resp., over
the part of $\Gamma_D$ adjacent to the boundary
dual volume $\dK$); if $g$ is continuous, the mean
value can be replaced by the value of $g$ at $\xK$ (resp., at $\xdK$).
We refer to \cite{ABH``double''} for details.

\medskip

\noindent $\bullet$
A {\it discrete Neumann datum  $s^{\ptTau}$ on $\Gamma_N$} is
a set of real values $(s_\ptK)_{\ptK\in \ptMrondGN}$.

In practice, $s_\ptK$ can be obtained by averaging
the ``continuous'' Neumann
datum $s$ over the degenerate volume $\K\subset \Gamma_N$. 
In the case $\Gamma_D=\text{\O}$, one should
be careful while using approximate quadratures to produce $s^\ptTau$ from
$s$. Indeed, some compatibility conditions between
Neumann data and source terms may arise while discretising elliptic equations
(this is the case of system \eqref{S1-p}, because the
difference of the two equations of the system is an elliptic equation,
and the compatibility condition \eqref{eq:Neumann-compatibility} is
needed for the solvability of the system).
The compatibility condition, expressed in terms of
$\int_{\Gamma_N} s_{i,e}$, should be preserved at the discrete level.
This is the case for the above choice of $s^\ptTau$: indeed, we have
the equality $ \int_{\Gamma_N} s=\int_{\Gamma_N} s^\ptTau$.

\medskip
\subsubsection{The discrete gradient operator}\label{ssec:discrete grad}~\\[-5pt]

\noindent $\bullet$ On the set $\R^\ptTau$ of discrete
functions $w^\ptTau$ on ${\Om}$, we define the {\it discrete gradient}
operator $\grad_{g^\ptTau}^\ptTau[\cdot]$
with Dirichlet data $g^\ptTau$ on $\Gamma_D$:
\begin{equation}\label{DiscrGrad}
\grad_{g^\ptTau}^\ptTau: w^\ptTau\in \R^\ptTau \mapsto \grad_{g^\ptTau}^\ptTau
w^\ptTau=\Bigl( \grad_\ptD w^\ptTau \Bigr)_{\ptD\in\ptDrond}
\in (\R^d)^\ptDrond,
\end{equation}
where  the entry  $\grad_{\!\ptD} w^\ptTau$ 
of the discrete field $\grad^\ptTau w^\ptTau$ relative
to  $\DM=\DM^{\ptKdotIKplus}$ is
\begin{equation}\label{DiscrGradD3}
\text{$\grad_{\!\ptD} w^\ptTau$  is s.t. } \left\{
\begin{array}{l}
\dsp \ProjD(\grad_{\!\ptD} w^\ptTau) =
\frac{\wplus-\wdot}{\ddotplus}
\,\nudotplus,\\[8pt]
\dsp \dProjD(\grad_{\!\ptD} w^\ptTau) = \grad F(\cdot),
\end{array}
\right.
\ee
where
\begin{itemize}
\item[$\cdot$] $F(\cdot)$ is the affine function from $\R^3$ to $\R$ 
that is constant in the direction $\Ndotplus$
orthogonal to $\KIL$ and that is  the {\it ad hoc} affine 
interpolation (namely, \eqref{eq:interpolation} below) of the values
$w^*_i$ at the vertices $x^*_i$, $i=1,\ldots,l$, of $\KIL$;

\smallskip
\item[$\cdot$] for the vertices of $\DM$ lying in
$\Om\cup\Gamma_N$, $\wdot\!=\!w_\ptKdot,\,\wplus\!=\!w_\ptKplus$, $w^*_i
=w_{\ptK^{\!*}_{\!i}}$, etc. (we use the simplified notation
in the diamond $\DM=\DM^{\ptKdotIKplus}$, as depicted in
Figure~\ref{Fig-3D}).  For the vertices of $\DM$ that
lie in $\Gamma_D$, the values of $g^\ptTau$ are
used: e.g., if $x_\ptKdot\in\Gamma_D$, then
we set $\wdot:=g_\ptKdot$ in the above formula.\smallskip
\end{itemize}

Clearly, if $l=3$ there is a unique consistent
interpolation of the values $w_1,w_2,w_3$.
For $l\geq 4$, no consistent interpolation exists, and 
we choose the linear form in $w^*_{\!i}$ that leads to the expression
\begin{equation}\label{eq:interpolation}
\dProjD(\grad_{\!\ptD} w^\ptTau)
\!=\!
\frac 2{\sum_{i=1}^l
\,\langle\,\Ndotplus,\overrightarrow{\xdotplus
\xidemi},\overrightarrow{x^*_i
x^*_{i\hspace*{-1pt}+\!1}}\,\rangle}\; \sum_{i=1}^l
 (w^*_{\!i\hspace*{-1pt}+\!1}-w^*_{\!i})\;
 \Bigl[\Ndotplus\times
\overrightarrow{\xdotplus \xidemi}\Bigr];
 \end{equation}
it is shown in \cite{AndrBend} that this choice is exact on
affine functions, and  that it leads to the discrete duality formula.
For an explicit formula of $\grad_{\!\ptD} w^\ptTau$, note that
$\vec p=\ProjD(\grad_{\!\ptD} w^\ptTau)$, $\vec{p^*}
=\dProjD(\grad_{\!\ptD} w^\ptTau)$ are given; then one 
expresses $\grad_{\!\ptD} w^\ptTau$ as
\begin{equation}\label{eq:grad-representation}
\grad_{\!\ptD} w^\ptbarTau=\frac 1{\Vol(\DM)} \sum_{i=1}^l
\biggl\{  \frac{\Vol( \SDM^{\ptKdotIKplus}_{\ptdKiIdKiun})}
{\overrightarrow{\xdot\xplus} \cdot \Ndotplus} (\wplus-\wdot)
\,\Ndotplus +  \frac 13  (w^*_{\!i\hspace*{-1pt}+\!1}-w^*_{\!i})\;
\Bigl[ \overrightarrow{\xdot\xplus}\times
\overrightarrow{\xdotplus \xidemi}\Bigr] \biggr\}
\end{equation}

\begin{rem}\rm \label{rem:DiscrGradSense}
In \eqref{DiscrGradD3}, the primal mesh $\overline{\Mrond}$
serves to reconstruct one component of the
gradient, which is the one in the direction
$\nudotplus$. The dual mesh $\overline{\dMrond}$ serves to
reconstruct, with the help of the formula \eqref{eq:interpolation},
the two other components, which are those lying in the plane
containing $\KdotIKplus$. The same happens for
version $(A)$ of the scheme. On the contrary, version $(C)$ only
reconstructs one direction of the discrete gradient
on the mesh $\dMrond$, while the third direction is reconstructed
 on a third mesh that we denote by $\diezMrond$.
\end{rem}

\begin{rem}\rm \label{rem:consistAffines}
We stress that our gradient approximation is consistent
(see \cite{AndrBend} for the proof). Indeed, let
$\wdot,\wplus, (w^*_{i,i\hspace*{-1pt}+\!1})_{i=1}^l$ be the
values at the points $\xdot,\xplus, (\xidemi)_{i=1}^l$,
respectively, of an affine on $\DM=\DM^\ptKdotIKplus$ function
$w$. Then $\grad_\ptD w^\ptTau$ coincides
with the value of $\grad w$ on $\DM$.
\end{rem}

\subsubsection{The discrete divergence operator}\label{ssec:discrete div}~
\medskip

\noindent $\bullet$ On the set $(\R^d)^\ptDrond$ of discrete
fields $\vec \Frond^\ptTau$, we define the
{\it discrete divergence} operator $\div_{\!\!s^\ptTau}^{\!\!\ptTau}[\cdot]$
with Neumann data $s^\ptTau$ on $\Gamma_N$:
\begin{equation}\label{DiscrDiv}
\begin{split}
\div_{\!\!s^\ptTau}^{\!\!\ptTau}:\vec \Frond^\ptTau\in (\R^d)^\ptDrond
& \mapsto \div_{\!\!s^\ptTau}^{\!\!\ptTau} \vec \Frond^\ptTau
\\ & =
\biggl(\,\Bigl(\div_{\!\!\ptK} \vec \Frond^\ptTau\Bigr)_{\ptK\in\ptMrond=\ptMrondOm\cup\ptMrondGN}\,,\,
 \Bigl(\div_{\!\!\ptdK} \vec \Frond^\ptTau
 \Bigr)_{\ptdK\in\ptdMrond}\,\biggr) \in \R^\ptTau,
\end{split}
\end{equation}
where the entries $\div_{\!\!\ptK} \vec \Frond^\ptTau$ (for $\K\in\MrondOm$) and  $\div_{\!\!\ptdK}
\vec \Frond^\ptTau$ of the discrete function
$\div^{\!\!\ptTau} \vec \Frond^\ptTau$ on $\Om$ are given by
\begin{equation}\label{DiscrDivM-simple}
\begin{array}{l}
\dsp
\forall\,\K\in \MrondOm,\;\;  \div_{\!\!\ptK} \vec \Frond^\ptTau=
\frac{1}{\Vol(\ptK)}
\sum_{\ptD\in\ptDrond :\, \ptD\cap \ptK\neq\text{\O}} \int_{\ptptl\ptK\cap\ptD}
\vec \Frond_D\cdot n_\ptK,\\[7pt]
\dsp \forall\,\dK\in \dMrond,\;\;
\div_{\!\!\ptdK} \vec \Frond^\ptTau=
 \frac{1}{\Vol(\ptdK)}
\sum_{\ptD\in\ptDrond :\, \ptD\cap \ptdK\neq\text{\O}}
\int_{\ptptl\ptdK\cap\ptD}\vec \Frond_D\cdot n_\ptdK,
\end{array}
\end{equation}
where $n_\ptK$ (resp., $n_\ptdK$) denotes the exterior unit
normal vector to $\K$ (resp., to $\dK$). Further, for $\K\in\MrondGN$, we mean that $n_\ptK$ points inside $\Om$,
and we adapt the following formal definition:
\begin{equation}\label{eq:division-by-zero-convention}
\begin{array}{l}
\dsp
\forall\,\K\in \MrondGN,\;\;   {\Vol(\K)}\div_{\!\!\ptK} \vec \Frond^\ptTau:=
\vec \Frond_D\cdot n_\ptK + s_\ptK   \;\\
\dsp\text{for the diamond $\DM$ such that $\overline{\DM}\cap \Gamma_N=\K$};
\end{array}
\end{equation}
thus, although $\Vol(\K)$ is zero, in calculations we only use the products $\Vol(\K)  \div_{\!\!\ptK} \vec \Frond^\ptTau$,
which are well defined thanks to convention \eqref{eq:division-by-zero-convention}. 
In practise, the discrete equations corresponding  to volumes of $\MrondGN$ will always read as
\begin{equation}\label{eq:boundary-div-eqns}
\vec \Frond_D\cdot n_\ptK + s_\ptK =0, \quad \K\in\MrondGN.
\end{equation}
Notice that the values of the Neumann data $s^\ptTau$ only appear 
in the convention \eqref{eq:division-by-zero-convention} for 
the degenerate primal volumes $\K\subset\Gamma_N$; 
at the same time,  in the volumes $\dK$ adjacent to $\Gamma_N$
the data $s^\ptTau$ are taken into account indirectly. Namely, let $\dK$ be 
a dual volume adjacent to $\Gamma_N$, and let $\DM$ be a diamond  intersecting $\dK$ and 
adjacent to $\Gamma_N$; then the value
$\Frond_D\cdot n_\ptK$ used for the definition of $\div_{\!\!\ptdK} \vec \Frond^\ptTau$
is linked to the data $s^\ptTau$ via equations \eqref{eq:boundary-div-eqns}.

\medskip
The formulas \eqref{DiscrDivM-simple} are standard for divergence discretisation
in finite volume methods; their interpretation is
straightforward, using the Green-Gauss theorem. 
The consistency of the discrete divergence operator  (in the weak sense) can
be inferred by duality from the one of the discrete gradient operator
(Remark~\ref{rem:consistAffines}) and from the discrete
duality property \eqref{eq:discr-duality}; see 
Proposition~\ref{PropConsistency}(iii) and \cite{AndrBend}.

For the explicit calculation of the right-hand sides in \eqref{DiscrDivM-simple},
one can further split diamonds into subdiamonds.
In a generic subdiamond, we use the following notation.
Consider $\SDM\in\Srond$; it is associated with a unique oriented diamond
which we denote $\DM^{\ptKdotIKplus}$, so that $\SDM$ is of the
form $\SDM=\SDM^{\ptKdotIKplus}_{\ptdKiIdKiun}$.
In order to cope with the vector orientation issues, given
$\SDM=\SDM^{\ptKdotIKplus}_{\ptdKiIdKiun}$ we define
\begin{equation*}
\epsSK:=
\begin{cases}
  0, \text{ if
$\K=\Kdot$}\\
1, \text{ if $\K=\Kplus$}
\end{cases}
, \qquad
\epsSdK:=
\begin{cases}
  0, \text{ if
$\dK=\K^*_i$}\\
1, \text{ if $\dK=\K^*_{i\hspace*{-1pt}+\!1}$}
\end{cases}
 .
\end{equation*}
 For $\K\in\Mrond$, we denote by $\Vrond(\K)$
the set of all subdiamonds $\SDM\in \Srond$ such that $\K\cap \SDM
\neq \text{\O}$. In the same way, for $\dK\in\dMrond$
we define the set $\dVrond(\dK)$ of the subdiamonds intersecting
$\dK$. Then, using the notation $\langle \cdot,\cdot,\cdot \rangle$ for the
mixed product on $\R^3$, we can express
formulae \eqref{DiscrDivM-simple} as
\begin{equation}\label{DiscrDivM}
\begin{array}{l}
\dsp
\forall\,\K\in \MrondOm\;\;\;\; \div_{\!\!\ptK} \vec \Frond^\ptTau=
\frac{1}{2\Vol(\ptK)}\!\!\!\!\!\!\sum_{\hspace*{3mm}\ptS\in\ptVrond(\ptK)}\!\!\!\!
(-1)^\ptepsSK  \langle\,\vec\Frond_{\!\ptS},
\overrightarrow{\xdotplus
\xidemi}, \overrightarrow{x^*_i x^*_{i\hspace*{-1pt}+\!1} },
\,\rangle 
\\[7pt]
\dsp
\forall\,\dK\in \dMrond\;\;\;\; \div_{\!\!\ptdK} \vec \Frond^\ptTau=
 \frac{1}{2\Vol(\ptdK)}
\!\!\!\!\!\!\!\!\!\!\sum_{\hspace*{5mm}\ptS\in\ptdVrond(\ptdK)}\!\!\!\!\!\!\!\!
(-1)^\ptepsSdK \! \langle\,\vec \Frond_{\!\ptS},
\overrightarrow{\xdot\xplus},
\overrightarrow{\xdotplus\xidemi}\,\rangle.
\end{array}
\end{equation}
In \eqref{DiscrDivM}, each subdiamond $\SDM$
in $\Vrond(\K)$ (or in $\dVrond(\dK)$) has the form
$\SDM=\SDM^{\ptKdotIKplus}_{\ptdKiIdKiun}$, with some
$\Kdot,\Kplus,\K^*_i,\K^*_{i\hspace*{-1pt}+\!1}$; the notations $\epsSK,
\epsSdK,\xdot,\xplus,\xdotplus,\xidemi,x^*_i,x^*_{i\hspace*{-1pt}+\!1}$
refer to $\SDM=\SDM^{\ptKdotIKplus}_{\ptdKiIdKiun}$  (see Figure~\ref{Fig-3D}).
Details can be found in \cite{AndrBend}.

\begin{rem}\label{rem:reduction-to-duality}
In practise it is not necessary to calculate the discrete divergence; 
indeed, with the help of the duality property, one can
express the discrete system of equations in the dual form, 
where the calculation of the discrete divergence of the solution 
is replaced by the calculation of the discrete gradient of 
a test function. Thus, as for the so-called mimetic finite difference methods, in order 
to formulate a DDFV scheme it is sufficient to calculate discrete gradients. This 
amounts to the ``discrete weak 
formulation'' \eqref{eq:dualformulation} we use in the sequel.
\end{rem}

\medskip
\subsubsection{The scalar products $\Bleft\cdot,\cdot\Bright_{\Om},
\Aleft\cdot,\cdot\Aright_{\Om},\Cleft \cdot,\cdot \Cright_{\Gamma_N}$
and discrete duality}~

\noindent
 $\bullet$  Recall that $\R^\ptTau$ is the space of all
discrete functions on $\Om$. For  $w^\ptTau,v^\ptTau \in
\R^\ptTau$, set
\begin{equation*}
\Bleft w^\ptTau,\,v^\ptTau   \Bright =
\frac{1}{3}\sum_{\ptK\in\ptMrond} \Vol(\K)\;w_\ptK  v_\ptK \;+\;
\frac{2}{3} \sum_{\ptdK\in\ptdMrond} \Vol(\dK)\;w_\ptdK  v_\ptdK.
\end{equation*}

Recall that $(\R^3)^\ptDrond$ is the space
of discrete fields on $\Om$. For $\vec \Frond^\ptTau,\vec
\Grond^\ptTau \in (\R^3)^\ptDrond$, set
\begin{equation*}
\Aleft \vec \Frond^\ptTau,\, \vec \Grond^\ptTau   \Aright =
\sum_{\ptD\in\ptDrond} \Vol(\DM) \, \vec \Frond_\ptD \cdot \vec
\Grond_\ptD.
\end{equation*}

\medskip
\noindent $\bullet$
Recall that in \eqref{eq:DiscreteSol}, for version $(B)$ of the scheme, given
a discrete function $v^\ptTau$,  we set
$$v^\ptTau(x)=\frac 13 \sum_{\ptK\in\ptMrond} v_\ptK \char_\ptK(x) + \frac 23
\sum_{\ptdK\in\ptdMrond} v_\ptdK \char_\ptdK(x).$$
Then the function $ v^{\ptl\ptTau}\in L^\infty(\Gamma_N)$
can be defined as the trace of $v^\ptTau$ on $\Gamma_N$.
This means, $v^{\ptl\ptTau}(x):=\frac{1}{3}v_\ptK +
\frac{2}{3}v_\ptdK$ where for $\mathcal H^{2}$-a.e $x\in
\Gamma_N$, $\K$ and $\dK$ are uniquely defined by the fact that
$x\in\overline{\K\cap \dK}$.

\medskip
\noindent $\bullet$
 Finally, for $\Cleft \cdot,\cdot \Cright_{\Gamma_N}$, we
 simply use the $L^2$ scalar product on $\Gamma_N$.

\medskip
Now a straightforward adaptation of the proof
of the discrete duality property in \cite{AndrBend} yields the
desired discrete duality property \eqref{eq:discr-duality}.

\section{The DDFV schemes and convergence results}\label{sec:Scheme-and-Results}

The time-implicit DDFV finite volume schemes for
Problem \eqref{S1-p},\eqref{S2},\eqref{S3} can be
formally (up to convention \eqref{eq:division-by-zero-convention})
written under the following general form:
\begin{equation}\label{all-AbstractScheme}
 \left\{\begin{array}{l}
\dsp \text{find
$\biggl((u_i^{\ptTau,n},u_e^{\ptTau,n},v^{\ptTau,n})\biggr)_{n=1,\ldots,N}\subset
(R^\ptTau)^3$ satisfying the equations}\\[12pt] \dsp
\frac{v^{\ptTau,n+1} - v^{\ptTau,n}}{\Delt}
-\div^\ptTau_{\!\!s_i^{\ptTau,n+1}}[\bM_{i}^\ptTau\grad^\ptTau_{\!g_i^{\!\ptTau,n+1}}
u_i^{\ptTau,n+1}]+
h^{\ptTau,n+1}
-\Iap^{\ptTau,n+1}=0,\\[5pt]
\dsp \frac{v^{\ptTau,n+1} - v^{\ptTau,n}}{\Delt}
+\div^\ptTau_{\!\!s_e^{\ptTau,n+1}}[\bM_{e}^\ptTau \grad^\ptTau_{\!g_e^{\!\ptTau,n+1}}
u_e^{\ptTau,n+1}]+
h^{\ptTau,n+1}
-\Iap^{\ptTau,n+1}=0,\\[8pt]
\dsp  v^{\ptTau,n+1}-(u_i^{\ptTau,n+1}-u_e^{\ptTau,n+1})=0,
\end{array}\right.
\end{equation}
\begin{equation}\label{eq:discr-IC}
v^{\ptTau,0}= v_0^\ptTau.
\end{equation}
For two rigorous interpretations of \eqref{all-AbstractScheme}, see 
Definition~\ref{def:scheme-solution} below.

We normalise $u_e^{\ptTau,n+1}$ by requiring, for all $n=1,\ldots,N$,
\begin{equation}\label{eq:discrete-normalization}
	\begin{split}
		&\text{if $\Gamma_D=\text{\O}$, then 
		$\sum_{\ptK\in{\ptMrondOm}}\! \Vol(\K) u_{e,\ptK}\!=\!0$,} 
		\\ &\quad \text{$\sum_{\ptdK\in\overline{\ptdMrond}}\! \Vol(\dK) u_{e,\ptdK}\!=\!0$,
		$\sum_{\ptdiezK\in\overline{\ptdiezMrond}}\! \Vol(\diezK) u_{e,\ptdiezK}\!=\!0$}
	\end{split}
\end{equation}
(the last condition is only meaningful for the meshing $(C)$).

The triple $(u_i^{\ptTau,n+1},u_e^{\ptTau,n+1},v^{\ptTau,n+1})$
constitutes the unknown discrete functions at time level $n$;
$v_0^\ptTau$ and $\Iap^{\ptTau,n+1}$ stand for the projections of
the initial datum $v_0$ and the source term $\Iap$ on the space of
discrete functions. Similarly,
$g_{i,e}^{\!\ptTau,n+1}$,$s_{i,e}^{\!\ptTau,n+1}$ are suitable
projections of the Dirichlet and Neumann data $g_{i,e},s_{i,e}$,
respectively. Notice that the boundary data are taken into account in the definition
of the discrete operators  $\grad^\ptTau_{\!g^{\!\ptTau}}$, $\div^\ptTau_{\!\!s^{\ptTau}}$.
The matrices $\bM_{i,e}^\ptTau(\cdot)$ are the
projections of $\bM_{i,e}(\cdot)$ on the diamond
mesh.
We will mainly work with the
mean-value projections; e.g., the projection $\mathbb{P}^\ptTau$ on $\Tau$ of $v_0$
would be the discrete function with the entry
$\frac{1}{\Vol(\ptK)}\dsp\int_\ptK v_0$ corresponding to a control volume $\K$.
For regular functions, the centre-valued projection $\mathbb{P}^\ptTau_c$
can be considered, where the entry $v_0(x_\ptK)$ corresponds to a
volume $\K$. We refer to Sections~\ref{ssec:fctsfieldsdata},~\ref{ssec:proj-ops} for 
details on  the projection operators in use.

\smallskip
A relation that links $h^{\ptTau,n+1}$ to $v^{\ptTau,n+1}$ closes the scheme;
we consider the following two choices: the fully implicit scheme,
\begin{equation}\label{eq:h-fully-impl}
h^{\ptTau,n+1}=\mathbb{P}^\ptTau h(v^{\ptTau,n+1}(\cdot)),
\end{equation}
and the linearised implicit scheme
\begin{equation}\label{eq:h-linearized-impl}
h^{\ptTau,n+1}=  \mathbb{P}^\ptTau\biggl(\,(\tilde
b(v^{\ptTau,n}(\cdot))-L)\,v^{\ptTau,n+1}(\cdot)-l \biggr).
\end{equation}
where  $\mathbb{P}^\ptTau$ is the projection operator acting from $L^1(\Om)$
into the space of the corresponding discrete functions;
further,  $v^{\ptTau,n+1}(\cdot)$ define the
piecewise constant functions reconstructed according to
\eqref{eq:DiscreteSol} from the values
$v^{\ptTau,n+1}=\Bigl(v^{\ptMrond,n+1},v^{\ptdMrond,n+1}\Bigr)$
(for versions $(A)$ and $(B)$)
or $v^{\ptTau,n+1}=\Bigl(v^{\ptMrond,n+1},v^{\ptdMrond,n+1},
v^{\ptdiezMrond,n+1}\Bigr)$ (for version $(C)$).
 The same convention applies to $v^{\ptTau,n}(\cdot)$.
We refer to Section~\ref{ssec:reaction-term} for a detailed
description of such discretisation of the ionic current term.

\begin{rem}\label{rem:h-reconstruction}\rm
In the discretisation of the ionic current term $h(v)$,
the choice \eqref{eq:h-fully-impl},\eqref{eq:h-linearized-impl}
is made to reconstruct the $L^1$ function $v^\ptTau(\cdot)$ and
then to re-project it on the mesh $\Tau$.
This is tricky and it may seem unnatural. But we explain in
Section~\ref{ssec:reaction-term} that this is the way to ensure
that the structure of the reaction terms in the discrete equations
yields exactly the same {\it a priori} estimates
as for the continuous problem.

The seemingly simpler choice $h^{\ptTau,n+1}
=h(v^{\ptTau,n+1})$ (instead of \eqref{eq:h-fully-impl})
does not have good structure properties; we can justify
the convergence of the associated scheme
by adding a penalisation term (cf. \cite{ABK}) whose role is to
make small the differences $v^{n+1}_\ptK-v^{n+1}_\ptdK$,
for $\K\cap\dK\neq \text{\O}$, in the left-hand
side of the scheme \eqref{all-AbstractScheme}.
\end{rem}

\begin{defi}\label{def:scheme-solution}
  A {\it discrete solution} is a set
$\biggl((u_i^{\ptTau,n+1},u_e^{\ptTau,n+1},v^{\ptTau,n+1})\biggr)_{n\in [0,N]}$
(in the sequel, we denote it by
$(u_i^{\ptTau,\ptDelt},u_e^{\ptTau,\ptDelt},v^{\ptTau,\ptDelt})$)
satisfying the initial data \eqref{eq:discr-IC}, the 
normalisation equations \eqref{eq:discrete-normalization}, and 
the closure relation \eqref{eq:h-fully-impl} or \eqref{eq:h-linearized-impl}; 
moreover, it should
solve system \eqref{all-AbstractScheme} in the following sense:\\[3pt]
$\cdot$ equalities in \eqref{all-AbstractScheme} hold component per component for all entries corresponding to primal volumes $\K\in \MrondOm$ and those corresponding
to the dual volumes $\dK\in\dMrond$;\\[2pt]
$\cdot$ for the entries corresponding to $\K\in\MrondGN$, convention \eqref{eq:division-by-zero-convention} is used, that is,
the equations take the form $v^{\ptTau,n+1}-(u_i^{\ptTau,n+1}-u_e^{\ptTau,n+1})=0$ and
$$ (\bM_{i,e})_\ptD \grad_{\ptD} u^{\ptTau,n+1}\cdot n_\ptK + (s_{i,e})^{n+1}_\ptK =0
\quad \text{for $\DM\in\Drond$ such that $\overline{\DM}\cap \Gamma_N=\K$};$$

Equivalently, $(u_i^{\ptTau,\ptDelt},u_e^{\ptTau,\ptDelt},v^{\ptTau,\ptDelt})$ 
is a discrete solution if
$v^{\ptTau,\delt}=u_i^{\ptTau,\delt}- u_e^{\ptTau,\delt}$ and
for all $\ph^\ptTau\in R^\ptTau$,
for all $n\in [0,N]$ the following identities hold:
\begin{equation}\label{eq:dualformulation}
\left\{\begin{array}{l}
\dsp
\frac1{\Delt} \Bleft{v^{\ptTau,n+1} - v^{\ptTau,n}},\ph^\ptTau\Bright_\Om
+\Aleft \bM_{i}^\ptTau\grad^\ptTau_{\!g_i^{\!\ptTau,n+1}}
u_i^{\ptTau,n+1},\grad^\ptTau_0 \ph^\ptTau\Aright_\Om\\[5pt]
\dsp\qquad\qquad  + \Cleft s_i^\ptTau, \ph^{\ptl\ptTau} \Cright_{\Gamma_N}
+ \Bleft h^{\ptTau,n+1}-\Iap^{\ptTau,n+1}, \ph^\ptTau\Bright_\Om =0,\\[8pt]
\dsp
\frac1{\Delt} \Bleft{v^{\ptTau,n+1} - v^{\ptTau,n}},\ph^\ptTau\Bright_\Om
-\Aleft \bM_{i}^\ptTau\grad^\ptTau_{\!g_e^{\!\ptTau,n+1}}
u_e^{\ptTau,n+1},\grad^\ptTau_0 \ph^\ptTau\Aright_\Om\\[5pt]
\dsp\qquad\qquad  - \Cleft s_e^\ptTau, \ph^{\ptl\ptTau} \Cright_{\Gamma_N}
+ \Bleft h^{\ptTau,n+1}-\Iap^{\ptTau,n+1}, \ph^\ptTau\Bright_\Om =0.
\end{array}\right.
\end{equation}
\end{defi}
Notice that the equivalence of the two above 
formulations of \eqref{all-AbstractScheme} is easy to establish;
namely, the discrete duality property \eqref{eq:discr-duality} is 
used together with the choice of discrete test 
functions $\ph^\ptTau$ that only contain one non-zero entry.

\medskip
The existence of solutions to the discrete equations
is obtained in a standard way from the Brouwer
fixed-point theorem and the coercivity enjoyed by our schemes;
the uniqueness proof mimics the one of Theorem~\ref{theo:L^2-contraction}.
More precisely, we have

\begin{prop}\label{prop:schemeexist}
Assume \eqref{eq:h-conseq},\eqref{monotone-bis}.
Whenever $\delt< \frac 1{2L}$, for all given boundary
data satisfying \eqref{eq:Neumann-compatibility} 
(if $\Gamma_D=\text{\O}$, we add \eqref{eq:discrete-normalization} to the scheme) and for all
given initial data \eqref{eq:discr-IC} there exists one and only one discrete solution to the scheme
\eqref{all-AbstractScheme},\eqref{eq:h-fully-impl}; likewise, there exists
one and only one discrete solution to the scheme
\eqref{all-AbstractScheme},\eqref{eq:h-linearized-impl}.
Moreover, for fixed boundary data, the discrete $L^2$ contraction property holds
for the $v^{\ptTau,\delt}$ component of  the solution of the fully implicit scheme
\eqref{all-AbstractScheme},\eqref{eq:discr-IC},\eqref{eq:h-fully-impl}:
Indeed, for all $n\!\in\![0,N]$,
\begin{equation}\label{eq:discr-contdep}
\Bleft v^{\ptTau,n+1}\!\!-\!\hat v^{\ptTau,n+1}\,,\,
v^{\ptTau,n+1}\!\!-\!\hat v^{\ptTau,n+1}\Bright_{\Om}
\;\leq\; e^{L(n+1)\delt}\Bleft v^{\ptTau,0}\!\!-\!\hat
v^{\ptTau,0} \,,\, v^{\ptTau,0}\!\!-\!\hat
v^{\ptTau,0}\Bright_{\Om}.
\end{equation}
\end{prop}

\begin{rem}\label{rem:optimization}\rm
Let us point out that the fully implicit scheme leads, at each
time level, to a nonlinear system of equations, and to
compute the solution given by Proposition~\ref{prop:schemeexist}
(or, rather, a reasonable approximation to it) we can
use the following variational formulation of the scheme:
\begin{equation*}
 \left|
 \begin{array}{ll}
 & \text{at the time level $n$, minimise over
 $\R^\ptTau\times\R^\ptTau$ the functional}\\[10pt]
 & \!\!\!J\Bigl[u_i^\ptTau,u_e^\ptTau\Bigr]\,:=\, \frac 1{2\delt}
 \Bleft  v^\ptTau\,,\,v^\ptTau \Bright_\Om - \frac 1{\delt}
 \Bleft  v^\ptTau\,,\,v^{\ptTau,n} \Bright_\Om
 +  {\dsp\int_\Om} H(v^{\ptTau}(\cdot))\\[8pt]
 & \qquad\qquad\qquad+ \Aleft \bM_i^{\ptTau}
 \grad^\ptTau_{g_i^{\ptTau,n+1}} u^\ptTau\,,\,
 \grad^\ptTau_{g_i^{\ptTau,n+1}} u^\ptTau
 \Aright_\Om \\[8pt] & \qquad\qquad\qquad+ \Aleft \bM_e^{\ptTau}
 \grad^\ptTau_{g_e^{\ptTau,n+1}} u^\ptTau\,,\,
 \grad^\ptTau_{g_e^{\ptTau,n+1}} u^\ptTau
 \Aright_\Om\\[8pt]
 & \qquad\qquad\qquad- \Cleft s_i^{\ptTau,n+1}\,,\,
 u_i^{\ptl\ptTau}\Cright_{\Gamma_N}-\Cleft s_e^{\ptTau,n+1}\,,
 \,u_e^{\ptl\ptTau}\Cright_{\Gamma_N}
 - \Bleft  \Iap^{\ptTau,n+1}\,,\,v^\ptTau \Bright_\Om,\\[14pt]
 & \text{where $v^\ptTau:=u_i^\ptTau-u_e^\ptTau$,
 and $H:z\mapsto \int_0^z h(s)\,ds$ is the primitive of $h$}
 \end{array}
 \right.
\end{equation*}
(in the case $\Gamma_D=\text{\O}$, the
constraint \eqref{eq:discrete-normalization}
should be added on the domain of the functional $J$).
Similarly to the argument in \cite{ABH``double''}, it is checked from the
discrete duality formula and from
formula \eqref{eq:reaction-summbyparts} in Section~\ref{ssec:reaction-term}
that the scheme \eqref{all-AbstractScheme} is the Euler-Lagrange
equation for the above problem. From the properties
of $h(\cdot)$ it follows that for $\delt<\frac 1{2L}$,
we are facing a minimisation problem for the convex coercive functional $J$.
Thus descent iterative methods can be used
for solving the discrete system \eqref{all-AbstractScheme} at each time step.
\end{rem}
Now we can state the main result of this paper.
\begin{theo}\label{th:convergence}
Assume \eqref{eq:h-conseq},\eqref{monotone-bis} hold with some $r\geq 2$.
Assume that the family of meshes satisfies the regularity assumptions
\eqref{eq:mesh-regularity},\eqref{eq:inclination-bound},\eqref{eq:neighbours-bound}
(and the analogous restrictions on the mesh $\diezMrond$, for
version $(C)$) stated in
Section~\ref{ssec:MeshesRegularity}. Then
\begin{itemize}
\item[(i)] the sequence of solutions
$\Bigl(u_i^{\ptTau,\delt}(\cdot),u_e^{\ptTau,\delt}(\cdot),v^{\ptTau,\delt}(\cdot)  \Bigr)$
to the fully implicit scheme \eqref{all-AbstractScheme},\eqref{eq:discr-IC},
\eqref{eq:h-fully-impl},\eqref{eq:DiscreteSol}
converges, as the approximation parameters $\Delta x,\Delta t$ tend to zero,
to the unique solution $(u_i,u_e,v)$ of Problem \eqref{S1-p},\eqref{S2},\eqref{S3};
the convergence is strong in $L^2(Q)\times L^2(Q)\times L^r(Q)$.
Moreover, the discrete gradients converge to $(\grad u_i,\grad u_e,\grad v)$
strongly in $\Bigl(L^2(Q)\Bigr)^3$;

\smallskip
\item[(ii)] For any $r<16/3$, the  statement analogous to (i)
holds for the discrete solutions of the linearised implicit scheme
\eqref{all-AbstractScheme},\eqref{eq:discr-IC},\eqref{eq:h-linearized-impl},\eqref{eq:DiscreteSol}.
\end{itemize}

If $\Gamma_D=\text{\O}$, the constraint \eqref{eq:discrete-normalization}
should be added to the equations of the scheme.
\end{theo}

In the same vein, the standard 2D DDFV construction can be applied to problem
\eqref{S1-p},\eqref{S2},\eqref{S3} on 2D polygonal domains. The convergence result of
Theorem~\ref{th:convergence}(i) remains true, and the one
of Theorem~\ref{th:convergence}(ii) extends to all $r<6$.
We stress that the realistic case $r=4$ is
covered by our convergence results.

\section{Discrete functional analysis tools 
for DDFV schemes}\label{sec:DiscreteCalculus}

For a given mesh $\Tau$ of $\Om$ as described
in Section~\ref{sec:DDFV}, the size of $\Tau$ is defined as
$$
\size(\Tau):=\max\left\{\max\limits_{\ptK\in\overline{\ptMrond}}
\text{diam}(\K)\,,\,\max\limits_{\ptdK\in\overline{\ptdMrond}}\text{diam}(\dK)
\,,\,\max\limits_{\ptD\in\ptDrond}\text{diam}(\DM)   \right\}.
$$
If the assumption $\xK\in\K$   is relaxed, $\text{diam}(\K)$ must be
replaced with $\text{diam}(\K\cap\{\xK\})$ in the above expression.

In what follows, we will always think of a family of meshes
such that $\size(\Tau)$ goes to zero.

\subsection{Regularity assumptions on the meshes}\label{ssec:MeshesRegularity}~

In different finite volume methods, one always needs some qualitative
restrictions on the mesh $\Tau$ (such as, e.g., $\xK\in \K$, or
the convexity of volumes and/or diamonds, or the mesh orthogonality, or
the Delaunay condition on a simplicial mesh). For the convergence
analysis with respect to families of such meshes, it is  convenient
(though not always necessary) to impose shape regularity
assumptions. These assumptions are quantitative: this means that 
the ``distortion'' of certain objects in a mesh
is measured with the help of a regularity constant
$\regmesh$, which is finite for each individual mesh but may get
unbounded if an infinite family of meshes is considered.
For the 3D DDFV meshes presented in this paper, there are
two main mesh regularity assumptions.
First, we require several lower bounds on $\dKL,\ddKdL$:
\begin{equation}\label{eq:mesh-regularity}
\left|
\!\!\!\!\begin{array}{l}
\begin{array}{ll}
\text{$\forall$ neighbours $\K,\L$},\; \diam(\K)+\diam(\L)\leq \regmesh \dKL;\\
\text{$\forall$ dual neighbours $\dK,\dL$},\; \diam(\dK)+\diam(\dL)\leq \regmesh \ddKdL;\\[5pt]
\end{array}\\
\begin{array}{ll}
\text{$\forall$ diamonds $\DM$ with vertices $\xK,\xL$ and with}&\\
\text{neighbour dual vertices $x_\ptdK,x_\ptdL$},
\diam(\DM)\leq \regmesh \min\{\dKL,\ddKdL\}.
\end{array}
\end{array}
\right.
\end{equation}
Further, we need  a bound on the inclination of the
(primal and dual) interfaces with respect to the
(dual or primal) edges:
\begin{equation}\label{eq:inclination-bound}
\left|\begin{array}{l}
\text{$\forall$ primal neighbour volumes $\K,\L$, the angle
$\alpha_{\!\ptK\!,\ptL}$ between $\overrightarrow{\xK x_\ptL}$ and the}\\
\text{plane $\KIL$ is separated from $0$ and $\pi$, meaning that
$\regmesh \cos \alpha_{\!\ptK\!,\ptL}\geq 1$};\\[5pt]
\text{$\forall$ neighbour vertices $\xdK,\xdL$ of $\KIL$, the
angle $\alpha^*_{\!\ptdK\!\!,\ptdL}$ between $\overrightarrow{\xdK \xdL}$}
\\
\text{and $\overrightarrow{x_{\ptdK\ptI\ptdL}x_\ptKIL}$ is separated from $0$ and $\pi$,
i.e., $\regmesh \cos \alpha^*_{\!\ptdK\!\!,\ptdL}\geq 1$}.
\end{array}\right.
\end{equation}

Also a uniform bound on the 
number of neighbours of  volumes / diamonds is useful:
\begin{equation}\label{eq:neighbours-bound}
\left|\begin{array}{l}
\text{Each primal volume $\K$ has at most $\regmesh$ neighbour primal volumes};\\
\text{each dual volume $\dK$ has at most $\regmesh$ neighbour dual volumes}.
\end{array}\right.
\end{equation}
For version $(C)$ of the scheme, we impose in addition
conditions on the third mesh $\diezMrond$;
moreover, the number of vertices of a diamond is restricted by $\regmesh$.
Recall that for versions $(A)$ and $(B)$ we
assume that all diamond has five ($=2+3$)
or six ($=2+4$) vertices, because the faces of the
primal volumes are assumed to be triangles or quadrilaterals; and the
convergence results are shown for the case of triangular primal faces.

For versions $(A)$ and $(B)$, when the number $l$ of vertices
of a face $\KIL$ exceeds three, the kernel of the linear form
used to reconstruct the discrete gradient in $\DM^\ptKIL$ is not
reduced to a constant at the vertices of $\DM^\KIL$.
This is a problem, e.g., for the discrete Poincar\'e inequality and
for the proof of discrete compactness. In general, the situation
with $l\leq 4$ vertices is not clear; for example, the discrete Poincar\'e inequality
holds on every individual mesh, but it is not an easy task to prove that
the embedding constant is uniform, even under rigid proportionality
assumptions on the meshes. The uniform Cartesian meshes
is one case with $l=4$ that can be treated
(see \cite{AndrBend}), but they are not suitable
for the application we have in mind.

In this paper, for a certain range of values of the power $r$
in \eqref{eq:h-conseq},\eqref{monotone-bis}, we use
Sobolev embedding inequalities of the discrete $H^1$
spaces into $L^q$, $q>1$; for these results to hold,
we may also require
\begin{equation}\label{eq:volume-ratios}
\left|\begin{array}{lrcl}
\text{$\forall$ primal volumes $\K$ and interfaces $\KIL$},
\mKIL\dKL \leq \regmesh\,\Vol(\K);\\
\text{$\forall$ dual volumes $\K$ and interface $\dKIdL$},
\mdKIdL\ddKdL \leq   \regmesh\,\Vol(\dK).
\end{array}\right.
\end{equation}

\subsection{Consistency of projections and discrete gradients}\label{ssec:proj-ops}

Here we gather basic consistency results for the DDFV discretisations.
Heuristically,  for a given function $\ph$  on $\Om$, the projection of
$\ph$ on a mesh $\Tau$ and subsequent application of the discrete
gradient $\grad^\ptTau$ should produce a discrete field sufficiently close
(for $\size(\Tau)$ small) to $\grad\ph$. Similarly,
for a given field $\vec \Frond$, the adequate
projection on the mesh and the application
of $\div^{\!\!\ptTau}$ to this projection should yield a discrete
function close to $\div \vec \Frond$.
In this paper, we mainly use the mean-value projections.
For scalar functions on $\Om$, two projections on $\R^\ptTau$ (which
has two components, namely the
projections on $\Mrond$ and on $\dMrond$) are used:
$$
\mathbb{P}^\ptTau:\;\ph\;\mapsto\; \left(\;\Bigl(\,\frac 1{\Vol(\K)}
\int_\ptK \ph\,\Bigr)_{\ptK\in\Mrond}   \;,\;
\Bigl(\,\frac 1{\Vol(\dK)} \int_\ptdK \ph\,\Bigr)_{\ptdK\in\dMrond}   \;\right)
=: \left(\; \mathbb{P}^\ptMrond\!\ph\;,\;
\mathbb{P}^\ptdMrond\!\!\ph \;\right),
$$
$$
\mathbb{P}_c^\ptTau:\;\ph\;\mapsto\; \left(\;\Bigl(\,\ph(x_\ptK)\,\Bigr)_{\ptK\in\Mrond}   \;,\;
\Bigl(\,\ph(x_\ptdK)\,\Bigr)_{\ptdK\in\dMrond}   \;\right)
=: \left(\; \mathbb{P}_c^\ptMrond\!\ph\;,\;
\mathbb{P}_c^\ptdMrond\!\!\ph \;\right);
$$
in case $\K$ is a degenerate volume in $\MrondGN$, $\Vol(\K)$ is zero and we replace the
 corresponding entry of $\mathbb{P}^{\ptTau}\ph$ by the mean 
 value $\meanint_\ptK \ph$ of $\ph$ over the face $\K\subset\Gamma_N$.
 Similarly, the Neumann data $s_{i,e}$ will be taken into account 
 through the values $\meanint_\ptK s_{i,e}$ for $\K\in \MrondGN$.

Further, if we are interested in the  values of  $\ph$ on the Dirichlet part of the boundary,
 then we use the projection
 \begin{align*}
 \mathbb{P}^{\ptl\ptTau}:\;\ph\;\mapsto\;
  \left(\;\Bigl(\,\dspmeanint_\ptK \ph\,\Bigr)_{\ptK\in\ptl\Mrond}   \;,\;
 \Bigl(\,\dspmeanint_{\overline{\ptdK}\cap\Gamma_D} \ph\,\Bigr)_{\ptdK\in\ptl\dMrond}   \;\right)
 =: \left(\; \mathbb{P}^{\ptl\ptMrond}\!\ph\;,\;
\mathbb{P}^{\ptl\ptdMrond}\!\!\ph \;\right).
 \end{align*}
In particular, the Dirichlet data $g_{i,e}$ will be taken into account in this way. 
For $\R^3$-valued fields on $\Om$, we use the projection on $(\R^3)^\ptDrond$
defined by
$$
\vec{\mathbb{P}}^\ptTau:\;\vec \Frond\; \mapsto\;
\left(\;\frac 1{\Vol(\DM)}
\int_\ptD \vec\Frond \;\right)_{\ptD\in\ptDrond}.
$$

\medskip
With each of these discrete functions, we associate 
piecewise constant functions of $x$ on $\Om$, on $\Gamma_D$ or on $\Gamma_N$,
according to the sense of the projection; then we can 
study convergence, e.g., of $\vec{\mathbb{P}}^\ptTau \vec\Frond$ 
to $\vec\Frond$ in Lebesgue spaces, as $\size(\Tau)\to 0$.
For the data $v_0$,$\Iap$, $\bM_{i,e}$, $g_{i,e}$, $s_{i,e}$, we need 
the consistency of the associated projection operators
(recall that $v_0$,$\Iap$ are projected on the meshes $\MrondOm$ and $\dMrond$, $\bM_{i,e}$ are projected on the diamonds, $g_{i,e}$ are projected
on the boundary volumes, and $s_{i,e}$ are projected on 
the degenerate interior primal volumes $\K\in\MrondGN$).
These consistency results can be shown in a straightforward 
way (see, e.g., \cite{ABH``double''}); for example, we have
$\mathbb{P}^\ptTau\,\Iap \longrightarrow \Iap$ in $L^2(\Om)$, and $\mathbb{P}^{\ptl\ptTau}\,g_{i,e} \longrightarrow g_{i,e}$ in $L^2(\Gamma_D)$.

\medskip
Note that for the study of weak compactness in Sobolev spaces
and convergence of discrete solutions, the consistency
results can be formulated for test functions only
(and the consistency for $\div^{\!\!\ptTau}\!\!\circ\! \vec{\mathbb{P}}^\ptTau$ is
formulated in a weak form, except on very symmetric meshes).
These results are shown under the regularity restrictions
\eqref{eq:mesh-regularity},\eqref{eq:inclination-bound},\eqref{eq:neighbours-bound} on the mesh;
let us give the precise statements.

\begin{prop}\label{PropConsistency}
Let $\Tau$ be a 3D DDFV mesh of $\Om$
as described in Section~\ref{sec:DDFV}.
Let $\regmesh$ measure the mesh regularity in the
sense of \eqref{eq:mesh-regularity},\eqref{eq:inclination-bound},\eqref{eq:neighbours-bound}.
Then the following results hold: \\[5pt]
(i)~~ For all $\ph\in \mathcal D(\overline{\Om})$,
$$
\Bigl\|\ph-\mathbb{P}^\ptMrond\!\ph\Bigr\|_{L^\infty(\Om)}\leq C(\ph)\,\size(\Tau), \quad
\Bigl\|\ph-\mathbb{P}^\ptdMrond\!\!\ph\Bigr\|_{L^\infty(\Om)}\leq C(\ph)\,\size(\Tau);
$$
and for version $(C)$, the analogous estimates hold for
$\Bigl\|\ph-\mathbb{P}^\ptdiezMrond\!\ph\Bigr\|_{L^\infty(\Om)}$.\\
Analogous estimates hold for the projections $\mathbb{P}^\ptMrond_c,\mathbb{P}^\ptdMrond_c$.

Similarly, for all  $\vec\Frond\in  \Bigl(\mathcal D(\overline{\Om})\Bigr)^3$,
$$
\Bigl\|\vec\Frond-\vec{\mathbb{P}}^\ptTau\!\!\vec
\Frond\Bigr\|_{L^\infty(\Om)}\leq C(\vec\Frond)\,\size(\Tau). $$
(ii)~~  For all $\ph\in \mathcal D(\Om\cup\Gamma_N)$,
$$
\Bigl\|\grad\ph-\grad^\ptTau_0
(\mathbb{P}_c^\ptTau\!\ph)\Bigr\|_{L^\infty(\Om)}
\leq C(\ph,\regmesh)\,\size(\Tau). $$
(iii)~~ For versions $(A)$ and $(B)$, assume that each
primal interface $\KIL$ is a triangle. For each
$\vec\Frond\in  \Bigl(\mathcal D(\overline{\Om})\Bigr)^3$ and
for all $w^\ptTau\in \R^\ptTau_0$,
$$
\biggl|\Bleft \mathbb{P}^\ptTau\!\Bigl(\div\vec\Frond\Bigr)
-\div^{\!\!\ptTau} (\vec{\mathbb{P}}^\ptTau\!\vec\Frond)
\;,\;w^\ptTau\Bright_{\Om}\biggr|\leq C(\vec\Frond,\regmesh)
\,\size(\Tau)\,\|\grad^\ptTau w^\ptTau\|_{L^1(\Om)}. $$
\end{prop}
We refer to \cite{AndrBend} for a proof of this result.

\subsection{Discrete Poincar\'e, Sobolev inequalities and
strong compactness}\label{ssec:PoincareSobolev}
~\\[3pt]
The key fact here is the following remark:
$$
\begin{array}{l}
\text{Assuming (for versions $(A)$, $(B)$) that each face $\KIL$ of
$\Mrond$ is a triangle},\\  \text{one gets the same
embedding results on the 3D DDFV meshes $(A),(B),(C)$}\\
\text{as the results known for the two-point
discrete gradients on $\Mrond$ and on $\dMrond$.}
\end{array}
$$
Indeed, for variants $(A),(B)$ it has been
already observed in the proof of Proposition~\ref{PropConsistency}(iii)
that the restriction $l=3$ on the number $l$ of dual 
vertices of a diamond $\DM^{\ptKdotIKplus}$
allows for a control by $|\grad_\ptD w^\ptTau|$
of the divided differences:
\begin{equation}\label{eq:divided-differences}
\frac{|\wplus-\wdot|}{\ddotplus} \leq
\Bigl| \grad_\ptD w^\ptTau\Bigr|,\qquad
\frac{|w^*_{\!i\hspace*{-1pt}+\!1}\!\!-\!w^*_i|}{\diiun} \leq
\Bigl| \grad_\ptD w^\ptTau\Bigr|
\end{equation}
(here $i=1,2,3$ and  by our convention, $w^*_4:=w^*_1$,
$d_{3,4}:=d_{1,3}$; cf.~Figure~\ref{Fig-3D}).
For version $(C)$, this kind of control is always true
for the divided differences along the edges of any of the three meshes.
Consequently, for a proof of the different embeddings,
we can treat the primal and the dual meshes in $\Tau$
separately, as if our scheme was one
with the two-point gradient reconstruction.

\medskip

First we give discrete DDFV versions of the embeddings
of the discrete $W^{1,p}_0(\Om)$ spaces,
where we refer to the embedding into
$L^p(\Om)$ (the Poincar\'e inequality), into $L^{p^*}(\Om)$ with
$p^*:=\frac{3p}{3-p}$, $p<3$ (the critical Sobolev embedding), as
well as the compact embeddings into $L^q(\Om)$ for all $q<+\infty$   or $q<p^*$.

\begin{prop}\label{prop:Poincare-and-Sobolev} 
Let $\Tau$ be a 3D DDFV mesh of $\Om$ as described in Section~\ref{sec:DDFV}. 
Let $\regmesh$   measure the mesh regularity in the 
sense \eqref{eq:inclination-bound} and \eqref{eq:volume-ratios}.
Assume (for versions $(A)$ and $(B)$) that 
each primal interface $\KIL$ is a triangle.

Let $w^\ptTau\in\R^\ptTau_0$. Then
$$
\|w^\ptMrond\|_{L^2(\Om)},\,\|w^\ptdMrond\|_{L^2(\Om)}\;
\leq\;C(\Om,\regmesh)\;\|\grad^\ptTau w^\ptTau\|_{L^2(\Om)}.
$$
Moreover,
$$
\|w^\ptMrond\|_{L^{6}(\Om)},\,\|w^\ptdMrond\|_{L^{6}(\Om)}\;\leq\;C(\Om,\regmesh)\;
\|\grad^\ptTau w^\ptTau\|_{L^2(\Om)}.
$$
\end{prop}

Notice that for the Poincar\'e inequality (the first statement), assumption
\eqref{eq:volume-ratios}
is not needed, cf.~\cite{AGW} for a proof.
Actually, with the hint of \cite[Lemma 2.6]{AGW}
the Sobolev embeddings for $q\leq 2\times1^*=3$ can be
obtained without using \eqref{eq:volume-ratios}.

The statement  follows in a very direct way from the proofs given in
 \cite{EyGaHe:book,CoudiereGH,EGH:Sushi-IMAJNA}.
Because of \eqref{eq:divided-differences}, the assumption
that the primal mesh faces are triangles (i.e., $l=3$)
is a key assumption for the proof.
In some of the proofs in these papers one refers
to admissibility assumptions on the mesh
(such as the mesh orthogonality and  assumptions of
the kind ``$|\xK-\xL|\leq \regmesh |\xK-\xKIL|$'',
see \cite{EyGaHe:book,EGH:Sushi-IMAJNA}), yet, as
in \cite{ABH``double''} (where the proof of the
Poincar\'e inequality is given for the $2D$ case),
these assumptions are easily replaced by the bounds
\begin{equation}\label{eq:volumes-compared}
\begin{split}
\mKIL\dKL &\leq C(\regmesh)\;
\min\Bigl\{\Vol(\ptD^\ptKIL),\Vol(\K),\Vol(\L)\Bigr\},\\
\mdKIdL\ddKdL &\leq  C(\regmesh)\;
\min\Bigl\{\Vol(\ptD^\ptKIL_{\ptdKIdL}),\Vol(\dK),\Vol(\dL)\Bigr\}
\end{split}
\end{equation}
that stem from the mesh regularity
assumptions  \eqref{eq:volume-ratios}
and \eqref{eq:inclination-bound}.

The embeddings of the discrete
$W^{1,p}(\Om)$ space contain an additional
term in the right-hand side, which is usually taken to be either
the mean value of $w^\ptTau$ on some fixed
part $\Gamma$ of the boundary $\ptl\Om$
(used when a non-homogeneous Dirichlet boundary
condition on $\Gamma$ is imposed),
or the mean value of   $w^\ptTau$ on some subdomain
$\omega$ of $\Om$ (the simplest choice
is $\omega=\Om$, used for the pure Neumann boundary conditions).
Let us point out that the strategy of Eymard, Gallou\"et and Herbin in
\cite{EGH:Sushi-IMAJNA} actually allows to obtain Sobolev
embeddings for the ``Neumann case'' as soon as the Poincar\'e inequality
is obtained. For the proof, one bootstraps the estimate
of $\int_{\Om} |w^\ptTau|^\alpha$. First obtained
from the Poincar\'e inequality with $\alpha=2$, it is extended to
$\alpha=2\cdot 1^*=2\frac 32 =3$ with the discrete
variant \cite[Lemma 5.2]{EGH:Sushi-IMAJNA} (where
one can exploit \eqref{eq:volumes-compared})
of the Nirenberg technique. In the same way, the bound of $\int_{\Om} |w^\ptTau|^\alpha$
is further extended to $\alpha=2(1^*)^2=2\Bigl( \frac 32\Bigr)^2$ and so on, until one
reaches the critical exponent $2^*=6$. The details are given in
\cite{ABR-Cross-VF}. Moreover, the Poincar\'e inequality
for the ``Neumann case'' (i.e., the embedding into $L^2(\Om)$ of
the discrete analogue of the space $\Bigl\{ u\in H^1(\Om)\;\Bigl|\;\int_\Om u=0\Bigr\}$) 
and for the case with control by the mean value on a part of the boundary, was
shown in \cite{EGH:Sushi-IMAJNA}, \cite{GlitzkyGriepentrog}.
Thus we can assume that the  analogue of Proposition~\ref{prop:Poincare-and-Sobolev}
with the additional terms $\Bigl|\frac 1{\Vol(\Om)}\int_\Om w^\ptMrond\Bigr|$,
$\Bigl|\frac 1{\Vol(\Om)}\int_\Om w^\ptdMrond\Bigr|$
or $\Bigl|\frac 1{|\Gamma_D|}\int_{\Gamma_D} w^\ptMrond\Bigr|$,$\Bigl|\frac 1{|\Gamma_D|}\int_{\Gamma_D} w^\ptdMrond\Bigr|$  in the right-hand side
of the estimates is justified.

Notice that the compactness of the sub-critical embeddings is easy to obtain by interpolation
of the $L^6$ embedding with the compact $L^1$ embedding
derived from the Helly theorem (indeed, the $L^1$ estimate 
of $\grad^\ptTau w^\ptTau$ can be seen as the $BV$ estimate of
the piecewise constant functions $w^\ptMrond$ and $w^\ptdMrond$).

Finally, notice that the same arguments that yield the
Poincar\'e inequality with a homogeneous boundary condition
also yield the trace inequality
\begin{equation}\label{eq:trace-ineq}
 \Bigl\|w^{\ptl\ptMrond}\Bigr\|_{L^2(\Gamma_N)}
 \leq  C(\Gamma_N,\Om,\regmesh) \biggl(\Bigl\|w^\ptMrond\Bigr\|_{L^2(\Om)}+
 \Bigl\|\grad^\ptTau w^\ptTau\Bigr\|_{L^2(\Om)} \biggr)
\end{equation}
(the inequalities on the mesh $\dMrond$ and, for the
case $(C)$, on the mesh $\diezMrond$ are completely analogous).
These inequalities are useful for treating
non-homogeneous Neumann boundary conditions
on a part $\Gamma_N$ of $\ptl\Om$.

\subsection{Discrete $W^{1,p}(\Om)$ weak compactness}\label{ssec:W1p}

In relation with Proposition~\ref{prop:Poincare-and-Sobolev}(ii), let
us stress that there is no reason that the components
$w^{\ptMrond_{\!h}}$, $w^{\ptdMrond_{\!h}}$ of a sequence
$\Bigl(w^{\ptTau_h})_h$ of discrete functions with $L^p$ bounded
discrete gradients converge to the same limit.
Counterexamples are constructed starting from two distinct smooth functions
discretised, one on the primal mesh $\overline{\Mrond}$,
the other on the dual mesh $\overline{\dMrond}$.

However, the proposition below shows that in our 3D DDFV framework, we
can assume that the ``true limit'' of the discrete
functions $w^{\ptTau_h}=\Bigl(w^{\ptMrond_{\!h}}$\!,
$w^{\ptdMrond_{\!\!h}} \Bigr)$  or $w^{\ptTau_h}=\Bigl(w^{\ptMrond_{\!h}}\!,
w^{\ptdMrond_{\!\!h}}\!,w^{\ptdiezMrond_{\!\!h}} \Bigr)$
coincides with the limit of \eqref{eq:DiscreteSol}.

\begin{prop}\label{prop:AsymptCompactness}~\\
(i)
Let $w^{\ptTau_h}\in\R^{\ptTau_h}$ be discrete functions on a
family $(\Tau_h)_h$ of 3D DDFV meshes of $\Om$ as
described in Section~\ref{sec:DDFV}, parametrised by $h\geq \size(\Tau_h)$.
Assume $\Gamma_D\neq\text{\O}$
and let $g^\ptTau=\mathbb{P}^\ptTau g$,
for some fixed boundary datum $g\in H^1(\Om)$.
Assume that the family $\Bigl( \grad^{\ptTau_h}_{g^{\ptTau_h}}
w^{\ptTau_h}\Bigr)_{h\in(0,h_{max}]}$ is bounded in ${L^2(\Om)}$.

Assume (for versions $(A)$ and $(B)$) that each primal interface $\KIL$ is a triangle.
Assume that $\sup_{h\in(0,h_{max}]} {{\rm reg}(\Tau_h)}<+\infty$,
where ${\rm reg}(\Tau_h)$ measures the regularity of $\Tau_h$ in the sense
\eqref{eq:mesh-regularity},\eqref{eq:inclination-bound},\eqref{eq:neighbours-bound}
and \eqref{eq:volume-ratios}.

According to the type of 3D DDFV meshing considered, let us
assimilate $w^{\ptTau_h}$ into the piecewise
constant functions $w^{\ptTau_h}(\cdot)$ defined
by \eqref{eq:DiscreteSol}; furthermore, let us assimilate the
discrete gradient $\grad^{\ptTau_h}_{g^{\ptTau_h}} w^{\ptTau_h}$
to the function $\Bigl(\grad^{\ptTau_h}_{g^{\ptTau_h}} w^{\ptTau_h} \Bigr)(\cdot)$ on $\Om$.
Then for any sequence $(h_i)_i$ converging to zero there exists
$w\in g+V$ such that, along a sub-sequence,
\begin{equation}\label{eq:weak-comp-Sobolev}
\left|
\begin{array}{l}
\text{$w^{\ptTau_{h_i}}(\cdot)$ converges to $w$ strongly in
$L^2(\Om)$ (in fact, in $L^q(\Om)$, $q<6$)}\\
\text{and  $\Bigl(\grad^{\ptTau_{h_i}}w^{\ptTau_{h_i}}\Bigr)(\cdot)$
converges to $\grad w$ weakly in $L^2(\Om)$}.
\end{array}\right.
\end{equation}
(ii) If $\Gamma_D=\text{\O}$, and if the
additional assumption of uniform boundedness of
$$
m_{w^{\ptMrond_{\!h}}}:=\frac 1{\Vol(\Om)} \int_\Om w^{\ptMrond_{\!h}}, \qquad
m_{w^{\ptdMrond_{\!\!h}}}:=\frac 1{\Vol(\Om)} \int_\Om w^{\ptdMrond_{\!\!h}}
$$
is imposed (with the analogous bound on the mesh $\diezMrond$ for
version $(C)$), then \eqref{eq:weak-comp-Sobolev}
holds with $w\in H^1(\Om)$.
\end{prop}

Let us illustrate the DDFV techniques by giving the
ideas of the proof. We justify (i)  for the
case of the meshing described in Section~\ref{sec:DDFV}
and the homogeneous Dirichlet boundary condition
on $\Gamma_D:=\ptl\Om$. The case of a non-homogeneous
Dirichlet condition is thoroughly treated in \cite{ABH``double''},
for the 2D DDFV schemes. The case of Neumann
boundary conditions is the simplest one.

\begin{proof}[Proof of Proposition~\ref{prop:AsymptCompactness}]
The strong compactness claim follows by the compactness
of the subcritical Sobolev embeddings of
Proposition~\ref{prop:Poincare-and-Sobolev}.
The weak $L^2$ compactness of
the family $\Bigl( \grad^{\ptTau_h} w^{\ptTau_h}\Bigr)_{h}$
is immediate from its $L^2(\Om)$ boundedness.
Thus if $w$ is the strong $L^2$ limit of a sequence $  w^{\ptTau_h}=
\frac 13 w^{\ptMrond_h}+ \frac 23 w^{\ptdMrond_{\!\!h}}$ as $h\to 0$ and
$\chi$ is the weak $L^2$ limit of the associated sequence of
discrete gradients $\grad^{\ptTau_h} w^{\ptTau_h}$, it only remains to show that
$\chi=\grad w$ in the sense of distributions and that $w$ has zero trace on $\ptl\Om$.
These two statements follow from the identity
\begin{equation}\label{eq:id-for-identification}
\forall\, \vec\Frond \in \mathcal D(\overline{\Om})^3 \;\;\;
\int_\Om \chi\cdot \vec\Frond + \int_\Om w\,\div\vec\Frond=0,
\end{equation}
that we now prove. We exploit the discrete duality and the consistency property
of Proposition~\eqref{PropConsistency}(i),(iii).

Take the projection $\vec{\mathbb{P}}^{\ptTau_h}\!
\vec\Frond\in (\R^3)^{\ptDrond_h}$, $w^{\ptTau_h}\in \R_0^{\ptTau_h}$
and write the discrete duality formula
\begin{equation}\label{eq:test-fct-compactness}
\Aleft\grad^{\ptTau_h} w^{\ptTau_h}\, ,
\, \vec{\mathbb{P}}^{\ptTau_h}\! \vec\Frond
\Aright_{\Om} \;+\;\Bleft
w^{\ptTau_h}\, ,\, \div^{\!\!\ptTau_h} \vec{\mathbb{P}}^{\ptTau_h}\! \vec\Frond
\Bright_{\Om}=0.
\end{equation}
According to the definition of $\Aleft \,\cdot\, ,\, \cdot\,\Aright_{\Om} $,
the first term in \eqref{eq:test-fct-compactness} is precisely the integral over $\Om$
 of the scalar product of the constant per diamond fields
 $\grad^{\ptTau_h} w^{\ptTau_h}$ and $\vec{\mathbb{P}}^{\ptTau_h}\! \vec\Frond$.
 By  Proposition~\eqref{PropConsistency}(i)
and the definition of $\chi$, this term converges to the first
term in \eqref{eq:id-for-identification} as $h\to 0$.
Similarly, introducing the projection
$\mathbb{P}^{\ptTau_h}\!\Bigl(\div\vec\Frond\Bigr)$
of $\div\vec\Frond$ on $\R^{\ptTau_h}$, from the
definition of $\Bleft \,\cdot\, ,\, \cdot\,\Bright_{\Om}$,
Proposition~\eqref{PropConsistency}(i)
and the definition of $w^{\ptTau_h}$ in \eqref{eq:DiscreteSol},
we see that, as $h\to 0$,
\begin{align*}
\;\Bleft
w^{\ptTau_h}\, ,\,\mathbb{P}^\ptTau\!\Bigl(\div\vec\Frond\Bigr)\
\Bright_{\Om} \longrightarrow & \frac 13 \int_{\Om} \Bigl(\lim_{h\to 0}
w^{\ptMrond_{\!h}}\Bigr)\,\div\vec\Frond
\\ & + \frac 23 \int_{\Om} \Bigl(\lim_{h\to 0} w^{\ptdMrond_{\!\!h}}\Bigr)
\,\div\vec\Frond=\int_\Om w\,\div\vec\Frond.
\end{align*}
It remains to invoke Proposition~\eqref{PropConsistency}(iii) and
the $L^1(\Om)$ bound on $\grad^{\ptTau_h} w^{\ptTau_h}$ to justify the fact that
$$
\lim_{h\to 0} \Bleft
w^{\ptTau_h}\, ,\, \div^{\!\!\ptTau_h} \vec{\mathbb{P}}^{\ptTau_h}\! \vec\Frond
\Bright_{\Om}= \lim_{h\to 0} \Bleft
w^{\ptTau_h}\, ,\, \mathbb{P}^\ptTau\!\Bigl(\div\vec\Frond\Bigr)
\Bright_{\Om}.
$$
For a proof of (ii) use the versions of the compact
Sobolev embeddings with control by the mean value in $\Om$,
and use test functions $\vec\Frond$
compactly supported in $\Om$.
\end{proof}

\subsection{Discrete operators, functions and
fields on $(0,T)\times\Om$}\label{ssec:time-dep}

We discretise our evolution equations in space using the
DDFV operators as described above. In this
time-dependent framework, analogous consistency
properties, Poincar\'e inequality and discrete
$L^p(0,T;W^{1,p}(\Om))$ compactness properties hold.

To be specific, given a DDFV mesh $\Tau$ of $\Om$ and a time step $\delt$,
one considers
the additional projection operator
$$
\mathbb{S}^\ptdelt:\; f\;\mapsto   \Bigl(f^{n} \Bigr)_{n\in [1,\Ndelt]}\subset L^1(\Om),
\qquad f^{n}(x):=\frac 1{\ptdelt}\int_{(n\hspace*{-1pt}-\!1)\ptdelt}^{n\ptdelt} f(t,x)\, dt.
$$
Here $f$ can mean a function in $L^1((0,T)\times\Om)$ or a field in
$\Bigl(L^1((0,T)\times\Om)\Bigr)^3$. The greatest integer
smaller than or equal to $T/\ptdelt$ is denoted by $\Ndelt$.

We define discrete functions $w^{\ptTau,\ptdelt}\in (\R^\ptTau)^\Ndelt$
on $(0,T)\times \Om$ as collections of discrete functions
$w^{\ptTau,n+1}$ on $\Om$ parametrised by $n\in [0,\Ndelt]\cap \N$.
Discrete functions $w^{\ptTau,\ptdelt}\in (\R^\ptTau)^\Ndelt$
on $(0,T)\times \overline{\Om}$ and discrete fields
$\vec\Frond^{\ptTau,\ptdelt}\in(\R^\ptDrond)^\Ndelt$ are defined similarly.
The associated norms are defined in a natural way; e.g., the
discrete $L^2(0,T;H^{1}_0(\Om))$ norm of a
discrete function $w^{\ptTau,n}\in
(\R^\ptTau_0)^\Ndelt$ is computed as
$$
\left( \sum\nolimits_{n=0}^{\Ndelt} \delt \, 
\|\grad^\ptTau w^{\ptTau,n+1}\|_{L^2(\Om)}^2\right)^{\frac12}.
$$
To treat space-time dependent test functions and fields as in
Proposition~\ref{PropConsistency},
one replaces the projection operators $\mathbb{P}^\ptTau$ (and its components
$\mathbb{P}^\ptMrond$,$\mathbb{P}^\ptdMrond$),
$\mathbb{P}^{\ptl\ptTau}$ and $\vec{\mathbb{P}}^\ptTau$
by their compositions with $\mathbb{S}^\ptdelt$. Then the statements and proof
of Proposition~\ref{PropConsistency} can be
extended in a straightforward way.

Also the statements of Proposition~\ref{prop:AsymptCompactness}
extend naturally to the time-dependent context; one only has to replace the statement
\eqref{eq:weak-comp-Sobolev} with the weak
$L^2(0,T;H^1(\Om))$ convergence statement:
\begin{equation*}
\left|
\begin{array}{l}
\text{$w^{\ptTau_{h},\delt_h}$ converges to
$w$ weakly in $L^2((0,T)\times\Om)$ and in $L^2(0,T;L^6(\Om))$;}\\
 \text{$\grad^{\ptTau_{h}}w^{\ptTau_{h},\delt_h}$ converges to $\grad w$
weakly in $L^2(\Om)$},
\end{array}\right.
\end{equation*}
as $\size(\Tau_h)+\delt_h\to 0$. It is natural that strong compactness on
the space-time cylinder $(0,T)\times \Om$ does not follow from a
discrete spatial gradient bound alone; one also
needs some control of time oscillations. It is also well known that
this control can be a very weak one (cf., e.g., the
well-known Aubin-Lions and Simon lemmas).  In the next section, we
give the discrete version of one of these results.

\subsection{Strong compactness in $L^1((0,T)\times \Om)$}\label{ssec:compL^1Q}

Below we state a result that fuses a basic space
translates estimate (for the ``compactness in space'')
with the Kruzhkov $L^1$ time compactness lemma (see \cite{Kruzhkov-zametki}). 
Actually, the Kruzhkov lemma is, by essence, a local compactness result.
For the sake of simplicity, we state the version suitable for discrete functions that
are zero on the boundary; the corresponding $L^1_{loc}([0,T]\times \Om)$
version can be shown with the same arguments
(cf. \cite{ABR-Cross-VF}), and this local version can be 
used for all boundary conditions.

\begin{prop}\label{prop:KruzhkovLemma}
Let $\Bigl(u^{\ptTau_h,\delt_h}\Bigr)_h\in (\R^{\ptTau_h}_0)^{\Ndelt_h}$
be a family of discrete functions on the cylinder $(0,T)\times \Om$ corresponding to
a family $(\delt_h)_h$ of time steps and to
a family $(\Tau_h)_h$ of 3D DDFV meshes of $\Om$ as described
in Section~\ref{sec:DDFV}; we understand that $h\geq \size(\Tau_h)+\delt_h$.
Assume that $\sup_{h\in(0,h_{max}]} {{\rm reg}(\Tau_h)}<+\infty$, where
${\rm reg}(\Tau_h)$ measures the regularity of $\Tau_h$ in the sense
\eqref{eq:mesh-regularity} and \eqref{eq:inclination-bound}.

Assume (for versions $(A)$ and $(B)$) that all primal
interfaces $\KIL$ for all meshes $\Tau_h$ are triangles.
For each $h>0$, assume that the discrete functions $v^{\ptTau_h,\delt_h}$
satisfy the discrete evolution equations
\begin{equation*}
    \text{for $n\in[0,N_h]$},\quad
    \frac{v^{\ptTau_h,{n+1}}-v^{\ptTau_h,n}}{\delt}=\div^{\!\ptTau_h}\,
    \vec \Frond^{\ptTau_h,n+1}+f^{\ptTau_h,n+1}
\end{equation*}
with some initial data $ v^{\ptTau_h,0} \in\R^{\ptTau_h} $,
source terms $ f^{\ptTau_h,\delt_h}  \in (\R^{\ptTau_h})^{\Ndelt_h}$
and discrete fields $\vec \Frond^{\ptTau_h,\delt_h}
\in ((\R^3)^{\ptDrond_h})^{\Ndelt_h} $.

Assume that there is a constant $M$ such that
the following uniform $L^1((0,T)\times\Om)$ estimates hold:
\begin{multline*}
\sum\nolimits_{n=0}^{N_h}
\delt\, \biggl(\Bigl\|\,v^{\ptMrond_{\!h},n+1}\,\Bigr\|_{L^1(\Om)}\!
+\Bigl\|\,v^{\ptdMrond_{\!h},n+1}\,\Bigr\|_{L^1(\Om)}\\
 + \Bigl\|\,f^{\ptMrond_{\!h},n+1}\,\Bigr\|_{L^1(\Om)}\!
 + \Bigl\|\,f^{\ptdMrond_{\!h},n+1}\,\Bigr\|_{L^1(\Om)}\!+
  \,\Bigl\|\,\vec \Frond^{\ptTau_h,n+1}\,
  \Bigr\|_{L^1(\Om)}\biggr)\leq M,
\end{multline*}
and
\\[-10pt]
\begin{equation*}
\sum\nolimits_{n=0}^{N_h}
\delt \, \Bigl\|\,\grad^{\ptTau_h}
u^{\ptTau_h,n+1}\,\Bigr\|_{L^1(\Om)} \leq M.
\end{equation*}
Assume that the families $\Bigl(\,b(u^{\ptMrond_h,0})\,\Bigr)_h$,
$\Bigl(\,b(u^{\ptdMrond_{\!h},0})\,\Bigr)_h$ are bounded in $L^1(\Om)$.

Then for any sequence $(h_i)_i$ converging to zero there
exist $\beta^o,\beta^* \in  L^1((0,T)\times\Om)$ such that, extracting
if necessary a sub-sequence,
$$
\text{$b(u^{\ptMrond_{h_i},\delt_{h_i}})\longrightarrow
\beta^o$, \quad
$b(u^{\ptdMrond_{h_i},\delt_{h_i}})
\longrightarrow \beta^*$, \quad
in $L^1((0,T)\times\Om)$ as $i\to\infty$}.
$$
\end{prop}

Notice that we only use the full strength of
Proposition~\ref{prop:KruzhkovLemma} to treat the
linearised implicit scheme. For the fully implicit scheme, more
traditional (although not much simpler) $L^2$ versions of
time translation estimates, inspired by the
technique of \cite{AL}, can be used (see \cite{EyGaHe:book}).

\subsection{Discretisation of the ionic current term}\label{ssec:reaction-term}

Consider the general situation where $w$ is discretised on 
a DDFV mesh. Moreover, assume that we also need to discretise
some scalar function $\psi(w)$ (in our context, this is the ionic 
current term $h$; in general reaction-diffusion 
systems, $\psi$ may represent reaction terms).

Then we discretise such reaction term on a 3D DDFV 
mesh of the kind $(B)$ by taking, for $\Psi=\psi(w)$,
\begin{equation}\label{eq:h-ptTau}
\begin{split}
    \Psi^\ptTau & := \biggl(\,\Bigl( \psi(\check w_\ptK) \Bigr)_{\ptK\in\ptMrond}\,,\,
    \Bigl( \psi(\check w_\ptdK)  \Bigr)_{\ptdK\in\ptdMrond}\,\biggr),\\
    \check w_\ptK & :=\frac1{3} w_\ptK
    +\frac 23 \sum_{\ptdK\in\overline{\ptdMrond}}
    \frac{\Vol(\ptK\cap\ptdK)}{\Vol(\ptK)} w_\ptdK, \\
    \check w_\ptdK & := \frac1{3} \sum_{\ptK\in \overline{\ptdMrond}}
    \frac{\Vol(\ptK\cap\ptdK)}{\Vol(\ptdK)} w_\ptK+ \frac 23  w_\ptdK
\end{split}
\end{equation}
In other words,
$$
\begin{array}{l}
\text{$\check w_\ptK$ and $\check w_\ptdK$ are the mean values}\\
\text{of the function
$w^\ptTau(\cdot):=\frac 13 w^{{\ptMrond}}(\cdot)
+\frac 23 w^{{\ptdMrond}}(\cdot)$ on $\K$ and on $\dK$, respectively.}
\end{array}
$$

With this choice, we have for all $w^\ptTau\in\R^\ptTau_0$ and
for all $\ph^\ptTau\in \R^\ptTau$,
\begin{equation*}\label{eq:reaction-summbyparts-0}
\Bleft (\psi(w))^\ptTau , \ph^\ptTau\Bright_\Om =
\int_\Om \psi\left(\frac 13 w^{{\ptMrond}}+\frac 23 w^{{\ptdMrond}}\right)
\;\left(\frac 13 \ph^{{\ptMrond}}+\frac 23 \ph^{{\ptdMrond}}\right)
=\int_\Om
\psi\Bigl(w^\ptTau(\cdot)\Bigr)\;\ph^\ptTau(\cdot).
\end{equation*}
For schemes $(A)$ and $(C)$, we use similar projection
formulas with the expression of
$w^\ptTau(\cdot)$ given by \eqref{eq:DiscreteSol}; this
always leads to the formula
\begin{equation}\label{eq:reaction-summbyparts}
\dsp \Bleft (\psi(w))^\ptTau , \ph^\ptTau\Bright_\Om  =
\int_\Om
\psi\Bigl(w^\ptTau(\cdot)\Bigr)\;\ph^\ptTau(\cdot).
\end{equation}
By the definition \eqref{eq:DiscreteSol} of $w^\ptTau(\cdot)$,
the definition of $\Bleft\,\cdot\,,\,\cdot\,\Bright_\Om$, and
Jensen's inequality, we get
\begin{equation}\label{eq:Bsquare>intsquare}
\forall w^\ptTau\in\R^\ptTau \quad \int_\Om \Bigl| w^\ptTau(\cdot)\Bigr|^2
\leq \Bleft w^\ptTau , w^\ptTau\Bright_\Om.
\end{equation}

Finally, notice that such choice of discretisation of the ionic current term
does not enlarge the stencil of the DDFV scheme
used for the discretisation of the diffusion.

\section{The convergence proofs}\label{sec:ConvergenceProof}

\subsection{Convergence of the fully implicit scheme}

The proof follows closely the existence proof for
Problem \eqref{S1-p},\eqref{S2},\eqref{S3} mentioned
in Section~\ref{sec:Theory-WellPosedness}.

\underline{Step 1} (proof of Proposition \ref{prop:schemeexist} -- uniqueness 
of a discrete solution).
Although Remark~\ref{rem:optimization} can be used to infer the
existence and uniqueness of a discrete solution, let us
give a proof that contains the essential calculations also
utilised in the subsequent steps. For the uniqueness and the
continuous dependence claim \eqref{eq:discr-contdep} we reason as in
Theorem~\ref{theo:L^2-contraction}, omitting
the regularisation step.
Namely,
using \eqref{all-AbstractScheme} for two solutions
$\biggl((u_i^{\ptTau,\delt},u_e^{\ptTau,\delt},v^{\ptTau,\delt})\biggr)$ and
$\biggl((\hat u_i^{\ptTau,\delt},\hat u_e^{\ptTau,\delt}, 
\hat v^{\ptTau,\delt})\biggr)$, by subtraction we get
\addtocounter{equation}{1}
$$
\leqno ({\theequation}_{i.e}) \qquad
\begin{array}{l}
\dsp \frac 1{\delt} \biggl( (v^{\ptTau,n+1}\!-\!\hat v^{\ptTau,n+1}) 
- (v^{\ptTau,n}\!-\!\hat v^{\ptTau,n})  \biggr)
\,-\, \Bigl(h^{\ptTau,n+1}\!\!-\!\hat h^{\ptTau,n+1}\Bigr)\\[7pt]
\dsp \qquad\quad -(-1)^{i,e}\,
\div^\ptTau_{\!\!s_{i,e}^{\ptTau,n+1}}[\,\bM_{i,e}^\ptTau \;
\Bigl(\grad^\ptTau_{\!g_{i,e}^{\!\ptTau,n+1}}
u_{i,e}^{\ptTau,n+1} \!\!-\! \grad^\ptTau_{\!g_{i,e}^{\!\ptTau,n+1}}
\hat u_{i,e}^{\ptTau,n+1} \Bigr)\,] \,=\,0
\end{array}
$$
with $(-1)^i\!:=\!1$, $(-1)^e\!:=\!-1$ and
$
h^{\ptTau,n+1}=\mathbb{P}^\ptTau h(v^{\ptTau,n+1}(\cdot))$, $\hat h^{\ptTau,n+1}
=\mathbb{P}^\ptTau h(\hat v^{\ptTau,n+1}(\cdot)).
$
For all $n$, we take the scalar product $\Bleft\,\cdot\,,\,\cdot\,\Bright_\Om$ of 
equations $({\theequation}_{i.e})$ with the discrete functions
$\ph^\ptTau:=\Bigl( u_e^{\ptTau,n+1} \!\!-\! \hat u_e^{\ptTau,n+1} \Bigr) $, respectively
(more precisely, we use the discrete weak 
formulations \eqref{eq:dualformulation} with test function $\ph^\ptTau$).
Then we subtract the relation obtained for $e$ from the relation obtained for $i$.
Finally, we use the discrete duality property \eqref{eq:discr-duality} on 
the divergence terms. Notice that the boundary terms
vanish, because both solutions correspond to the same 
Dirichlet and Neumann data $g_{i,e}^{\ptTau,\delt},s_{i,e}^{\ptTau,\delt}$;
in particular, we can use \eqref{eq:discr-duality} because
$$\Bigl(\grad^\ptTau_{\!g_{i,e}^{\!\ptTau,n+1}}
u_{i,e}^{\ptTau,n+1} \!\!-\! \grad^\ptTau_{\!g_{i,e}^{\!\ptTau,n+1}}
\hat u_{i,e}^{\ptTau,n+1} \Bigr)= \grad^\ptTau_{\!0} \Bigl(
u_{i,e}^{\ptTau,n+1} \!\!-\!\hat u_{i,e}^{\ptTau,n+1}\Bigr),$$ further, the terms 
coming from $\Gamma_N$ are $\Cleft s_{i,e}^{\ptTau,n+1}\!\!-\! 
s_{i,e}^{\ptTau,n+1} , u_{i,e}^{\ptTau,n+1} \!\!-\!
\hat u_{i,e}^{\ptTau,n+1} \Cright_{\Gamma_N}=0$.
The outcome of the calculation is the following equality:
\begin{multline}\label{eq:big-equality}
\frac{1}{\delt}
\Bleft \,\Bigl(v^{\ptTau,n+1}\!\!-\!\hat v^{\ptTau,n+1}\Bigr)
- \Bigl(v^{\ptTau,n}\!\!-\!\hat v^{\ptTau,n}\Bigr)\,,\,\Bigl(v^{\ptTau,n+1}\!
-\!\hat v^{\ptTau,n+1}\Bigr) \,\Bright_\Om\\
+ \Aleft \,\bM_i^{\ptTau}\Bigl(\grad^\ptTau_{\!g_i^{\!\ptTau,n+1}} u_i^{\ptTau,n+1} \!\!
-\!\grad^\ptTau_{\!g_i^{\!\ptTau,n+1}}  \hat u_i^{\ptTau,n+1}  \Bigr)\,,\,
\Bigl(\grad^\ptTau_{\!g_i^{\!\ptTau,n+1}} u_i^{\ptTau,n+1}\!\!
-\! \grad^\ptTau_{\!g_i^{\!\ptTau,n+1}} \hat u_i^{\ptTau,n+1}\Bigr)\, \Aright_\Om\\ +
\Aleft \,\bM_e^{\ptTau}\Bigl(\grad^\ptTau_{\!g_e^{\!\ptTau,n+1}} u_e^{\ptTau,n+1}\!\!
-\! \grad^\ptTau_{\!g_e^{\!\ptTau,n+1}} \hat u_e^{\ptTau,n+1}  \Bigr)\,,\,
\Bigl(\grad^\ptTau_{\!g_e^{\!\ptTau,n+1}} u_e^{\ptTau,n+1}\!\!
-\!\grad^\ptTau_{\!g_e^{\!\ptTau,n+1}}  \hat u_e^{\ptTau,n+1}\Bigr)\, \Aright_\Om \\
+ \Bleft\, \Bigl(h^{\ptTau,n+1}\!\!-\!\hat h^{\ptTau,n+1}\Bigr)\,,\,
\Bigl(v^{\ptTau,n+1}\!\!-\!\hat v^{\ptTau,n+1}\Bigr) \Bright_\Om = 0.
\end{multline}
Then we sum over $n\in[0,k]$, $k\leq N$. Using the convexity
inequality $a(a-b)\ge \frac{1}{2}(a^2-b^2)$,
the positivity of $\bM_{i,e}^\ptTau$ and the definition of
$\tilde b$, using \eqref{eq:reaction-summbyparts} we get
\begin{multline}\label{eq:tmp-1}
\frac 12 \Bleft \,\Bigl(v^{\ptTau,k+1}\!\!-\!\hat v^{\ptTau,k+1}\Bigr)
\,,\,\Bigl(v^{\ptTau,k+1}\!-\!\hat v^{\ptTau,k+1}\Bigr) \,\Bright_\Om\\
+ \sum_{n=0}^k \delt \int_\Om \biggl(\tilde h(v^{\ptTau,n+1}(\cdot))
-\tilde{h}(\hat v^{\ptTau,n+1}(\cdot))\biggr)
\biggl(v^{\ptTau,n+1}(\cdot)-\hat v^{\ptTau,n+1}(\cdot)\biggr)
\\
\leq  \frac 12 \Bleft \,\Bigl(v^{\ptTau,0}\!\!-\!\hat v^{\ptTau,0}\Bigr)
\,,\,\Bigl(v^{\ptTau,0}\!-\!\hat v^{\ptTau,0}\Bigr) \,\Bright_\Om
+  L\sum_{n=0}^k \delt
\int_\Om \Bigl|v^{\ptTau,n+1}(\cdot)-\hat v^{\ptTau,n+1}(\cdot)\Bigr|^2.
\end{multline}
Then the second term is non negative, and we get \eqref{eq:discr-contdep}
from \eqref{eq:Bsquare>intsquare} and the discrete Gronwall inequality.

In particular, it follows that for fixed initial and boundary data
there is uniqueness of $v^{\ptTau,k+1}$ for all $k$.
Then, returning to \eqref{eq:big-equality}, we find out that there is uniqueness for
$\grad^\ptTau_{\!g_{i,e}^{\!\ptTau,n+1}} u_{i,e}^{\ptTau,n+1}$ for all $n$.
We conclude the uniqueness of
$u_{i,e}^{\ptTau,n+1}$ using the discrete Poincare
inequality (and, in the case $\Gamma_D=\text{\O}$, using
condition \eqref{eq:discrete-normalization}).

\underline{ Step 2} (proof of Proposition \ref{prop:schemeexist} -- existence of a discrete solution).
Regarding the question of existence, we reason by induction in $n$.
Using the discrete weak formulation \eqref{eq:dualformulation} 
with the test function $\ph^\ptTau=u_{i,e}^\ptTau$, subtracting the equations
obtained for subscripts $i$ and $e$, we find
\begin{align*}
& \frac{1}{\delt}
\Bleft \, v^{\ptTau,n+1}
 \,,\, v^{\ptTau,n+1}  \,\Bright_\Om
 + \int_\Om \tilde h(v^{\ptTau,n+1}(\cdot))\,v^{\ptTau,n+1}(\cdot)
 \\ & \qquad
 + \Aleft \,\bM_i^{\ptTau}\grad^\ptTau_{\!g_i^{\!\ptTau,n+1}} u_i^{\ptTau,n+1}\,,\,
\grad^\ptTau_{\!g_i^{\!\ptTau,n+1}} u_i^{\ptTau,n+1}\, \Aright_\Om
\\ & \qquad\qquad
+ \Aleft \,\bM_e^{\ptTau}\grad^\ptTau_{\!g_e^{\!\ptTau,n+1}} u_e^{\ptTau,n+1}\,,\,
\grad^\ptTau_{\!g_e^{\!\ptTau,n+1}} u_e^{\ptTau,n+1}\, \Aright_\Om
\\ & = \frac{1}{\delt}
\Bleft \,   v^{\ptTau,n} \,,\, v^{\ptTau,n+1}\,\Bright_\Om+
\int_\Om  \biggl(L \Bigl|v^{\ptTau,n+1}(\cdot)|^2+l\,v^{\ptTau,n+1}(\cdot)
\biggr)
\\ & \qquad\quad
+ \Bleft\, \Iap^{\ptTau,n+1}\,,\, v^{\ptTau,n+1} \Bright_\Om
 + \Cleft s_i^{\ptTau,n+1}\,,\,u_i^{\ptTau,n+1}  \Cright_{\Gamma_N}
 + \Cleft s_e^{\ptTau,n+1}\,,\,u_e^{\ptTau,n+1}  \Cright_{\Gamma_N} .
\end{align*}
Then using the condition $\frac 1\delt>L$,
property \eqref{eq:Bsquare>intsquare}, the Cauchy-Schwarz
inequality and the equivalence of all norms on $\R^\ptTau$,
we deduce the {\it a priori} estimate
\begin{multline*}
\frac \gamma 2 \Biggl(\Aleft\grad^\ptTau_{\!g_i^{\!\ptTau,n+1}}
u_i^{\ptTau,n+1}\,,\, \grad^\ptTau_{\!g_i^{\!\ptTau,n+1}} u_i^{\ptTau,n+1}\Aright
+ \Aleft\grad^\ptTau_{\!g_e^{\!\ptTau,n+1}} u_e^{\ptTau,n+1}\,,\,
\grad^\ptTau_{\!g_e^{\!\ptTau,n+1}} u_e^{\ptTau,n+1}\Aright \Biggr)\\
+  \left(\frac 1{2\delt} -L\right)\Bleft \, v^{\ptTau,n+1}
 \,,\, v^{\ptTau,n+1}  \,\Bright_\Om \;\leq\;
 C(v^{\ptTau,n},\Iap^{\ptTau,n+1},s_{i,e}^{\ptTau,n+1},
 g_{i,e}^{\ptTau,n+1},l,\gamma,\Omega,\Gamma_D).
\end{multline*}
The left-hand side allows us to bound the discrete
solutions {\it a priori}, if $\Gamma_D\neq\text{\O}$.
The case of pure Neumann BC is slightly more delicate.

We take advantage of the above estimate to apply the
Leray-Schauder topological degree theorem. Let us look at the
most delicate case $\Gamma_D=\text{\O}$.

For $\theta\in[0,1]$, we consider the initial data $\theta v_0$, the Dirichlet
and Neumann boundary data $\theta g_{i,e}$ and $\theta s_{i,e}$,
and the source term $\theta \Iap$. We consider a
family $\mathcal F^\theta$ of maps on the space
$$
Sp:=\biggl\{ \biggl((u_i^{\ptTau,n},u_e^{\ptTau,n},v^{\ptTau,n})
\biggr)_{n=1,\ldots,N}\subset
(R^\ptTau)^3 \biggl| \text{\; \eqref{eq:discrete-normalization} holds for all $n$}\;
\biggr\}
$$
defined as follows. First, given an element in $Sp$ denoted $U^{\ptTau,n+1}$,
introduce the following notation for the expressions in 
the left-hand side of equations \eqref{all-AbstractScheme}
with data scaled by $\theta$:
\begin{equation*}\label{eq:notation-Ftheta}
\begin{array}{l}
\alpha_i^\theta( U^{\ptTau,\delt}):=\frac{v^{\ptTau,n+1} - v^{\ptTau,n}}{\Delt}
-\div^\ptTau_{\!\!\theta s_i^{\ptTau,n+1}}[\bM_{i}^\ptTau\grad^\ptTau_{\!\theta g_i^{\!\ptTau,n+1}}
u_i^{\ptTau,n+1}]+
h^{\ptTau,n+1}
- \theta \Iap^{\ptTau,n+1}, \\
\alpha_e^\theta( U^{\ptTau,\delt}):=\frac{v^{\ptTau,n+1} - v^{\ptTau,n}}{\Delt}
+\div^\ptTau_{\!\!\theta s_e^{\ptTau,n+1}}[\bM_{e}^\ptTau \grad^\ptTau_{\!\theta g_e^{\!\ptTau,n+1}}
u_e^{\ptTau,n+1}]+
h^{\ptTau,n+1}
-\theta\Iap^{\ptTau,n+1}, \\
\gamma^\theta( U^{\ptTau,\delt}):=v^{\ptTau,n}-(u_i^{\ptTau,n}-u_e^{\ptTau,n}); \\
\end{array}
\end{equation*}
recall that for the volumes $\K\in \MrondGN$, the 
convention \eqref{eq:division-by-zero-convention}
applies, so that the entry of $\alpha_{i,e}^\theta( U^{\ptTau,n+1})$
corresponding to the volumes $\K\in \MrondGN$ should be taken equal to
$(\bM_{i,e})_\ptD^{n+1}\cdot\nuK\!+(s_{i,e})^{n+1}_\ptK$, where 
the diamond $\DM$ is the one with $\K\subset \ptl\DM$. 
Then we define $\mathcal F^\theta$ as the element of
$\R^\ptTau\times\R^\ptTau\times\R^\ptTau$ given by $\Bigl(\alpha_i^\theta, \alpha_i^\theta-\alpha_e^\theta, \gamma^\theta  \Bigr)$.
This definition implies that $\mathcal F^\theta$ maps $Sp$ into itself,
thanks to the definition of the discrete divergence (which 
ensures the consistency of the fluxes)
and to the constraint \eqref{eq:Neumann-compatibility} that is 
preserved at the discrete level.

With this definition, it is evident that the zeros of $\mathcal F^\theta$ are solutions
of the scheme \eqref{all-AbstractScheme} with data scaled by $\theta$.
The  estimate that we have just deduced is uniform in $\theta$; it provides a 
uniform in $\theta$ bound on
some norm  of possible zeros of $\mathcal F^\theta$ in $Sp$.
Therefore from the Leray-Schauder theorem and the
existence of a trivial zero of $\mathcal F^0$ we infer existence
of a zero for $\mathcal F^\theta$, in particular for $\theta=1$.

\underline{Step 3} (Estimates of the discrete solution).
We make the same calculation as in Step 1, but
with $\hat u_{i,e}^{\ptTau,\delt}$ set to zero; and we sum over $n=1,\dots,k$.
Using the convexity inequality $a(a-b)\ge \frac{1}{2}(a^2-b^2)$,
the positivity of $\bM_{i,e}^\ptTau$ and the definition of $\tilde b$,
using \eqref{eq:reaction-summbyparts} we get the identity
\begin{align*}
& \frac 12 \Bleft \, v^{\ptTau,k+1}
 \,,\, v^{\ptTau,k+1}  \,\Bright_\Om
 + \int_0^{(k+1)\delt}\!\!\int_\Om
 \tilde h(v^{\ptTau,\delt}(\cdot))\,v^{\ptTau,\delt}(\cdot)
\\ & \qquad
+ \gamma \int_0^{(k+1)\delt}\!\!\int_\Om
\biggl( \Bigr|\Bigl(\grad^\ptTau_{\!g_i^{\!\ptTau,\delt}}
u_i^{\ptTau,\delt}\Bigr)(\cdot)\Bigr|^2
+\Bigr|\Bigl(\grad^\ptTau_{\!g_e^{\!\ptTau,\delt}}
u_e^{\ptTau,\delt}\Bigr)(\cdot)\Bigr|^2\biggr)
\\ &  \leq \frac 12 \Bleft \, v^{\ptTau,0}
 \,,\, v^{\ptTau,0}  \,\Bright_\Om
 +\int_0^{(k+1)\delt}\!\!\int_\Om  \biggl(L \Bigl|v^{\ptTau,\delt}(\cdot)|^2
 +l\,v^{\ptTau,\delt}(\cdot)\biggr)
 \\  & \qquad +
 \int_0^{(k+1)\delt}\!\!\int_\Om  \Iap^{\ptTau,\delt}(\cdot)\, v^{\ptTau,\delt}(\cdot)
\\ & \qquad
+ \int_0^{(k+1)\delt}\!\!\int_{\Gamma_N}
 \biggl(s_i^{\ptl\ptTau,\delt}(\cdot)\,,\,u_i^{\ptl\ptTau,\delt}(\cdot)
 +   s_e^{\ptl\ptTau,\delt}(\cdot)\,,\,u_e^{\ptl\ptTau,\delt}(\cdot)\biggr).
\end{align*}
Using the Cauchy-Schwarz inequality,  property \eqref{eq:Bsquare>intsquare},
the trace inequality \eqref{eq:trace-ineq} for each
components of  the solution, and the discrete Gronwall
inequality, we deduce the following uniform bounds:
\begin{align}
\label{est:Linty-L2}
 \|v^{\ptTau,\delt}(\cdot)\|_{L^\infty(0,T;L^2(\Om))} \leq C;\\
\label{est:Lr}
 \|v^{\ptTau,\delt}(\cdot)\|_{L^r(Q)} \leq C;\\
\label{est:L2-0T-H1}
 \|u_{i,e}^{\ptTau,\delt}(\cdot)\|_{L^2(Q)} +
 \Bigl\|\Bigl(\grad^\ptTau_{g^{\ptTau,\delt}}
 u_{i,e}^{\ptTau,\delt}\Bigr)(\cdot)\Bigr\|_{L^2(Q)} \leq C,
\end{align}
where $C$ depends on $\regmesh,\gamma, \alpha,L,l,
\|v_0\|_{L^2(\Om)},\|\Iap\|_{L^2(Q)},
\|g_{i,e}\|_{L^2(0,T;H^1(\Om))}$, and
$\|s_{i,e}\|_{L^2((0,T)\times\Gamma_N)}$.
As usual, in order to obtain \eqref{est:L2-0T-H1}, the case
$\Gamma_D=\text{\O}$ is treated separately, using the normalisation
property \eqref{eq:discrete-normalization} for $u_e$,
the $L^2(Q)$ bound on $v$ and the fact that $u_e=u_i-v$.

\underline{Step 4} (Continuous weak formulation for the discrete solutions).
We take a test function $\ph\in \mathcal D\Bigl((0,T]\times\Bigl(\Om\!\cup\!\Gamma_N\Bigr)\Bigr)$
and discretise it as follows:
$$
\ph^{\ptTau,\delt}:=  \mathbb{P}^\ptTau_c \!\circ\! \mathbb{S}^\ptdelt \ph;
\;\; \text{on the Dirichlet boundary, we take}\;\;
\ph^{\ptl\ptTau}:=0.
$$
Then we use the discrete weak 
formulation \eqref{eq:dualformulation} with 
test function $\delt\ph^{\ptTau,n+1}$ at time level $n$,  and sum 
over $n$. What we get is
\begin{multline}\label{eq:discr-cont-1}
\sum_{n=0}^N  \Bleft v^{\ptTau,n+1}\!\!-\! v^{\ptTau,n}\,,\,
\ph^{\ptTau,n+1} \Bright_\Om
+ \sum_{n=0}^N \delt \Bleft h^{\ptTau,n+1}\,,\,\ph^{\ptTau,n+1}\,\Bright_\Om\\
+ \sum_{n=0}^N \delt \Aleft   \bM_i^{\ptTau}
\grad^\ptTau_{\!g_i^{\!\ptTau,n+1}} u_i^{\ptTau,n+1}\,,\,
\grad_0 \ph^{\ptTau,n+1} \Aright_\Om\\
= \sum_{n=0}^N \delt \Bleft \Iap^{\ptTau,n+1}\,,\,\ph^{\ptTau,n+1} \Bright_\Om
+   \sum_{n=0}^N \delt \Cleft s_i^{\ptTau,n+1}\,,\,
\ph^{\ptl\ptTau,n+1}  \Cright_{\Gamma_N}.
\end{multline}
The equation for the components $u_e^{\ptTau,\delt}$ is analogous.

We use summation by parts on the first
term in \eqref{eq:discr-cont-1}, the Lipschitz continuity of $\ptl_t \ph$,
the definition of $v^{\ptTau,0}$ and
Proposition~\ref{PropConsistency}(i), to see that this term equals
$$
-\iint_{Q} v^{\ptTau,\delt}(\cdot)\, \ptl_t\ph- \int_\Om v_0\,\ph(0,\cdot)
+r^1_\ph(\size(\Tau), \delt)(1+\|v_0\|_{L^1(\Om)}),
$$
where $r_\ph$ denotes a generic remainder term such that
$r_\ph(\size(\Tau),\delt)\to 0$ as $\size(\Tau),\delt\to 0$.
Thanks to \eqref{eq:reaction-summbyparts}, the second
term in \eqref{eq:discr-cont-1} is merely
\begin{multline*}
  \sum_{n=0}^N \delt \int_\Om h(v^{\ptTau,n+1}(\cdot))\;\ph^{\ptTau,n+1}(\cdot)=
\iint_Q h(v^{\ptTau,\delt}(\cdot))\;\ph^{\ptTau,\delt}(\cdot)\\
=\iint_Q h(v^{\ptTau,\delt}(\cdot))\;\ph+
r_\ph(\size(\Tau),\delt)\,C(\regmesh)\Bigl\|h(v^{\ptTau,\delt})\Bigr\|_{L^1(Q)}.
\end{multline*}
Because the discrete gradients are constant per
diamond, thanks to the definition of $\bM_i^{\ptTau,n+1}$
and to the discrete gradient consistency result of 
Proposition~\ref{PropConsistency}(ii), the third term in \eqref{eq:discr-cont-1} is equal to
\begin{multline*}
\sum_{n=0}^N \delt \int_\Om \bM_i (\cdot)
\grad^\ptTau_{\!g_i^{\!\ptTau,n+1}} u_i^{\ptTau,n+1}(\cdot)\cdot
\grad_0 \ph^{\ptTau,n+1}(\cdot)=\iint_Q \bM_i(\cdot)
\grad^\ptTau_{\!g_i^{\!\ptTau,\delt}} u_i^{\ptTau,\delt}(\cdot)\cdot
\grad  \ph\\ + r_\ph(\size(\Tau),\delt)\,C(\regmesh)\|\bM_i\|_{L^\infty(\Om)}
\Bigl\|\grad^\ptTau_{g^{\ptTau,\delt}} u_i^{\ptTau,\delt}\Bigr\|_{L^1(Q)}.
\end{multline*}
Similarly, thanks to further consistency results, the two last 
terms in \eqref{eq:discr-cont-1} can be rewritten as
$$
\iint_Q \Iap\, \ph+\int_0^T\!\!\int_{\Gamma_N} s_i\, \ph
+ r_\ph(\size(\Tau), \delt) \,C(\regmesh)\biggl(
 \|\Iap\|_{L^1(Q)}+\|s_i\|_{L^1((0,T)\times\Gamma_N}\biggr).
$$
Gathering the above calculations, we end up with the weak form
of the discrete equation:
\begin{multline}\label{eq:weak-scheme}
\iint_{Q} \biggl(-v^{\ptTau,\delt}(\cdot) \ptl_t\ph
+ \bM_i (\cdot) \grad^\ptTau_{\!g_i^{\!\ptTau,\delt}} u_i^{\ptTau,\delt}(\cdot)\cdot
\grad  \ph +h(v^{\ptTau,\delt}(\cdot))\;\ph\biggr)\\
= \int_\Om v_0\,\ph(0,\cdot)+\iint_Q \Iap\, \ph
+\int_0^T\!\!\int_{\Gamma_N} s_i\, \ph
\\ + \,C\,r_\ph(\size(\Tau), \delt)
\biggl(1+\Bigl\|\grad^\ptTau_{g^{\ptTau,\delt}} u_i^{\ptTau,\delt}\Bigr\|_{L^1(Q)}\biggr);
\end{multline}
the constant $C$ depends on the data of the problem, $\Om$ and $\regmesh$.
The second equation of the system is analogous, with $u_i^{\ptTau,\delt}$ replaced by
$u_e^{\ptTau,\delt}$ and with signs changed accordingly.

\underline{Step 5} (Convergences via compactness).
All the convergences below are along a sub-sequence
of a sequence $\delt_m,\Tau_m$ of time steps and meshes
with $\size(\Tau_m)+\delt_m$ tending to zero as $m\to \infty$.
In what follows, we drop the subscripts ``$m$'' in the notation;
indeed, after the identification of the limits, the uniqueness of a solution to
Problem~\eqref{S1-p},\eqref{S2},\eqref{S3} will
permit to suppress the extraction argument.

From \eqref{est:L2-0T-H1} and the compactness results in
Sections~\ref{ssec:W1p},~\ref{ssec:time-dep},
we readily find that
\begin{equation}\label{eq:convL2}
u_{i,e}^{\ptTau,\delt}(\cdot) \to u_{i,e}, \;\;
\Bigl(\grad^\ptTau_{g^{\ptTau,\delt}} u_{i,e}^{\ptTau,\delt}\Bigr)(\cdot)
\to \grad u_{i,e} \;\;\text{weakly in $L^2(Q)$},
\end{equation}
as $\size(\Tau),\delt\to 0$; and $u_{i,e}-g_{i,e}\in L^2(0,T;V)$.
Because $v^{\ptTau,\delt}=u_{i}^{\ptTau,\delt}-u_{e}^{\ptTau,\delt}(\cdot)$,
analogous convergences hold for
$v^{\ptTau,\delt}$ and its discrete gradient, the corresponding
limits being $v :=u_i-u_e$ and $\grad v$, respectively.

Moreover, if $\Gamma_D=\ptl\Om$, $g_{i,e}\equiv 0$, we can use
the compactness result of Section~\ref{ssec:compL^1Q} and infer
the strong convergence of $v^{\ptTau,\delt}(\cdot)$ in $L^1(Q)$;
by the preceding remark, the limit is identified with $v$. In the case
of other boundary conditions, we use the local version of
Proposition~\ref{prop:KruzhkovLemma} (shown in \cite{ABR-Cross-VF} for
traditional finite volume schemes; the adaptation to DDFV schemes is straightforward).
Up to now we only have the $L^1(0,T;L^1_{loc}(\Om))$
convergence of $v^{\ptTau,\delt}(\cdot)$ to $v$.
In both cases, the uniform up-to-the-boundary estimate \eqref{est:Lr}
and the interpolation argument yield the strong
convergence of $v^{\ptTau,\delt}(\cdot)$ to
$v$ in $L^{r-\eps}(Q)$, for all $\eps>0$, and the
weak $L^r(Q)$ convergence. In particular,
thanks to the growth assumption \eqref{eq:h-conseq} on $h$ we have
\begin{equation}\label{eq:conv-h}
h(v^{\ptTau,\delt}(\cdot)) \to h(v), \;\;\text{strongly in $L^1(Q)$ as
$\size(\Tau),\delt\to 0$}.
\end{equation}

\underline{Step 6} (Passage to the limit in the continuous weak formulation).
In view of the properties \eqref{eq:convL2},\eqref{eq:conv-h} of Step 5,
the passage to the limit in \eqref{eq:weak-scheme} and the corresponding equation
for $u_e^{\ptTau,\delt}$ is straightforward. We conclude that the limit triple
$\Bigl(u_i,u_e,v\Bigr)$ of $\biggl(u_i^{\ptTau,\delt},u_e^{\ptTau,\delt},v^{\ptTau,\delt}   \biggr)$
is a weak solution of Problem~\eqref{S1-p},\eqref{S2},\eqref{S3}.
In view of the uniqueness of a weak solution, we can
bypass the ``extraction of a sub-sequence" part in Step 5.
This ends the convergence proof.

\underline{Step 7} (Strong convergences).
We will prove that the functions $u_{i,e}^{\ptTau,\delt}$ and their
discrete gradients  converge strongly to $u_{i,e}$,$\grad u_{i,e}$, respectively,
in $L^2(Q)$, while $v^{\ptTau,\delt}$ converges strongly to $v$ in $L^r(Q)$.
To this end, we will utilise monotonicity arguments
to improve the weak convergences  to the strong ones.

 By the established weak convergences
and the strong $L^2$ convergence of $v^{\ptMrond,0},v^{\ptdMrond,0}$
and (for version $(C)$) of $v^{\ptdiezMrond,0}$ to $v_0$, we get
\begin{align}
&\lim_{\size(\ptTau),\delt \to 0} \Biggl(
\frac 12 \Bleft v^{\ptTau,0},\, v^{\ptTau,0}\Bright_\Om 
+ \sum_{n=0}^N \delt \Bleft \Iap^{\ptTau,n+1}\,,\,v^{\ptTau,n+1} \Bright_\Om
\nonumber \\ & \qquad \qquad\qquad\qquad
+   \sum_{n=0}^N \delt \Cleft s_i^{\ptTau,n+1}\,,\,
u_i^{\ptl\ptTau,n+1}  \Cright_{\Gamma_N}  +\sum_{n=0}^N \delt \Cleft s_i^{\ptTau,n+1}\,,\,
u_e^{\ptl\ptTau,n+1}  \Cright_{\Gamma_N}\Biggr) \nonumber
\\[-5pt] & \qquad\qquad
=\frac{1}{2}\int_\Om |v^0|^2+\iint_Q \Iap\,v
+ \int_0^T\!\!\int_{\Gamma_N} \Bigl(s_i u_i+s_e u_e  \Bigr). \label{eq:equal-limits}
\end{align}
First, as in Step 2, take the discrete solutions $u^{\ptTau,\delt}_{i,e}$
as test functions in the discrete
equations and subtract the resulting identities.
Next, with the help of the regularisation Lemma~\ref{lem:regularisation},
take $u_{i,e}$ as test functions in the two equations of the
system, and subtract the resulting identities.
Comparing the two relations with the help of \eqref{eq:equal-limits},
using in addition inequality \eqref{eq:Bsquare>intsquare}, we infer
\begin{multline*}
\lim_{\size(\ptTau),\delt \to 0}\Biggl(\;
 \frac 12 \int_\Om \Bigl| v^{\ptTau,\delt}(T)\Bigr|^2
 + \iint_Q \tilde h(v^{\ptTau,\delt}(\cdot))\,v^{\ptTau,\delt}(\cdot)\\
 + \iint_Q \biggl( \bM_i(\cdot)\Bigl(\grad^\ptTau_{\!g_i^{\!\ptTau,\delt}}
 u_i^{\ptTau,\delt}\Bigr)(\cdot)
 \cdot \Bigl(\grad^\ptTau_{\!g_i^{\!\ptTau,\delt}} u_i^{\ptTau,\delt}\Bigr)(\cdot)
\\  \qquad \qquad\qquad\qquad
+ \bM_e (\cdot) \Bigl(\grad^\ptTau_{\!g_e^{\!\ptTau,\delt}}
u_e^{\ptTau,\delt}\Bigr)(\cdot)\cdot
\Bigl(\grad^\ptTau_{\!g_e^{\!\ptTau,\delt}} u_e^{\ptTau,\delt}\Bigr)(\cdot) \biggr)\;\Biggr)\\
\leq \;\frac 12 \int_\Om |v(T)|^2+\iint_Q \tilde h(v)\,v +
\iint_Q \biggl( \bM_i\grad u_i\cdot \grad u_i
+\bM_e\grad u_e\cdot\grad u_e\biggr).
\end{multline*}

Furthermore, let us assume for simplicity that $L=0$, $l=0$ (which
means that $h(0)=0$ and $h(r)r\geq 0$, so that the Fatou 
lemma can be used); to treat the general
case, use the test function $\zeta(t):=\exp(2L(T-t))\char_{[0,T)}(t)$ in
order to absorb the terms containing $L|v|^2\zeta$ into
the term $ \frac{v^2}{2}\ptl_t\zeta$.

By the Fatou lemma and properties of weak
convergence (with respect to weighted vector-valued $L^2(Q)$ spaces
with weights the matrices $\bM_{i,e}>0$),
we conclude that the above inequality is actually an equality.
Using the fact that weak convergence plus convergence of norms
yields strong convergence in uniformly convex Banach
spaces, using an easy refinement 
of the Fatou lemma\footnote{This result is sometimes 
referred to as the Schaeffe lemma, and it can be stated as follows: 
$\dsp \hspace{5pt}\biggl[\;f_n\geq 0,\; f_n\to f \text{\;a.e. on 
$\Om$},\; \int_{\Om} f_n \to \int_\Om f  \text{\;as $n\to\infty$} \;\biggr] 
\;\; \Longrightarrow \;\; \biggl[\; f_n\to f 
\text{\;in $L^1(\Om)$} \text{\;as $n\to\infty$} \;\biggr]$.} 
(convergence of the integrals implies the strong convergence), we conclude that
\begin{align*}
 & \Bigl(\grad^\ptTau_{\!g_{i,e}^{\!\ptTau,\delt}} u_{i,e}^{\ptTau,\delt}\Bigr)
 (\cdot)\; \to\; \grad u_{i,e} \;\;\;\text{
strongly in $L^2(Q)$,}\\
& \tilde h(v^{\ptTau,\delt}(\cdot))\,v^{\ptTau,\delt}(\cdot) \;\to\;
\tilde h(v)v \;\;\;\text{
strongly in $L^1(Q)$.}
\end{align*}
Using the lower bound in \eqref{eq:h-conseq} and the Vitali theorem, we infer that
$\|v^{\ptTau,\delt}(\cdot)\|_{L^r(Q)}$ converges to $\|v\|_{L^r(Q)}$, thus the weak $L^r(Q)$ 
convergence of $v^{\ptTau,\delt}(\cdot)$ to $v$ is upgraded to the claimed 
strong convergence.

Finally, the strong $L^2(Q)$ convergence of the discrete
gradients of $u_{i,e}^{\ptTau,\delt}$ ensures
a uniform estimate on their translates in time:
$$
\int_0^{T\!-\!\tau}\!\!\!\!\int_\Om | \grad^\ptTau u_{i,e}^{\ptTau,\delt}(t\!+\!\tau,x) 
-  \grad^\ptTau u_{i,e}^{\ptTau,\delt}(t,x) |^2\,dxdt \to 0 
\text{\; as $\tau\!\to\! 0$, uniformly in ${\Tau,\delt}$}.
$$
Then the discrete Poincar\'e inequality yields a
uniform control of the $L^2(Q)$ time translates
$$
\int_0^{T-\tau}\!\!\int_\Om | u_{i,e}^{\ptTau,\delt}(t+\tau,x)
-u_{i,e}^{\ptTau,\delt}(t,x) |^2\,dxdt
$$
of $u_{i,e}^{\ptTau,\delt}$ (here we also use
the uniform time translates of the discrete Dirichlet
data $g_{i,e}^{\ptTau,\delt}$; as usual,
the case $\Gamma_D=\text{\O}$ is treated separately).
Because we can control the space translates of $u_{i,e}^{\ptTau,\delt}$ 
through the uniform $L^2(Q)$ estimate of
$\grad^\ptTau u_{i,e}^{\ptTau,\delt}$ (see Section~\ref{ssec:PoincareSobolev}), 
we conclude that $u_{i,e}^{\ptTau,\delt}$ 
converge strongly in $L^2(Q)$ to to $u_{i,e}$.

\begin{rem}\label{rem:strong-convergences}\rm
Under some stronger proportionality assumptions
on the meshes,  consistency properties
similar to those of Proposition~\ref{PropConsistency}(i),(ii) hold
not only for test functions, but also for
functions in $L^2(0,T;H^1(\Om))$ (cf. \cite{ABH``double''}).
Using the argument of \cite{ABH``double''}, we conclude
that the discrete solution $u_{i,e}^{\ptTau,\delt}$ converges
strongly in $L^2(\Om)$ to $u_{i,e}$. Indeed, from the
discrete Poincar\'e inequality we derive the estimate
\begin{align*}
&\Bigl\| u_{i,e}^{\ptMrond,\delt}
- \mathbb{P}^\ptMrond \circ \,\mathbb{S}^\delt u_{i,e} \Bigr\|_{L^2(Q)}
\\ & \qquad
\leq C(\regmesh,\Om) \Bigl\|  \grad^\ptTau_{\!g_{i,e}^{\!\ptTau,\delt}} u_{i,e}^{\ptTau,\delt}
-  \grad^\ptTau_{\!g_{i,e}^{\!\ptTau,\delt}}\mathbb{P}^\ptTau
\, \circ \mathbb{S}^\delt u_{i,e}\Bigr\|_{L^2(Q)}
\end{align*}
and analogous estimates on the meshes $\dMrond$ and (for version $(C)$) $\diezMrond$.
Then the consistency and strong convergence of the discrete gradients imply the
desired result: we find $\Bigl\| u_{i,e}^{\ptMrond,\delt}
-u_{i,e} \Bigr\|_{L^2(Q)}\to 0$ as $\size(\Tau),\delt\to 0$, and so forth.
\end{rem}

\subsection{A linearised implicit scheme and its convergence}\label{ssec:linear-scheme}
We follow step by step the preceding proof and indicate the
modifications needed to take into account the
linearised-implicit treatment of the ionic current term.
We notice that throughout the calculations of the preceding proof,
\begin{equation}\label{eq:replacement}
\text{$\tilde h(v^{\ptTau,\delt}(\cdot))$ should be replaced by
$b(v^{\ptTau,\delt}(\cdot-\delt))v^{\ptTau,\delt}(\cdot)$};
\end{equation}
where we have set $v^{\ptTau,\delt}(t,\cdot)
:=v^{\ptTau,0}(\cdot)$ for $t\in(-\delt,0]$,
and by $v^{\ptTau,\delt}(\cdot-\delt)$ we mean the function
$(t,x)\in Q\mapsto v^{\ptTau,\delt}(t-\delt,x)$.

\underline{Step 1}. We cannot get the continuous 
dependence with the same technique, but by induction, we get 
uniqueness. Indeed, as soon as the 
uniqueness of $v^{\ptTau,n}$ is justified, we have
$$
\int_\Om \Bigl( b(v^{\ptTau,n}(\cdot))v^{\ptTau,n+1}(\cdot)
- b(v^{\ptTau,n}(\cdot)) \hat v^{\ptTau,n+1}(\cdot)\Bigr)
\Bigl( v^{\ptTau,n+1}(\cdot) - \hat v^{\ptTau,n+1}(\cdot)\Bigr)\geq 0.
$$
This inequality plays the same role as the 
non-negativity of the second term in \eqref{eq:tmp-1}.

\underline{Steps 2 and 4}. The arguments are unchanged,
except for \eqref{eq:replacement}.

\underline{Step 3}. The estimates \eqref{est:Linty-L2} and \eqref{est:L2-0T-H1} remain true.
Notice that the function $b$ is non-negative. Therefore
the estimate \eqref{est:Lr} is replaced by the following one:
\begin{equation}\label{est:mixed}
\iint_Q b(v^{\ptTau,\delt}(\cdot-\delt))\,\Bigl|v^{\ptTau,\delt}(\cdot)\Bigr|^2 \leq C.
\end{equation}
Now we use new arguments. Namely the discrete Sobolev embedding inequality
of Section~\ref{ssec:PoincareSobolev} yields a uniform
$L^2(0,T;L^6(\Om))$ bound on $v^{\ptTau,\delt}$; then
interpolation with the $L^\infty((0,T),L^2(\Om))$ bound \eqref{est:Linty-L2} ensures that
\begin{equation}\label{est:L16-tiers}
 \|v^{\ptTau,\delt}(\cdot)\|_{L^{10/3}(Q)} \leq C.
\end{equation}

\underline{Step 5}. The main difference is the way we ensure the strong
$L^1(Q)$ convergence of $v^{\ptTau,\delt}$ and
the weak $L^1(Q)$ convergence of the
associated ionic current term
\begin{equation}
   \label{eq:ionic-current-linearized-scheme}
   h^{\ptTau,\delt}(\cdot):=b(v^{\ptTau,\delt}(\cdot-\delt))v^{\ptTau,\delt}(\cdot) \,
   -\,L v^{\ptTau,\delt}(\cdot) \,-\,l.
\end{equation}
It is sufficient to treat the nonlinear part of $h^{\ptTau,\delt}(\cdot)$; thus
we can ``forget'' about the two last
terms in \eqref{eq:ionic-current-linearized-scheme}.
Let us first notice that the definition of $b$ and the
growth bound \eqref{eq:h-conseq} on $h$
imply that for some constant $\beta=\beta(\alpha,L,l)$ we have
$$
b(z)\;\leq\; \beta\, (1+|z|^{r-2}).
$$
Then the assumption $r-2<16/3-2=10/3$ made in
Theorem~\ref{th:convergence}(ii) and the uniform $L^{10/3}(Q)$
bound \eqref{est:L16-tiers}
on $v^{\ptTau,\delt}$  ensure the equi-integrability of
the functions $ b(v^{\ptTau,\delt}(\cdot-\delt))$ on $Q$.
Now for any measurable set $E\subset Q$, for all $\delta >0$,
\begin{multline*}
\iint_E |b(v^{\ptTau,\ptDelt}(\cdot-\delt))\,v^{\ptTau,\delt}(\cdot)|
\leq \frac 1\delta \iint_E  b(v^{\ptTau,\delt}(\cdot-\delt))
\\ + {\delta} \iint_Q
b(v^{\ptTau,\delt}(\cdot-\delt,x))\,\Bigl|v^{\ptTau,\delt}(\cdot)\Bigr|^2.
\end{multline*}
Thus estimate \eqref{est:mixed} and the aforementioned
equi-integrability of $ b(v^{\ptTau,\delt}(\cdot-\delt))$
ensure the equi-integrability of the ionic current term
$h^{\ptTau,\delt}(\cdot)$. In particular, from \eqref{eq:h-linearized-impl}
we infer $L^1(Q)$ bounds on the components
of the discrete function $h^{\ptTau,\delt}$:
$$
\sum_{n=0}^N \delt \biggl(\|h^{\ptMrond,n+1}\|_{L^1(\Om)}
+\|h^{\ptdMrond,n+1}\|_{L^1(\Om)}\biggr) \leq C,
$$
for version $(C)$, the term on $\diezMrond$ is also controlled.
At this stage, the full strength of Proposition~\ref{prop:KruzhkovLemma}
is put into service: indeed, we only have the $L^1$ control
of the right-hand side of the discrete  evolution equations \eqref{all-AbstractScheme}.
We infer the strong $L^1(Q)$ convergence (along a sub-sequence) for $v^{\ptTau,\delt}$.
Then from the Vitali theorem and the fact that
$v^{\ptTau,\delt}(\cdot-\delt)-v^{\ptTau,\delt}(\cdot)\to 0$ in $L^1(Q)$
and a.e., we get the strong
convergence of $h^{\ptTau,\delt}(\cdot)$ to $h(v)$.

\underline{Steps 6 and 7}. The arguments remain unchanged, taking
into account \eqref{eq:replacement}.

\section{Numerical experiments}\label{sec:Numerics}

For the numerical simulations, the bidomain problem \eqref{S1-p} is
reformulated in terms of $v$ and $u_e$ only; the elimination of
$u_i$, thanks to the relation $v=u_i-u_e$, decreases the
number of unknowns per primal/dual volume from three to two.

In terms of $v$ and $u_e$, the parabolic type
problem \eqref{S1-p} is turned into the
following elliptic-parabolic problem:
\begin{equation}
    \label{eq:num-model}
   \left\{
     \begin{split}
       \dsp
        \Div (\bM_e(x)+\bM_i(x))\Grad u_e
        +
        \Div \bM_i(x)\Grad v = 0
        \qquad (t,x)\in Q,
        \\ \dsp
        \varepsilon \pt v
        +\varepsilon^2\Div \bM_e(x)\Grad u_e
        +h[v]= \Iap
        \qquad (t,x)\in Q.
    \end{split}
    \right.
\end{equation}
For $\varepsilon=1$, problem  \eqref{eq:num-model} is equivalent
to the original problem  \eqref{S1-p}; indeed,
the first line in \eqref{eq:num-model} results from
the summation of the two lines in \eqref{S1-p}, with $u_i=v+u_e$.
The bidomain problem will be considered under formulation
\eqref{eq:num-model} throughout this section. We point
out that numerical schemes associated with
formulation \eqref{eq:num-model} are equivalent with
numerical schemes for the formulation \eqref{S1-p}, following
the same algebraic operations on the discrete equations.

A scaling parameter $\varepsilon$ has also been introduced in \eqref{eq:num-model}.
Its presence clearly makes no difference for the mathematical
study of the previous sections, but it greatly helps the solutions of
\eqref{eq:num-model} to behave as  excitation potential waves (which waves
\eqref{eq:num-model} is supposed to model). More precisely, following
the analysis in \cite{ColliGuerriTentoni90}, such
a scaling parameter together with a cubic shape for
$h[v]:=v(v-1)(v-\alpha)$ provide a simplified model for 
spreading of excitation in the myocardium;
the parameters $\varepsilon$ and $\alpha$ have been
set respectively to 1/50 and 0.2. The way  excitation
waves are generated is detailed in Subsection \ref{subsec:num-settings}.

The convergence of the DDFV space discretisations for the bidomain problem
has been  justified in Theorem~\ref{th:convergence} for two
different discretisations (the fully implicit in time and the
linearised semi-implicit one) of the ionic current term. 
Here we complement these theoretical studies by the
numerical experiments on the third (and the most
important in practise) case of fully explicit in time
discretisation of the ionic current term.
The implementation of the scheme  is detailed
in Section~\ref{subsec:num-implem}.
Although the theoretical study of this scheme is made difficult
for technical reasons, we do observe convergence numerically.

The convergence result in Theorem \ref{th:convergence} involves
a comparison between the exact solution and the discrete
solution in the $L^2(Q)$ norm, the discrete solution
being interpreted as the weighted sum of two piecewise
constant functions (on the primal and on the dual cells). The
practical computation of this $L^2$ distance is  difficult.
Namely, a (coarse) numerical solution on a given (coarse) mesh
has to be compared with a second (reference)
numerical solution computed on a reference mesh aimed to
reproduce the exact solution, unknown in practise. The reference
mesh will not be here a refinement of the coarse one, thus
the precise computation of an $L^2$ norm between the coarse
and the reference numerical solutions, thought as
piecewise constant functions, is an awful task.
Therefore we use a slightly different convergence indicator, as
presented in Section~\ref{subsec:num-results}.
The quantity \eqref{eq:def-space-time-err} provides a more
convenient measure of the square of the $L^2(Q)$ norm.
The convergence with respect to this variant of discrete $L^2$ norm will be studied here.
Notice that under reasonable regularity assumptions on the mesh,
the two discrete $L^2$ norms are equivalent.

Convergence will be studied both in 2D and in 3D. The introduction
of the two dimensional case is mostly intended to confirm the
results in dimension three: because of the numerical
facilities in dimension two (smaller growth
of the problem size under refinement), it allows a
deeper insight into the asymptotic behaviours.

\subsection{Settings}\label{subsec:num-settings}
Problem \eqref{eq:num-model} is considered for a reaction
term $v\mapsto h[v]$ set on the domain $\Om=[0,1]^d$,  $d=2,3$ denoting
the space dimension. We consider solutions under the form of  excitation potential
waves spreading from the domain centre
towards its boundary, as depicted in Figure \ref{fig:2d-prop} in the
two dimensional case, relatively to some medium anisotropy detailed below.
\begin{figure}[!ht]
  \centering
  \begin{tabular}{ccc}
    \includegraphics[width=110pt]{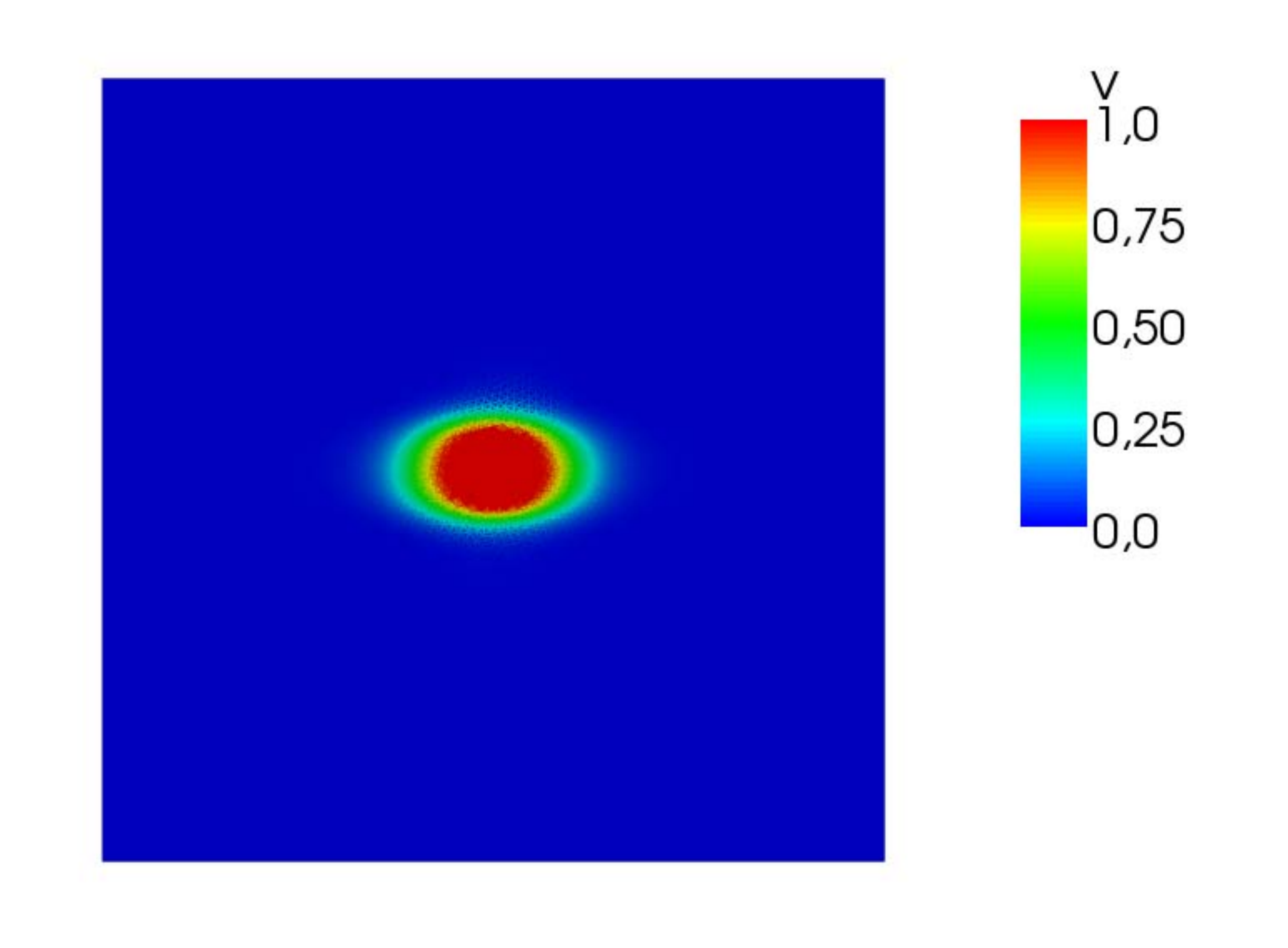}
    &
    \includegraphics[width=110pt]{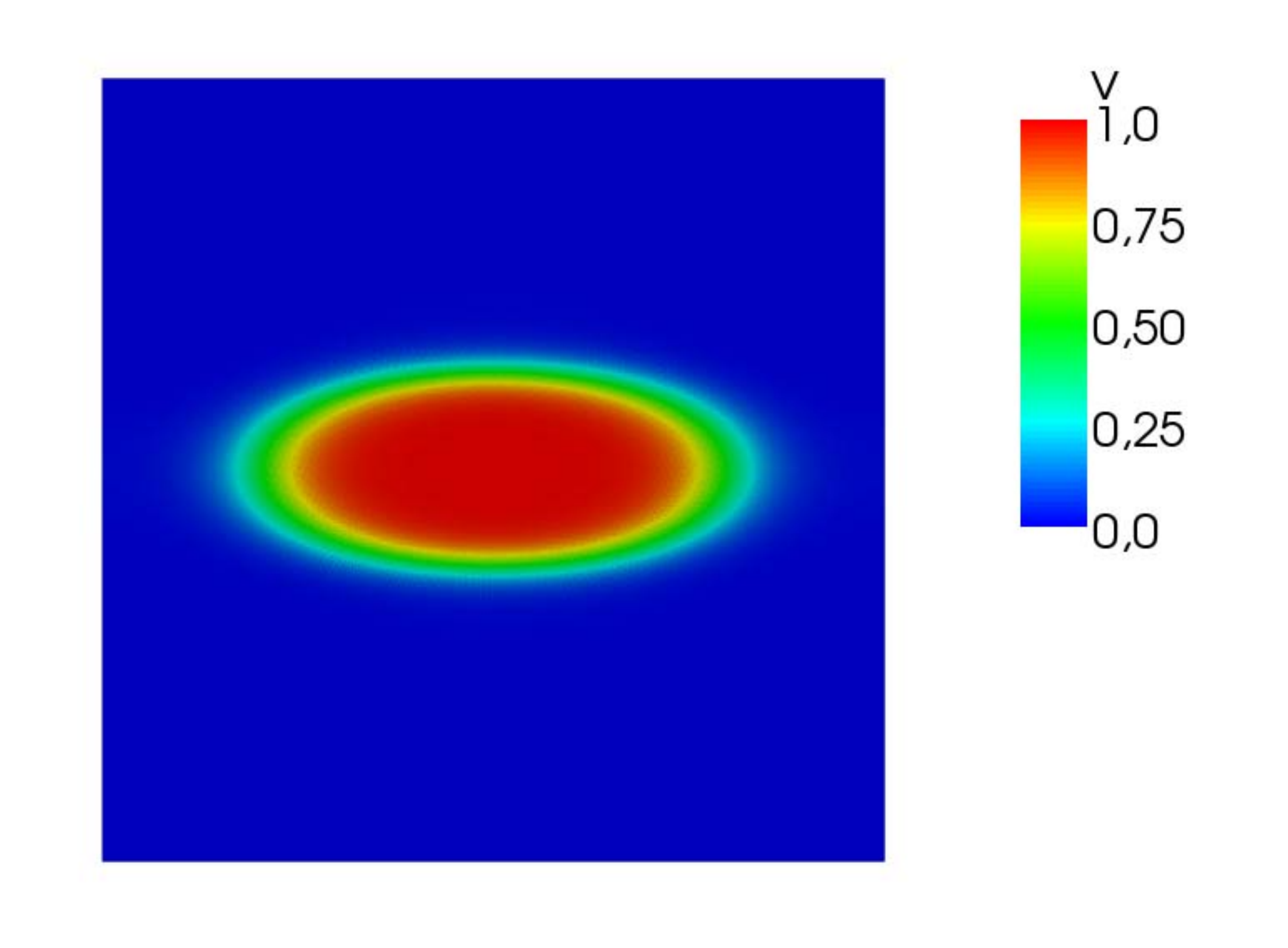}
    &
    \includegraphics[width=110pt]{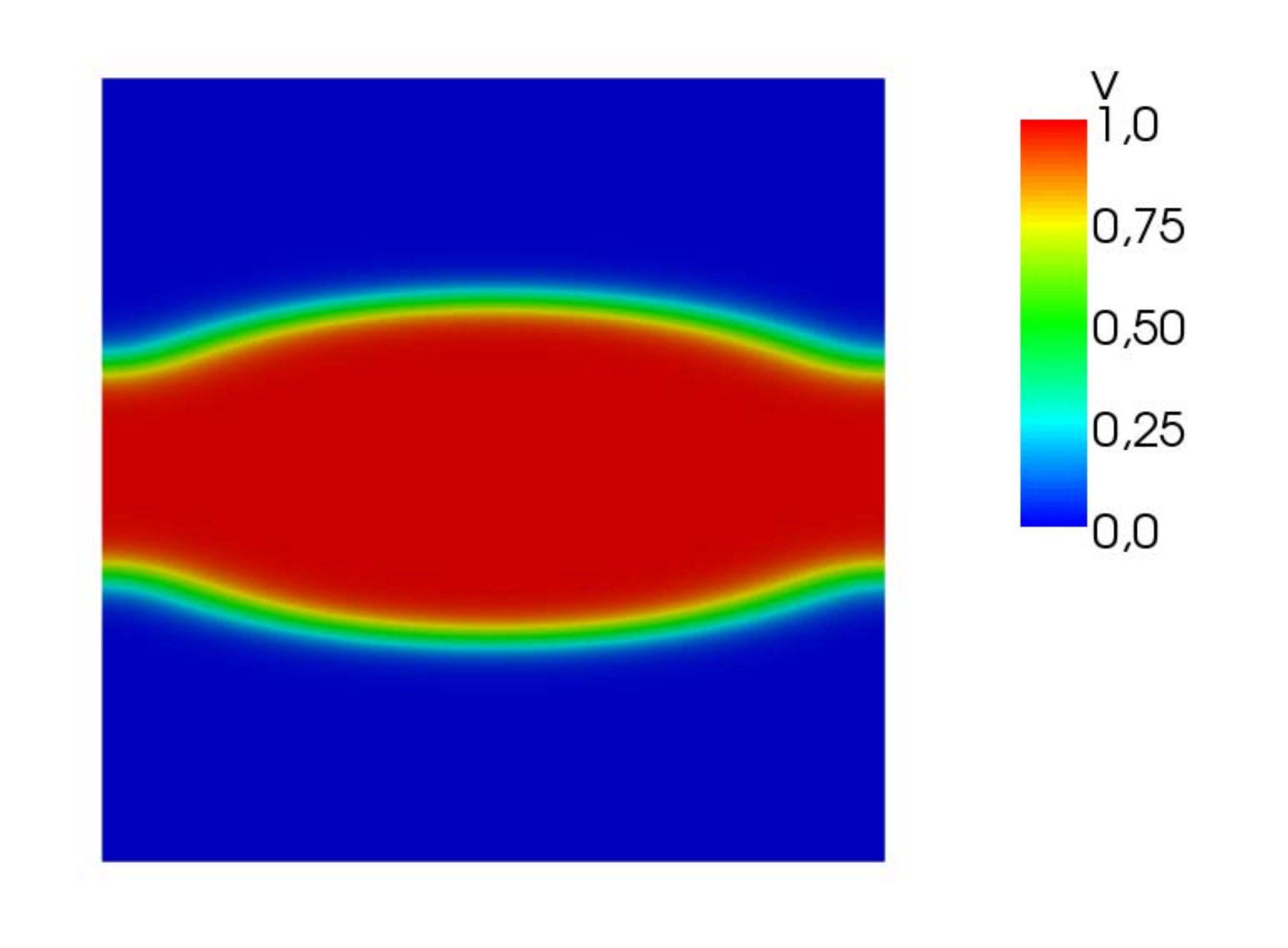}
    \\
    \includegraphics[width=110pt]{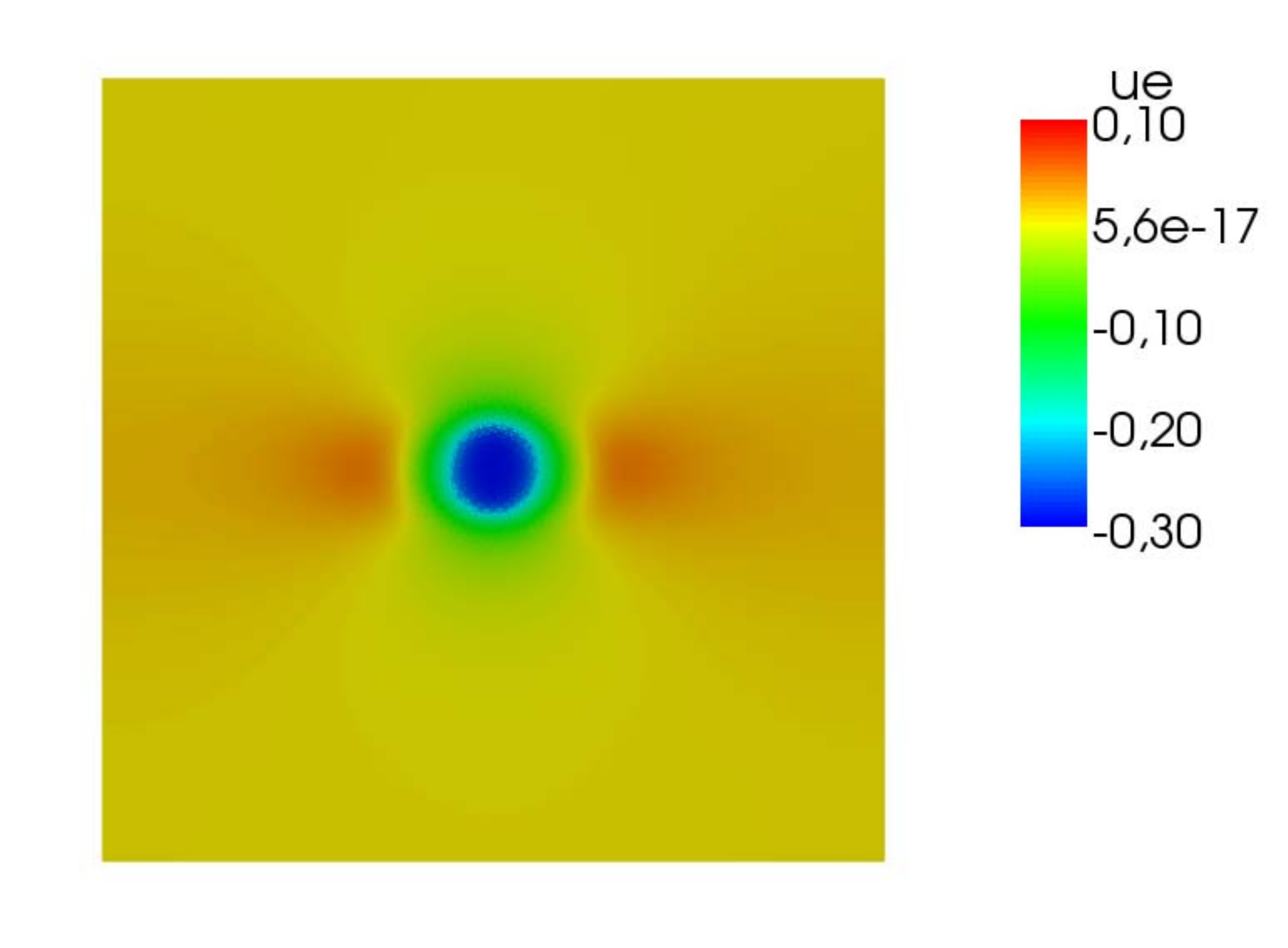}
    &
    \includegraphics[width=110pt]{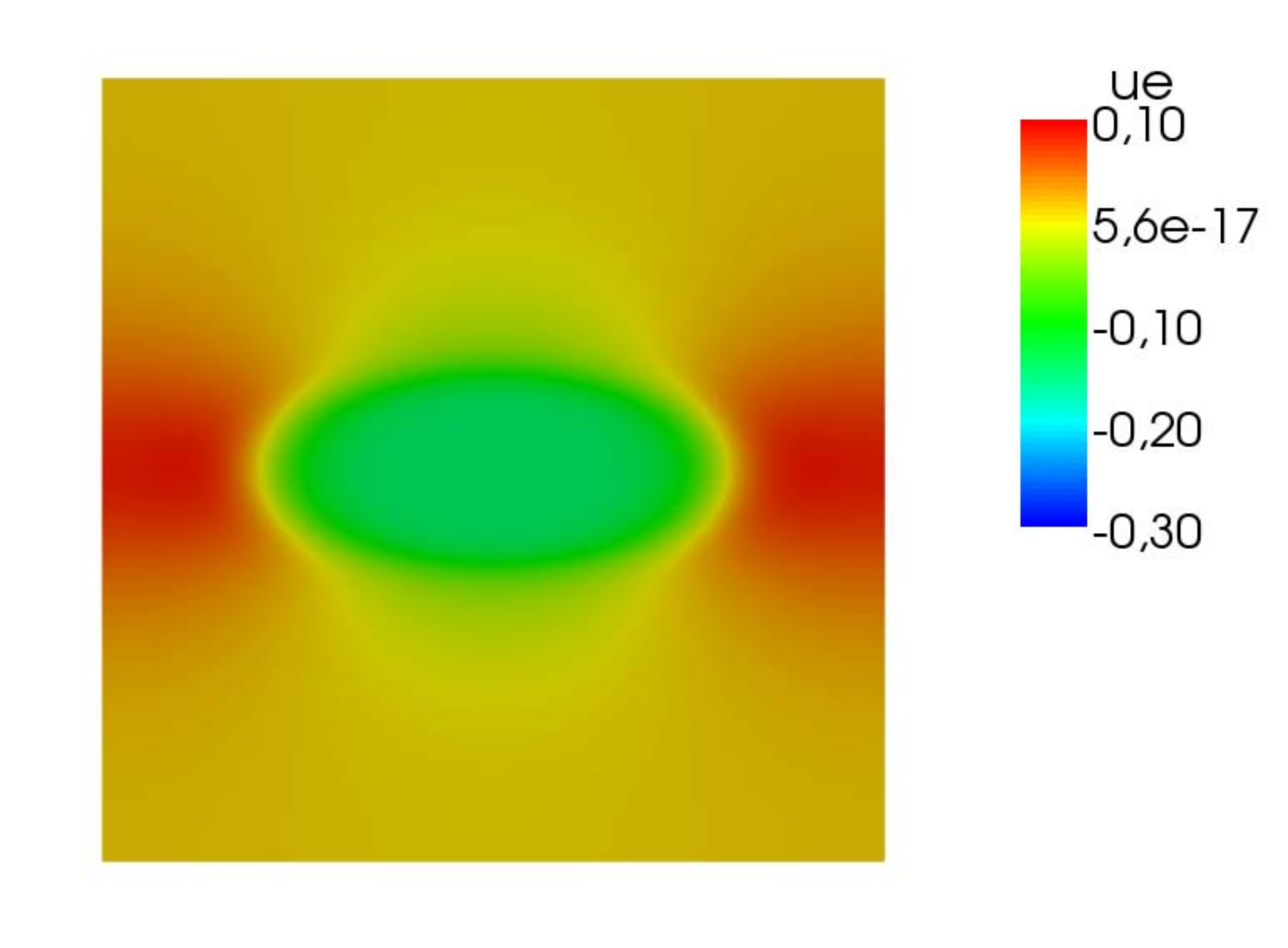}
    &
    \includegraphics[width=110pt]{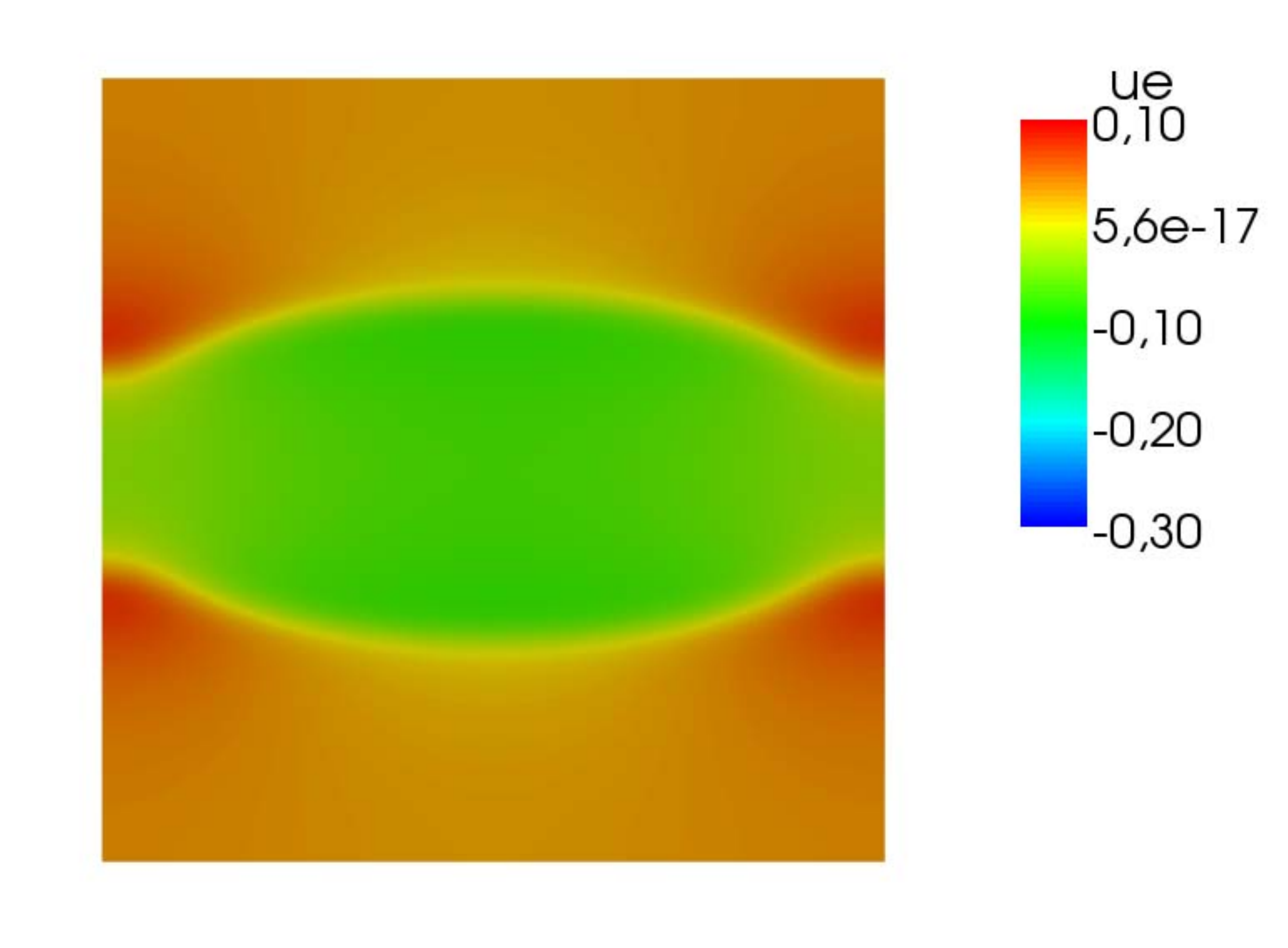}
  \end{tabular}
  \caption{
  Reference solution in the two-dimensional case.
    Above: spreading of the transmembrane potential $v$ excitation wave.
    From left to right, the three pictures correspond to times $t$=0.2, 0.6 and 1.2 after
    the stimulation initiation, stimulation duration being 0.1.
    Below: associated extracellular potential $u_e$.
  }
  \label{fig:2d-prop}
\end{figure}

Excitation is initiated by applying a centred stimulation
during a short period of time, precisely: $\Iap(x,t)=0.9$ for $1<t<1.1$
and $\vert x-x_0\vert<0.1$ ($x_0$ denoting the centre
of $\Om$) and $\Iap(x,t)=0$ otherwise. The
initial condition for $v$ is  uniformly set to 0. A homogeneous
Neumann boundary condition is considered
on $\partial\Om$, uniqueness is ensured by adding
the normalisation condition \eqref{eq:Ui-normalized} on $u_e$. 
The domain $\Om$ is assumed to be composed of a
bundle of parallel horizontal muscular fibres, resulting
in the following choice for the anisotropy tensors $\bM_i(x)$ and $\bM_e(x)$:
\begin{equation}
  \label{eq:num-def-tensor}
  \bM_i(x)=\text{Diag}(\lambda_i^l,\lambda_i^t,\lambda_i^t)~,\quad
  \bM_e(x)=\text{Diag}(\lambda_e^l,\lambda_e^t,\lambda_e^t)~,
\end{equation}
the values for the longitudinal ($l$) and transverse ($t$)
conductivities for the intra and extra-cellular medias
have been taken from \cite{leguyader_01}: the resulting
anisotropy ratios for the intra and extra-cellular medias
respectively are 9.0 and 2.0 between the longitudinal
and transverse directions.

The numerical solution for the transmembrane
potential $v$  takes the form of an excitation wave propagating
across the domain from the stimulation site towards the
boundary and from the rest potential $v=0$ to the activation potential $v=1$.
A sharp but smooth wavefront for $v$ displays an elliptic
shape away from the boundary, which is induced
by the media anisotropy. A reference solution is
generated on a mesh using a 1 147 933 (resp. 479 873)
nodes in dimension 3 (resp. 2).

\subsection{Implementation}\label{subsec:num-implem}
Let us fix a mesh $\Tau$.
For simplicity, discrete functions $w,v \in \R^\ptTau$ will also
be considered as one-column real matrices in this
subsection, $w^T$, $v^T$ denoting their transpose
one-row real matrices. Let us first introduce the
mass matrix (diagonal here) $\mass\in \text{Mat}(\R^\ptTau)$:
\begin{equation*}
  \forall w,v \in \R^\ptTau:\quad
  \Bleft w,v\Bright_{\Om} = w^T \mass v.
\end{equation*}
Relative to \eqref{eq:num-def-tensor}, uniform
discrete tensors $\bM_{i,e}^\ptTau$ are considered here, with value
\begin{displaymath}
  \bM_i^\ptTau=\text{Diag}(\lambda_i^l,\lambda_i^t,\lambda_i^t)~,\quad
  \bM_e^\ptTau=\text{Diag}(\lambda_e^l,\lambda_e^t,\lambda_e^t)~,
 \end{displaymath}
on each diamond $\DM$. The discrete gradient being defined
relative to the homogeneous Neumann
boundary condition, and simply denoted by
$\grad^\ptTau$, the two stiffness matrices
$\stiff_i$ and $\stiff_e$ are introduced as:
\begin{equation*}
  \forall w,v \in \R^\ptTau:\quad
\Aleft \bM_i^\ptTau \grad^\ptTau v,
\grad^\ptTau w\Aright_{\Om}=
v^T \stiff_i w
,\quad
\Aleft \bM_e^\ptTau \grad^\ptTau v, \grad^\ptTau w\Aright_{\Om}=
v^T \stiff_e w.
\end{equation*}
These stiffness matrices are positive, symmetric
matrices, although not definite since a Neumann homogeneous
boundary condition is considered.

The following semi-implicit Euler scheme is considered: given
$v^n,\Iap^n\in\R^\ptTau$, determine
$v^{n+1},u_e^{n+1}\in\R^\ptTau$ such that,
\begin{equation}\label{eq:num-sc1}
  \left\{\begin{array}{l}
      \dsp
      \Div^\ptTau[ (\bM_e^\ptTau+\bM_i^\ptTau)\grad^\ptTau u_e^{n+1}]
      +
      \Div^\ptTau\left[ \bM_i^\ptTau\grad^\ptTau v^{n+1}\right] = 0,
      \\[5pt] \dsp
      \varepsilon \frac{v^{n+1} - v^{n}}{\Delt}
      + \varepsilon^2
      \div^\ptTau[\bM_e^\ptTau \grad^\ptTau u_e^{n+1}]+
      h[v^n] = \Iap^{n}.
    \end{array}\right.
\end{equation}
Writing separately the scalar product
$\Bleft \cdot,\cdot\Bright_{\Om}$
of each line in \eqref{eq:num-sc1} with all test functions
$w\in\R^\ptTau$
and using the discrete duality \eqref{eq:discr-duality} leads
to the following equivalent formulation written in matrix form:
\begin{equation}
  \label{eq:num-sc2}
  M \left\vert
    \begin{array}{l}
      \dsp
      u_e^{n+1}
      \\[15pt]\dsp
      v^{n+1}
    \end{array}
  \right.
  =
  \left\vert
    \begin{array}{l}
      \dsp
      0
      \\[15pt]\dsp
      \mass [v^{n} + \Delt( \Iap^{n}-h[v^n])/\varepsilon ]
    \end{array}
  \right.
  ~~ ,\quad
  M:=
  \left [
    \begin{array}{cc}
      \dsp
      \stiff_i+\stiff_e & \stiff_i
      \\[5pt] \dsp
      -\varepsilon \Delt \stiff_e & \mass
    \end{array}
  \right ].
\end{equation}
To ensure uniqueness on $u_e$, the normalisation condition
\eqref{eq:Ui-normalized}
 is discretised as:
\begin{equation}
  \label{eq:num-norm-cond-disc}
  \Bleft u_e^{n+1}, U^\ptMrond  \Bright_{\Om} =
  0 =
  \Bleft u_e^{n+1}, U^\ptdMrond \Bright_{\Om},
\end{equation}
where $U^\ptMrond$ (resp. $U^\ptdMrond$) is the discrete function equal to 1 relatively to each primal (resp. dual) control volumes and to 0 elsewhere.

The implementation of \eqref{eq:num-sc2} therefore reads as a three-step algorithm:

at each time step,
\begin{enumerate}
\item
  compute $y_2:=\mass [v^{n} + \Delt( \Iap^{n}-h[v^n])/\varepsilon ]$.
  The matrix $\mass$ being diagonal, computations for this step are cheap;
\item
  determine a solution $x=(u_e^{n+1},v^{n+1})^T$ to the global system $Mx=y$ for
  $y=(0,y_2)^T$;
\item
  normalise $u_e^{n+1}$ using condition \eqref{eq:num-norm-cond-disc}.
  For the same reason as for step one, this step is a cheap one.
\end{enumerate}

Because of the large size of the considered problem
(1.1 million of nodes in 3D for the most refined
mesh, i.e. 2.2 million of lines for the matrix $M$) and
because of the relatively non-compact sparsity pattern for $M$ in 3D,
step 2 is not an easy task. Therefore, a careful attention
has to be paid to the preconditioning of $M$: the strategy
adopted here is  detailed in \cite{pierre-2010-precond}.

\subsection{Numerical tests and results}\label{subsec:num-results}

The convergence of the DDFV scheme is numerically analysed
comparing the reference solution described in Subsection \ref{subsec:num-settings}
with numerical solutions obtained on coarser meshes.
In 3D, four tetrahedral meshes have been considered:
from 2~559 to   1~147~933 nodes, between two meshes
the \textit{mesh size} is divided by 2, two successive meshes
are not obtained via refinement.
In 2D, six meshes are used: from 489 to 479~873  nodes.
The time step $\Delt$ is also divided by two each time the
space resolution is divided by 2; the starting time step (on the coarsest mesh) is 0.02.

To compare numerical solutions defined on different meshes, a
projection is needed: this is done as follows.
Let  $\Tau^r$ and $\Tau^c$ be the reference mesh and a coarser
mesh respectively, and let $\R^{\ptTau_r}$, $\R^{\ptTau_c}$  respectively
denote the associated spaces of discrete functions.
Consider the simplicial mesh $\mathcal{S}^r$ (respectively $\mathcal{S}^c$)
whose  cells are obtained by cutting all diamonds of $\Tau^r$ (resp. $\Tau^c$)
in two along the interface.   A discrete function
$u^r\in\R^{\ptTau_r}$ (resp. $u^c\in\R^{\ptTau_c}$) consists
in one scalar associated to each vertex and each cell centre of the
mesh $\Tau^r$ (resp. $\Tau^c$), thus to each vertex of
$\mathcal{S}^r$ (resp. $\mathcal{S}^c$). It is therefore natural
to associate to $u^c$ (resp. $u^r$) the continuous
function $\tilde{u}^r$ (resp. $\tilde{u}^c$) piecewise affine on the
cells of $\mathcal{S}^r$ (resp. $\mathcal{S}^c$) and
whose values at the vertices of $\mathcal{S}^r$ (resp. $\mathcal{S}^c$) are given by the
discrete function $u^r$ (resp. $u^c$). A projection
$u^{c\to r}\in\R^{\ptTau_r}$ of a (coarse) discrete function
$u^c\in\R^{\ptTau_c}$ is then simply defined by computing the
values of the function $\tilde{u}^c$ on the vertices of  $\mathcal{S}^r$.
The relative error between $u^r\in\R^{\ptTau_r}$ and $u^c\in\R^{\ptTau_c}$
in $L^2(\Om)$ norm is defined as
\begin{equation}
  \label{eq:def-space-err}
  e_{\Om,2}(u^r,u^c)^2:=\dfrac{
    \int_\Om
    \vert \tilde{u}^r - \tilde{u}^{c\to r} \vert ^ 2
    dx
  }
  {
    \int_\Om
    \vert \tilde{u}^r \vert ^ 2
    dx
  }.
\end{equation}
Numerically, these integrals are evaluated using an order two
Gauss quadrature on the cells of $\mathcal{S}^r$, leading to
an exact evaluation up to rounding errors.

The three following tests have been performed.

\subsubsection{Test 1; activation time convergence}
\begin{figure}[!ht]
  \centering
  \begin{tabular}{ccc}
    \includegraphics[width=110pt]{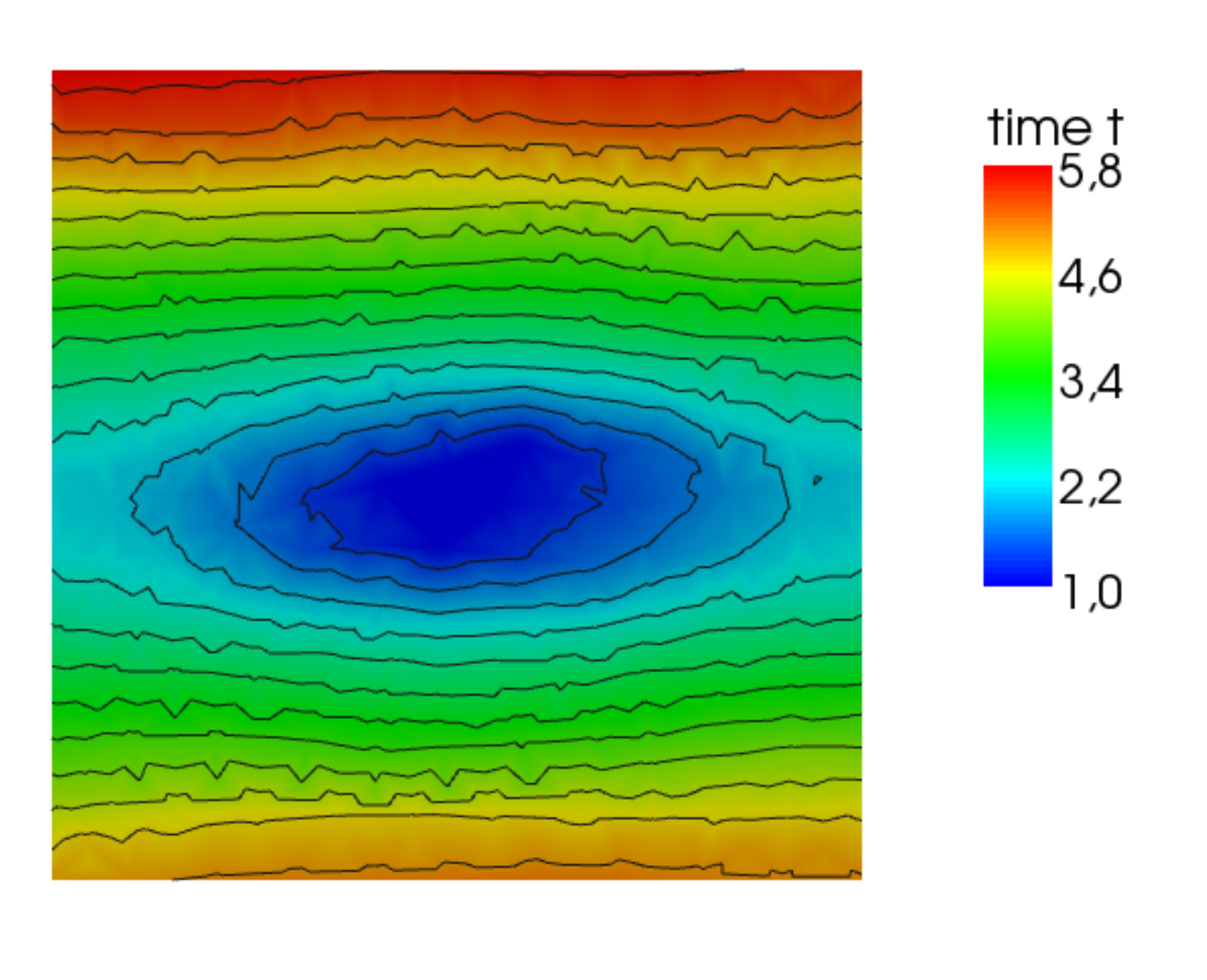}
    &
    \includegraphics[width=110pt]{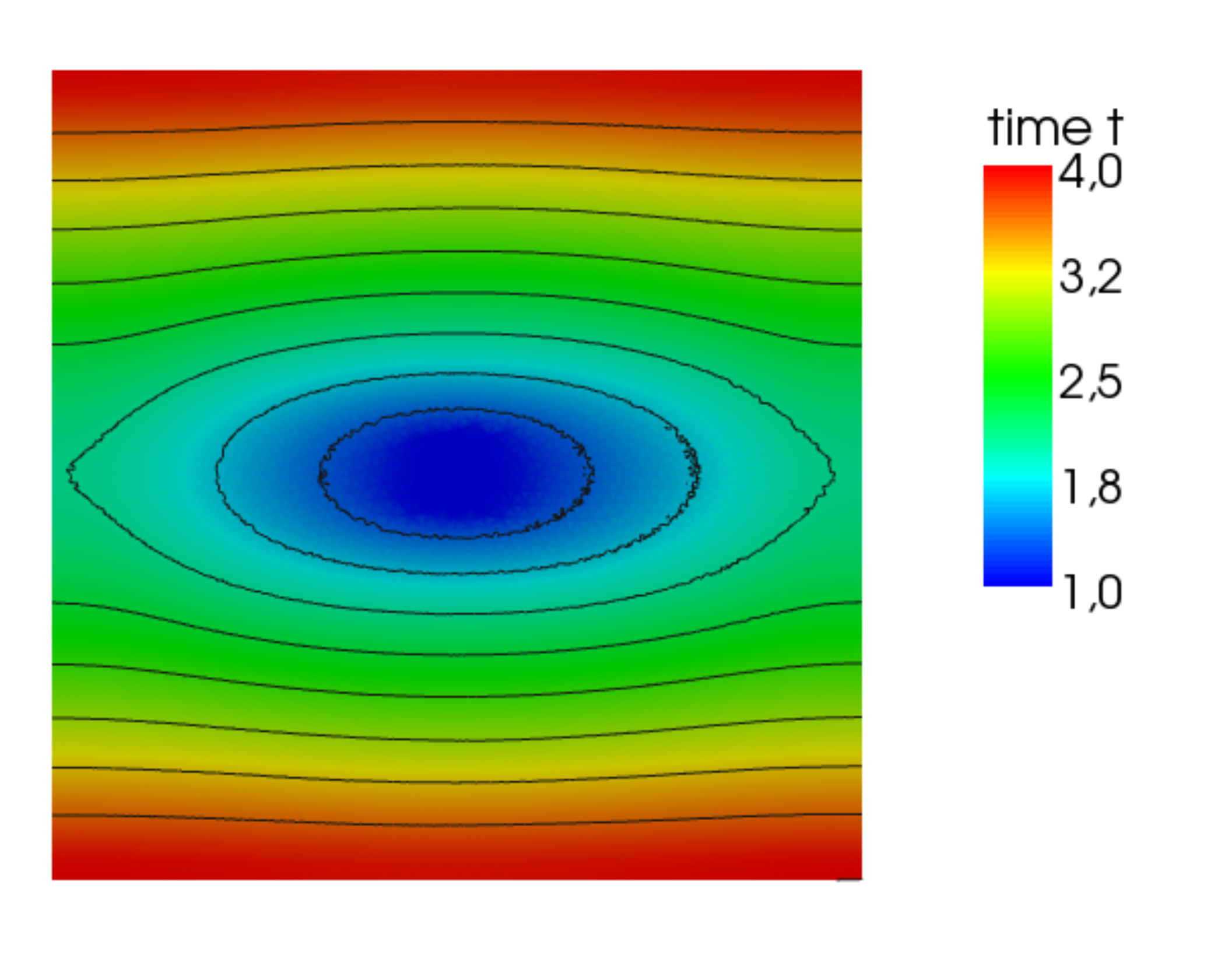}
    &
    \includegraphics[width=110pt]{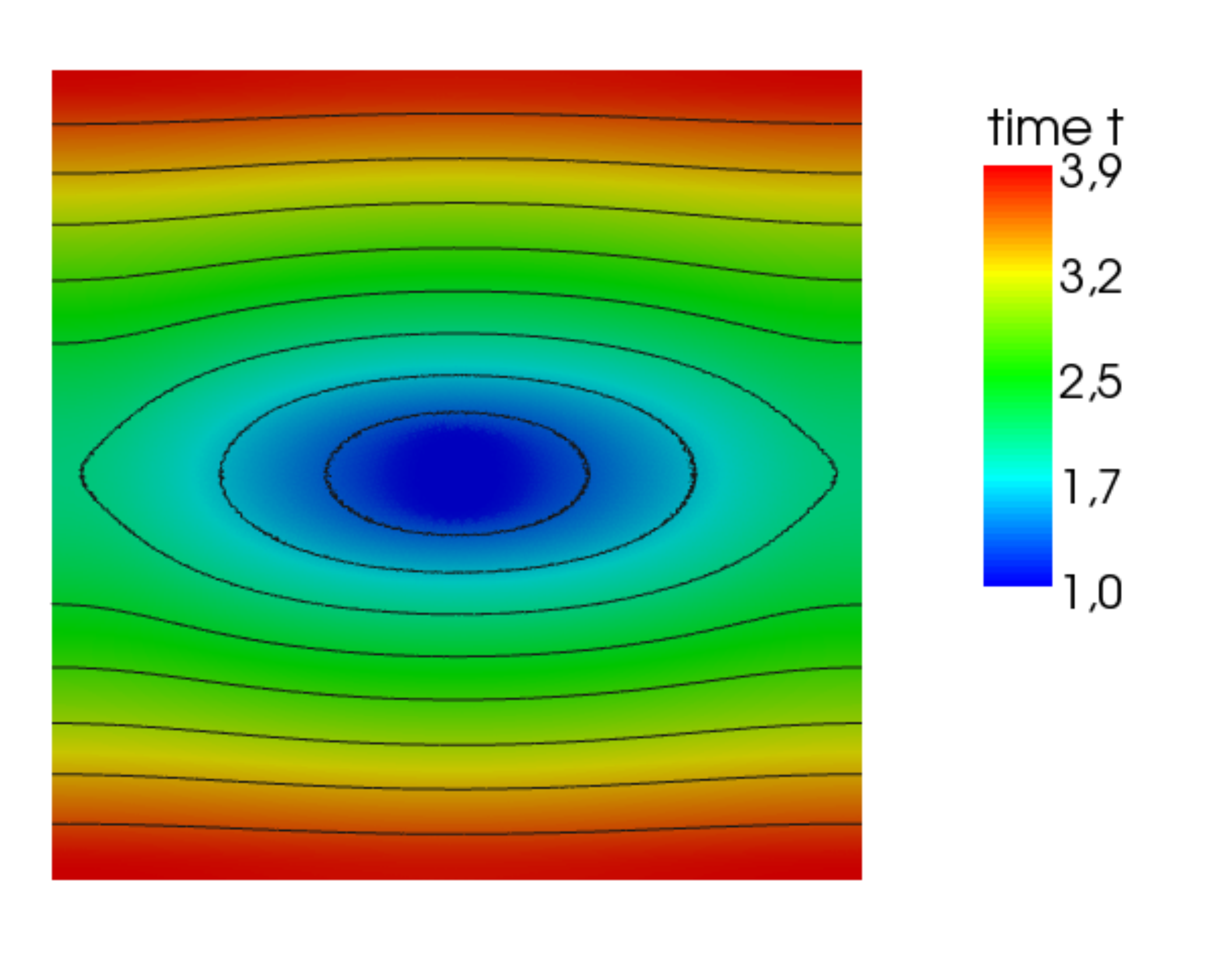}
  \end{tabular}
  \caption{
  Activation time in dimension 2 for three different meshes, the
  isolines (in black) are separated by 1/3 unit of time. The stimulation
  is initiated at time $t=1$. From the left (coarsest mesh)
  to the right (reference solution) the three
  different activation time mappings
  have been computed on meshes with 439, 7569 and 479 873 nodes respectively.}
  \label{fig:2d-conv-act-time}
\end{figure}
The activation time mapping $\phi:~\Om\mapsto \mathbb{R}$ is
defined at each point $x$ as the time $\phi(x)=t$ such that the
transmembrane potential $v(x,t)=s$ for the threshold value $s:=0.9$. The value
$\phi(x)$ tells us at what time the excitation wave reaches the point $x$, the
activation time mapping thus is of crucial importance in terms
of physiological interpretation of the model.
Activation time in 2D computed on various meshes are depicted on Figure
 \ref{fig:2d-conv-act-time}.
The discrepancy between the activation mappings
$\phi^r$ and $\phi^c$ computed at the reference and coarse
levels respectively is evaluated using the relative error in the
$L^2(\Om)$ norm defined in \eqref{eq:def-space-err}.
Numerical results for activation time convergence are given in
Table \ref{tab:act-time-conv}.

\begin{table}
  \centering
  \begin{tabular}{ccc}
    \begin{tabular}{|lc|clc|}
      \hline
      ~~~\# nodes &~~~~~~&~~~~~& errors $e_{\Om,2}$ &~~~~~~~~
      \\ \hline \hline
      489&&&
      1.611
      &\\ \hline
      1913&&&
      9.130  $10^{-2}$
      &\\ \hline
      7 569&&&
      1.776 $10^{-2}$
      &\\ \hline
      30 113&&&
      4.850 $10^{-3}$
      &\\ \hline
      120 129&&&
      1.139 $10^{-3}$
      &\\ \hline
      479 873 &&& reference
      &\\ \hline
    \end{tabular}
    &$\qquad$&
    \begin{tabular}{|lc|clc|}
      \hline
      ~~~\# nodes &~~~~~~&~~~~~& errors $e_{\Om,2}$ &~~~~~~~~
      \\ \hline \hline
      2 559&&&
      1.110
      &\\ \hline
      19 500&&&
      8.195 $10^{-2}$
      &\\ \hline
      148 242&&&
      1.281 $10^{-2}$
      &\\ \hline
      1 147 933 &&& reference
      &\\ \hline
    \end{tabular}
    \\
    \\
    2D case && 3D case
  \end{tabular}
  \caption{Activation time mappings convergence.
  The errors are relative errors in $L^2(\Om)$ norm
  as defined in \eqref{eq:def-space-err}}
  \label{tab:act-time-conv}
\end{table}

Convergence is numerically observed here both in 2D and in 3D.
Moreover, the figures obtained in the two dimensional case indicate
an order two convergence relatively to
the mesh size. Such a conclusion, although
plausible, is not possible in the three dimensional case: to be observed it would
require a much finer reference mesh which is not affordable
in terms of computational effort.

\subsubsection{Test 2; space convergence}
Let us denote by $v^r$ and $u_{e}^r$, (resp. $v^c$ and $u_{e}^c$) the
transmembrane potential and extracellular potential
computed on the reference mesh (resp. a coarse mesh).
The discrepancy between $v^c$, $v^r$ and $u_{e}^r$, $u_{e}^c$
at  a chosen time $t$ has been computed using the relative
error in the $L^2(\Om)$ norm \eqref{eq:def-space-err}.
Three fixed times $t$ have been considered:
$t=$ 1.2, 1.6 and 2.2, corresponding to the reference
solution depicted in Figure \ref{fig:2d-prop}.

\begin{table}
  \centering
  \begin{tabular}{c}
    \begin{tabular}{|l|ccc|ccc|}
      \hline
      \# nodes
      & $t=1.2$& $t=1.6$& $t=2.2$
      & $t=1.2$& $t=1.6$& $t=2.2$
      \\ \hline \hline
      489
      & 0.51 & 0.37  & 0.45
      & 0.45 & 0.34  & 0.59
      \\ \hline
      1913
      & 0.24 & 8.87 $10^{-2}$  & 0.14
      & 0.22 & 8.27 $10^{-2}$  & 0.17
      \\ \hline
      7569
      & 0.13  & 6.01 $10^{-2}$  & 1.42 $10^{-2}$
      & 0.12  & 5.38 $10^{-2}$  & 1.99 $10^{-2}$
      \\ \hline
      30113
      & 6.24  $10^{-2}$ & 3.29 $10^{-2}$ & 1.28 $10^{-2}$
      & 5.62  $10^{-2}$ & 2.89 $10^{-2}$ & 1.71 $10^{-2}$
      \\ \hline
      120129
      & 1.25 $10^{-2}$ & 5.26 $10^{-3}$ & 1.74 $10^{-3}$
      & 1.13 $10^{-2}$ & 5.31 $10^{-3}$ & 2.86 $10^{-3}$
      \\ \hline
    \end{tabular}
    \\
    2D case
    \\ \\
    \begin{tabular}{|l|ccc|ccc|}
      \hline
      \# nodes
      & $t=1.2$& $t=1.6$& $t=2.2$
      & $t=1.2$& $t=1.6$& $t=2.2$
      \\ \hline \hline
      2559
      &0.43 & 0.57 & 0.65
      &0.42 & 0.63 & 0.91
      \\ \hline
      19 500
      &0.22 & 0.14 & 0.20
      &0.19 & 0.17 & 0.28
      \\ \hline
      148 242
      &0.10           & 6.21 $10^{-2}$ & 2.88 $10^{-2}$
      &9.45 $10^{-2}$ & 6.40 $10^{-2}$ & 4.33 $10^{-2}$
      \\ \hline
    \end{tabular}
    \\
    3D case
    \\
  \end{tabular}
\caption{ Convergence of the transmembrane potential 
$v$ and of the extracellular potential $u_e$ at three fixed times: $t=$ 1.2, 1.6 and 2.2. 
The reference solution for these chosen times are depicted in Figure \ref{fig:2d-prop}.
Above (resp. below) are reported the errors in the two (resp. three) dimensional case.
On each table line the three first figures after the number of
nodes correspond to the errors on $v$, whereas the three last ones correspond to $u_e$.
Errors are relative errors in $L^2(\Om)$ norm as defined in \eqref{eq:def-space-time-err}.}
 \label{tab:space-conv}
\end{table}

Because of the particular wavefront-like shape of the solution, this error
can be geometrically reinterpreted as follows. Consider at time $t$ the
sub-region of $\Om$ that is activated according to the reference
solution $u_{e}^r$ but not activated according to the coarse solution $u_{e}^c$. The numerator
in \eqref{eq:def-space-time-err} simply measures
the square root of the area of this sub-region.
Using the elliptic shape of activated regions, one gets that $e_{\Om,2}(u^r,u^c)$
measures the square root of the relative error on the wavefront propagation velocity
(more precisely the square root of the sum of the axial
and transverse wavefront propagation velocities relative errors).
This error, as in test case 1, is of prime physiological importance.
Numerical results for this test are displayed in Table \ref{tab:space-conv}.
Although convergence is well illustrated, no particular
asymptotic behaviour can be inferred from these results.

\subsubsection{Test 3; space and time convergence}
A numerical space and time convergence indicator $e_{Q,2}$ is
introduced here, aiming to reproduce an $L^2(Q)$ relative error
between a coarse and the reference solution  ($Q=(0,T)\times \Om$).
Convergence is measured using this indicator, and this third test
therefore is intended to numerically illustrate
the convergence result of Theorem \ref{th:convergence}.

\begin{table}
  \centering
  \begin{tabular}{ccc}
    \begin{tabular}{|lc|clc|}
      \hline
      ~~~\# nodes &~~~~~~&~~~~~& errors $e_{Q,2}$ &~~~~~~~~
      \\ \hline \hline
      489&&&
      0.481
      &\\ \hline
      1913&&&
      0.237
      &\\ \hline
      7569&&&
      6.469 $10^{-2}$
      &\\ \hline
      30113&&&
      1.746 $10^{-2}$
      &\\ \hline
      120129&&&
      4.167 $10^{-3}$
      &\\ \hline
      479 873 &&& reference
      &\\ \hline
    \end{tabular}
    &$\qquad$&
    \begin{tabular}{|lc|clc|}
      \hline
      ~~~\# nodes &~~~~~~&~~~~~& errors $e_{Q,2}$&~~~~~~~~
      \\ \hline \hline
      2559&&&
      0.673
      &\\ \hline
      19 500&&&
      0.219
      &\\ \hline
      148 242&&&
      4.920  $10^{-2}$
      &\\ \hline
      1 147 933 &&& reference
      &\\ \hline
    \end{tabular}
    \\
    \\
    2D case && 3D case
  \end{tabular}
  \caption{Space and time convergence for the transmembrane potential $v$.
    Errors are relative errors in the  $L^2(Q)$ norm
    as defined in \eqref{eq:def-space-time-err}
  }
  \label{tab:space-time-conv}
\end{table}

The transmembrane potentials $v^r$ and $v^c$ have been
recorded at the same times, namely $t_n=n \delta t$, for $\delta t=1/100$
unit of time,  $n=0, \dots, T/\delta t$ and $T=3$.
The corresponding numerical solutions are denoted $v^{r}_n$
and $v^{c}_n$. A relative error in the $L^2(Q)$ norm
between $v^r$ and $v^c$ is introduced as follows:
\begin{equation}
  \label{eq:def-space-time-err}
  e_{Q,2} (v^r,v^c) ^2:=\dfrac
  {
    \sum_{n=0}^N
    \int_\Om
    \vert \tilde{v_n}^r - \tilde{v_n}^{c\to r} \vert ^ 2
    dx
    \delta t
  }
  {
    \sum_{n=0}^N
    \int_\Om
    \vert \tilde{v_n}^r \vert ^ 2
    dx
    \delta t
  }
  ~,\quad N=T/\delta t.
\end{equation}

Numerical results  are given in Table \ref{tab:space-time-conv}:
convergence both in 2D and in 3D is observed.
The two dimensional case results indicate an order two
convergence relatively to the mesh size. As in the first test
case, such a conclusion is not possible in the
three dimensional case: a much finer reference
mesh (not affordable) would be needed.

\end{document}